\documentclass[10pt, oneside]{amsart} 	% use "amsart" instead of "article" for AMSLaTeX format
\usepackage{geometry}  		% See geometry.pdf to learn the layout options. There are lots.
\usepackage{graphicx}				% Use pdf, png, jpg, or eps� with pdflatex; use eps in DVI mode
\usepackage{amssymb}

\usepackage{hyperref}

\usepackage{amssymb}
\usepackage{amsthm}
\usepackage{graphicx}
\usepackage{mathtools}
\usepackage{amsmath,amsfonts}
\usepackage{epsfig}
\usepackage{comment,todonotes}
\usepackage{color}
 \usepackage{amsmath}
\usepackage{bbm}

\usepackage{bm}

\newcommand{\diag}{{\rm diag}}

\newcommand{\len}{{\rm L}}

\newcommand{\BBox}{\mathbf{Box}}%{\bf Box}
\newcommand{\bfY}{\mathbf{Y}}
\newcommand{\bfT}{\mathbf{T}}
\newcommand{\bft}{\mathbf{t}}
\newcommand{\bfX}{\mathbf{X}}
\newcommand{\bfZ}{\mathbf{Z}}
\newcommand{\bfV}{\mathbf{V}}
\newcommand{\bfM}{\mathbf{M}}
\newcommand{\bfL}{\mathbf{L}}
\newcommand{\bfD}{\mathbf{D}}
\newcommand{\bfy}{\mathbf{y}}
\newcommand{\bfp}{\mathbf{p}}
\newcommand{\bfq}{\mathbf{q}}
\newcommand{\bfh}{\mathbf{h}}
\newcommand{\bfx}{\mathbf{x}}
\newcommand{\bfe}{\mathbf{e}}
\newcommand{\bfz}{\mathbf{0}}
\newcommand{\bgamma}{\boldsymbol{\gamma}}

\newcommand{\Z}{{\mathbb Z}}
\newcommand{\bd}{\mathbf{d}}
\newcommand{\R}{{\mathbb R}}

\newcommand{\N}{{\mathbb N}}
\newcommand{\cM}{{\mathcal M}}

\newcommand{\dd}{{\rm{d}}}

\newcommand{\eucl}{{\rm eucl}}

\newcommand{\bX}{{\mathbf X}}

\newcommand{\eps}{\varepsilon}
\newcommand{\dmu}{d\mu}

\newcommand{\bp}{{\bm p}}
\newcommand{\cF}{\mathcal{F}}

\DeclareMathOperator{\diam}{diam}

\DeclareMathOperator{\osc}{osc}

\DeclareMathOperator{\Lip}{Lip}

\def\vint_#1{\mathchoice
  {\mathop{\vrule width 6pt height 3 pt depth -2.5pt
   \kern -8pt \intop}\nolimits_{#1}}%
  {\mathop{\vrule width 5pt height 3 pt depth -2.6pt
   \kern -6pt \intop}\nolimits_{#1}}%
  {\mathop{\vrule width 5pt height 3 pt depth -2.6pt
   \kern -6pt \intop}\nolimits_{#1}}%
  {\mathop{\vrule width 5pt height 3 pt depth -2.6pt
   \kern -6pt \intop}\nolimits_{#1}}}

\theoremstyle{plain}

\newtheorem{theorem}[equation]{Theorem}
\newtheorem{lemma}[equation]{Lemma}
\newtheorem{proposition}[equation]{Proposition}
\newtheorem{corollary}[equation]{Corollary}

 \numberwithin{equation}{section}

\theoremstyle{definition}
\newtheorem{definition}[equation]{Definition}

\theoremstyle{remark}
\newtheorem{remark}[equation]{Remark}

\title[Non-smooth sub-Riemannian PI manifolds]{Non-smooth sub-Riemannian manifolds with non-integer dimension and a Poincar\'e inequality}

\author{Sylvester Eriksson-Bique}
\email{sylvester.d.eriksson-bique@jyu.fi}
\address{Department of Mathematics and Statistics
P.O. Box 35
FI-40014 University of Jyväskylä}

\author{Enrico Le Donne}
\email{ledonne@msri.org}
\address{Université de Fribourg
Département de mathématiques
Ch. du musée 23
CH-1700 Fribourg
Switzerland}

\thanks{S.E.-B. was supported by the Research Council of Finland grant 354241. E.L.D. was supported by the Swiss National Science Foundation
	(grant 200021-204501 `\emph{Regularity of sub-Riemannian geodesics and applications}').
The work was started and concluded at the University of Fribourg, and the authors acknowledge the kind environment, thanking Stefan Wenger for helpful discussions. 
}

\date{\today}
\begin{document}

\maketitle

\begin{abstract}
For every integer $d\geq 3$ and real number $Q\in \left(d,\frac{d(d-1)+2}{2}\right]$, we construct a $Q$-Ahlfors regular metric measure manifold $(M,\bd,\mu)$ of topological dimension $d$ and analytic dimension $2$ that satisfies a Poincar\'e inequality.  For every $Q>\frac{d(d-1)+2}{2}$, the construction yields a metric measure manifold $(M,\bd,\mu)$ that is doubling, satisfies a Poincar\'e inequality, and for which $M$ has $\mu$-essential Hausdorff dimension $Q$.  The manifold is $\R^d$ equipped with Lebesgue measure and a sub-Riemannian distance given by vector fields that are locally H\"older continuous but not Lipschitz continuous. These metrics give an uncountable family of new models that are genuinely distinct from the only known families of metric manifolds satisfying a Poincar\'e inequality: weighted and deformed Carnot groups and smooth sub-Riemannian spaces. The construction is modeled on a singular variant of a filiform Carnot group. 

%All the constructions yield linearly locally contractible manifolds whose blow-ups are also spaces that arise from a similar construction. In particular, the examples are self-similar as a class: if one considers the class of examples with controlled parameters, then the class is compact in the Gromov-Hausdorff topology. These constructions are examples of PI manifolds and yield a negative answer to a series of open questions about possible analytic structures on metric manifolds: \textbf{1)} Are Carnot groups the only possible infinitesimal model for PI manifolds? \textbf{2)} Can PI manifolds have non-integer essential Hausdorff dimension? \textbf{3)} Are all PI manifolds rectifiable in terms of (smooth) sub-Riemannian manifolds? Further, these examples show that once $d\geq 3$, there is an uncountable family of quasisymmetrically inequivalent metrics and measures on $\R^d$ supporting a Poincar\'e inequality, all with different Hausdorff dimensions. This implies that the Heinonen-Semmes question about the existence of metric spaces with non-integer dimensions that satisfy a Poincar\'e inequality can also be solved in the class of metric manifolds. 
\end{abstract}

\tableofcontents

\section{Introduction}
\subsection{Overview}
Let $M$ be a topological manifold equipped with a distance $\bd$ and a measure $\mu$. We call $(M,\bd,\mu)$ a {\em metric measure manifold} if $\bd$ induces the manifold topology on $M$ and if $\mu$ is a Radon measure that is positive and finite for each ball in $M$ of positive radius. Motivated by an extensive study on the quasiconformal geometry and analysis on metric spaces, we are interested in those examples that satisfy some Poincar\'e inequality; see \cite{BonkICM, KleinerICM, HeNonSmooth, Bourdon2002}. When such spaces are additionally complete and satisfy a doubling bound, they have a Lipschitz differentiability structure by the seminal result of Cheeger \cite{Ch}; see Section~\ref{sec: analytic} for a definition. Furthermore, in this setting, there is often rigidity for quasiconformal maps \cite{HK}, which is one of the primary motivations for their applications in geometric group theory; see \cite{Bourdon2002, KleinerICM}. In many applications of this theory, one can obtain topological information, which in turn hopefully would restrict the possible differential structures on these spaces. We will see, however, that when the topological dimension $d$ is at least $3$, the topological space $\R^d$ allows for a wide variety of Ahlfors regular metrics supporting a Poincar\'e inequality. We will also see that the family of spaces that we present in this paper contains uncountably many quasisymmetrically distinct spaces, all with different Hausdorff dimensions. The richness of these structures suggests that, when $d\geq 3$, the classification of metrics and measures satisfying a Poincar\'e inequality is a very difficult task -- or perhaps impossible. The case $d=2$ is very intriguing, but our methods do not reach that case; see \cite{DavidKleiner, KleinerSchioppa, davidseb,murugananalytic} for discussion on related problems.

\subsection{Main results}

A measure $\mu$ on a metric space $M$ is said to be {\em doubling}, if there exists a constant $D>0$ such that for all $x\in M$ and $r\in (0,\infty)$ we have $0<\mu(B(x,2r))\leq D\mu(B(x,r))$. As introduced in \cite{HK}, a metric measure space $(M,\bd,\mu)$ is said to satisfy a {\em Poincar\'e inequality} if for some $C\geq 1, \lambda \geq 1$ we have for all Lipschitz functions $u:M\to \R$ and all balls $B(x,r)\subset M$ that
\begin{equation}\label{eq:PIineq}
\frac{1}{\mu(B(x,r))} \int_{B(x,r)} \left|u- \frac{1}{\mu(B(x,r))} \int_{B(x,r)} u \dd\mu \right|\dd\mu \leq C r \frac{1}{\mu(B(x,\lambda r))} \int_{B(x,\lambda r)} \Lip[u] \dd\mu,
\end{equation}
where $\Lip[u](x):=\limsup_{y\to x} \frac{|u(x)-u(y)|}{\bd(x,y)}$ is the pointwise Lipschitz constant. This is one of many equivalent forms of a Poincar\'e inequality; see \cite{Keith,HajKosSobolevMeets,HajlaszKoskela}. There is also a $(1,p)$-variant of the Poincar\'e inequality, for $p\in (1,\infty)$, and the inequality above is sometimes called a $(1,1)$-Poincar\'e inequality.  A metric measure space $(M,\bd,\mu)$ equipped with a doubling measure and satisfying a Poincar\'e inequality is called a {\em PI space}. If, in addition, a PI space has the topology of a manifold, we shall call it a {\em PI manifold}.

The measure $\mu$ is said to be $Q$-{\em Ahlfors regular} with constant $C>0$, 
if $C^{-1}r^Q \leq \mu(B(x,r))\leq Cr^Q$ for all $r\in (0,2\diam(M))$ and $x\in M$. In this case, the measure $\mu$ is, up to bounded density, determined entirely by the metric and can be taken to be the Hausdorff $Q$-measure.
This property is particularly interesting for quasiconformal analysis, since for $Q>1$ complete $Q$-Ahlfors regular spaces satisfying a Poincar\'e inequality are examples of Loewner spaces; see \cite{HK} for more details and discussion.
Typical examples of (locally) Ahlfors regular PI manifolds are equiregular sub-Riemannian manifolds and their tangents: Carnot groups equipped with Carnot-Carath\'eodory metrics; see \cite{LeDonneGTM}.

In \cite[Question 1.3]{KleinerSchioppa} the authors asked the following. \vspace{.1cm}

\noindent \textbf{Question} \emph{Pick $n \geq 2$. Suppose $(X,\mu)$ is a PI space homeomorphic to $\R^n$, such that any pointed Gromov-Hausdorﬀ limit of any
sequence of pointed rescalings of $X$ is also homeomorphic to $\R^n$. What
are the possibilities for the analytic dimension of $(X,\mu)$? The Hausdorﬀ dimension? For instance, is there such a PI space homeomorphic
to $\R^3$ (or even $\R^2$) with analytic dimension 1? Or one which is Ahlfors
$Q$-regular for large $Q$?}\vspace{.1cm}

A similar question was posed in  \cite[Question 1.1]{DavidKleiner}.\vspace{.1cm}

\noindent \textbf{Question:}
    \emph{Let $(X,d,\mu)$ be a PI space and assume $X$ is a topological $n$-manifold satisfying LLC or all blow-ups are $n$-manifolds. What can be said about the dimension of the diﬀerentiable structure, the Hausdorﬀ dimension, or the structure of blow-ups of $(X,d,\mu)$?}\vspace{.1cm}

 In particular, due to a lack of examples, one might ask if there are PI manifolds that are not locally given by a sub-Riemannian structure whose blow-up at a.e. point is a Carnot group, or if they can have non-integer essential Hausdorff dimensions. Our main theorem constructs families of PI-manifolds which do not have Carnot blow-ups and have non-integer essential Hausdorff dimension. Moreover, our spaces are \emph{linearly locally contractible} (LLC) in the sense that there exists a constant $C>0$ such that every ball $B(x,r)$ is contractible within $B(x, C r)$, and their blow-ups are still manifolds.

\begin{theorem}\label{thm: mainthm} Let $d\in \N$ with $d\geq 3$ and $Q\in (d,\frac{d(d-1)+2}{2}]$. There exists a nonempty collection $\mathcal{F}$ of pointed metric measure manifolds so that each $M\in \mathcal{F}$ is isomorphic to some $(\R^d,\bd,\mu, \bfz)$ satisfying:
\begin{enumerate} 
\item $(\R^d,\bd)$ is a  complete geodesic metric space, that is linearly locally contractible and whose distance induces the standard topology of $\R^d$,
\item $\mu$ is equal to a multiple of the Lebesgue measure on $\R^d$,
\item $\mu$ is $Q$-Ahlfors regular with respect to $\bd$, 
\item $M$ has analytic dimension $2$, and
\item $M$ satisfies a Poincar\'e inequality,
\end{enumerate}
and $\cF$ has the properties:
\begin{enumerate}
 \item[(6)] $\cF$ is invariant under scalings and changes of basepoints, and
 \item[(7)] $\cF$ is compact with respect to pointed measured Gromov-Hausdorff convergence.
\end{enumerate}
\end{theorem}
 Our result thus partially answers the second sub-question of \cite[Question 1.3]{KleinerSchioppa}, and the second and third sub-questions of \cite[Question 1.1]{DavidKleiner}, by showing that (essential) Hausdorff dimensions can be non-integers and blow-ups can be different from Carnot groups. 

By \cite[Theorem 3.4]{Tyson} the Hausdorff dimension $Q$ is a quasisymmetric invariant of a complete $Q$-Ahlfors regular space satisfying a Poincar\'e inequality for $Q>1$. It follows that for distinct values of $Q$ the spaces constructed are quasisymmetrically distinct. We note that when $Q=d$, all complete $Q$-Ahlfors regular linearly locally contractible manifolds are (locally) uniformly rectifiable by \cite[Theorem 1.13]{davidbilip}, satisfy a Poincar\'e inequality by \cite{S}, and thus necessarily have analytic dimension $d$ by \cite[Theorem 1.2]{batelirect}.

We express the main theorem in terms of a collection of spaces to study the Gromov-Hausdorff tangents of each element $M$ in $\mathcal{F}$. Let $r>0$. The {\em rescaling by factor} $r$ of a pointed metric measure space $(M,\bd,\mu,x)$ is the space $(M,\bd/r,c\mu,x)$, where $c:=\mu(B(x,r))^{-1}$. 
A {\em weak tangent} of $(M,\bd,\mu)$ is a pointed measured Gromov-Hausdorff limit of a sequence of dilations of factor $r_i$ of $(M,\bd/r_i,\mu/\mu(B(x_i, r_i)), x_i)$ with basepoints $x_i\in M$ for some $r_i\searrow 0$; see Section~\ref{sec: GH} for details. In general, such tangents may have a different topology from $M$. An immediate corollary of the previous theorem is that $\mathcal{F}$ is closed with respect to the operation of taking pointed measured Gromov-Hausdorff tangents. 
%{\color{red}Since the class of spaces consists of $Q$-Ahlfors regular metric measure spaces, one could remove the measure entirely from the discussion and just consider the class of metric spaces compact with respect to pointed Gromov-Hausdorff limits. Indeed, the measure $\mu$ could always be chosen as the Hausdorff $Q$-measure at the expense of replacing the second condition in Theorem~\ref{thm: mainthm} with comparability with the Lebesgue measure. 
%}
%h as the same properties as the original space.

 \begin{corollary}\label{cor:tangents} 
 Every family $\cF $ with the properties of Theorem~\ref{thm: mainthm} is such that
 for every $M\in \cF$, every weak tangent of $M$ is in $\cF$.
 \end{corollary}
In particular, the Gromov-Hausdorff tangents of the spaces in question are not simply Carnot groups.  The construction that we present in this article is quite general and depends on several parameters, including the value $Q$.
Our results are stronger for $Q$ smaller than the specific threshold.
Namely, we have Theorem~\ref{thm: mainthm} and Corollary~\ref{cor:tangents} for $Q\in (d,\frac{d(d-1)+2}{2}]$. 
In the case of $Q>\frac{d(d-1)+2}{2}$, our construction cannot yield an Ahlfors-regular metric measure manifold, but it still provides an interesting class of doubling metric measure manifolds. These examples are dimensionally complex, and the volume $\mu(B(x,r))$ is bounded from above and below by $r^{Q\pm\epsilon}$ only at small scales and almost every point. At exceptional points, different asymptotics hold for volume. For this reason, we need to use the notion of essential Hausdorff dimension,
i.e., the infimum of the Hausdorff dimension of full-measure sets.
%if every $\mu$-positive measure set $K\subset X$ has Hausdorff dimension at least $Q$. 

\begin{theorem}\label{thm: mainthm2} Let $d\in \N$ with $d\geq 3$ and $Q>\frac{d(d-1)+2}{2}$. There exists a PI manifold $(\R^d,\bd,\mu)$ that satisfies:
\begin{enumerate} 
%\item $(\R^d,\bd)$ is linearly locally contractible,
\item $(\R^d,\bd)$ is a linearly locally contractible complete geodesic metric space, whose distance induces the standard topology of $\R^d$,
\item every weak tangent of $(\R^d, \bd)$ is homeomorphic to $\R^d$,
\item $\mu$ is equal to a multiple of the Lebesgue measure on $\R^d$,
\item $\mu$ is doubling, 
\item the space satisfies a Poincar\'e inequality,
\item the space has analytic dimension $2$, and
\item the space has $\mu$-essential Hausdorff dimension $Q$ and every set of positive measure has Hausdorff dimension at least $Q$.
\end{enumerate}
\end{theorem}

These examples should be contrasted with the Grushin plane and its almost Riemannian generalizations, where the Hausdorff dimension is large and possibly irrational, but only due to the presence of a closed singular subset, which is negligible from the measure viewpoint. In the previous examples, the $Q$-dimensional behavior is present throughout the space and is visible in all positive measure sets. 

There is a small difference between Theorems~\ref{thm: mainthm} and~\ref{thm: mainthm2}. In the latter case, we do not state the result in terms of a class of spaces. The reason for this is mostly technical and is related to the fact that the $\mu$-essential Hausdorff dimension is not a very stable notion. However, we do define a class of spaces $\mathcal{M}^{\rm doubl}_{D,\epsilon,\delta}$ that is invariant under scalings and changes of basepoints, whose elements are all PI-manifolds that contain our examples, and this collection is closed with respect to taking weak tangents. However, the essential Hausdorff dimensions here could change under blow-up. At the end of the introduction, we will give a technical definition of a class that would be also invariant under scalings and changes of basepoints, with constant essential Hausdorff dimension $Q$ and which contains our examples in Theorem~\ref{thm: mainthm2}.

\subsection{Broader context}
Since its introduction, it has been unclear how restrictive assuming a Poincar\'e inequality is. For example, in \cite[Question 19]{HS}, Heinonen and Semmes asked whether, for metric measure spaces that satisfy a Poincar\'e inequality, one must have $Q\in \N$; see \cite{heinsemsurv} for a recent survey on these problems. Soon thereafter, Bourdon-Pajot \cite{BourPajPI} and Laakso \cite{La00} gave examples for which this was not the case. These examples are homeomorphic to the Menger sponge, where the homeomorphism follows from Anderson's theorem \cite{Anderson}, and very much not topological manifolds; see also \cite{CS} for planar examples homeomorphic to the Sierpi\'nski carpet. Before the findings of this paper, the only examples of metric measure manifolds $(M,\bd,\mu)$ with a Poincar\'e inequality were the following:
\begin{enumerate}
\item Smooth manifolds equipped with smooth sub-Finsler metrics equipped with Riemannian volume forms. These examples include all Carnot groups; see \cite{J, L}.
\item Smooth sub-Riemannian manifolds equipped with Muckenhoupt-weighted Riemannian volume forms. Examples are \emph{$1$-admissible weights}; see \cite{HKM}.
\item Strong-$A_\infty$ weighted deformations of Riemannian manifolds as constructed by Semmes; see \cite{S, DSinf,Semmes}.
%LLC and Ahlfors $n$-regular manifolds as introduced by Semmes: 
% \cite{S} A manifold $(M,d_\omega,\mu)$
%constructed by deforming the metric by a strong-$A_\infty$ weight and taking the Hausdorff $d$-measure of the associated metric space. 
\item Strong-$A_\infty$-deformations of Carnot groups; see \cite{DSinf, HK}.
%This follows from the techniques in \cite{DSinf} and \cite{HK}.
\end{enumerate}
In fact, all of the mentioned examples are rectifiable in terms of subsets of smooth sub-Riemannian manifolds. These examples can further be easily summarized in terms of their 
infinitesimal models by looking at their blow-ups: in each of these cases, at almost every point, the spaces possess Gromov-Hausdorff tangents that are Carnot groups equipped with their Haar measures. Furthermore, when $d=2$, all of these examples have Euclidean space $\R^2$ as their blow-up, and all of the known examples collapse to only one family of examples: rectifiable surfaces equipped with a weighted Hausdorff measure. When $d= 3$, the previously known examples have only two possible blow-ups: the Heisenberg group and the Euclidean space. In contrast, Theorems~\ref{thm: mainthm} and~\ref{thm: mainthm2} give a large family of PI manifolds and infinitesimal models thereof.

 This discrepancy between the known examples and the structural properties of PI spaces motivated the above questions to ask if there are other models for infinitesimal geometries of PI spaces and what restrictions the manifold topology places on each PI space; see \cite{KleinerSchioppa}. This question, along with the Heinonen-Semmes question in \cite{HS}, motivates our results here. 

\subsection{Construction}

%First, we stress that the situation in which $Q=d$ in Theorem~\ref{thm: mainthm} would correspond to the Euclidean space -- which is included only for the sake of a simple statement. 
Each metric $\bd$ in the theorems above is some type of sub-Riemannian metric associated with continuous vector fields. 
On the one hand, one can consider the broad setting of continuous Carnot-Carathéodory structures. We shall provide some general results on their length structures; see Section~\ref{sec Rectifiable curves in CC structures}. 
On the other hand, we shall focus on the following class of more specific examples, which can be called ``filiform-type sub-Riemannian metrics''.  
Fix $d\in \N$ with $d\geq 3$. Let $f:\R\to \R$ be some continuous function and define two continuous vector fields on $\R^d$ with coordinates $x_1, \ldots, x_d$:
\begin{equation}\label{eq:vectdef}
\bfX := \partial_{x_1}, \ \ \ \ \ \ \ \ \ \ \bfY_f := \sum_{i=2}^d (f(x_1))^{i-2} \partial_{x_i}.
\end{equation}
These vector fields are uniquely integrable. For each $x,y\in \R^d$, we can define
\begin{multline}\label{eq:metricdef}
\bd(x,y):=\bd_f(x,y):=\inf\bigg\{ \int_0^1 \sqrt{|a(t)|^2 + |b(t)|^2} \dd t \,\bigg|\, 
 \bgamma:[0,1]\to \R^d \text{ absolutely continuous, } \\
 \bgamma(0)=x,\bgamma(1)=y, \bgamma'(t)=a(t)X_{\gamma(t)}+b(t)Y_{\gamma(t)}, \text{ for a.e. } t\in [0,1]\bigg\},
\end{multline}
where we used the vector fields \eqref{eq:vectdef}. To get an initial idea of the construction, the reader may wish to concentrate on $d=3$, in which case the vector fields are $\bfX:=\partial_{x}, \bfY_f = \partial_{y}+f(x)\partial_z$.

Without some assumption on $f$, this metric may be infinite. However, it is not too difficult to see that $\bd(x,y)<\infty$ for all $x,y\in \R^d$ if and only if $f$ is not constant; cf. Lemma~\ref{lem: characterization}. Further, a similar argument shows that if $f$ has positive oscillation, i.e.,
\begin{equation}\label{oscillation}
    \osc(f,I):=\sup_{t,s\in I}|f(t)-f(s)|>0,\qquad \text{ for every interval } I\subset \R,
\end{equation} then $\bd_f$ induces the Euclidean topology; see Lemma~\ref{lem: characterization}. In this work, intervals are considered to contain more than one point.

If $f(x)=x$, then the vector fields $X, Y$ generate a \emph{filiform Lie algebra}, and the metric is nothing but the usual Carnot-Carathéodory metric on $\R^d$ giving the exponential coordinates of the second type for a filiform Lie group; see \cite[Chapter~9]{LeDonneGTM}. In this situation, we have dimension $Q=\frac{d(d-1)+2}{2}$ in Theorem~\ref{thm: mainthm}. If $f(x)=x^k$ with $k\geq 2$, then the construction yields a sub-Riemannian manifold that is a variant of the Grushin plane; see e.g. \cite{MM, B}. Such examples are not equiregular in the sense of \cite{B}.

The key to our approach is to allow for $f$ to be a much more general continuous, not necessarily differentiable, function. We say that $f:\R\to \R$ has {\em doubling oscillation} if there exists a $D\geq 1$ 
such that 
%and for every $x\in \R$ and $r>0$ we have
\begin{equation}\label{eq:doublingf}
0<\osc(f,[x-2r,x+2r])\leq D\osc(f,[x-r,x+r]), \qquad \forall x\in \R,\forall r>0.
\end{equation}
The constant $D$ is called a {\em doubling constant} of $f$. We further say that $f$ has {\em reverse doubling oscillation} if there exist $\epsilon,\delta\in (0,1)$ such that for every pair of intervals of positive length $I\subset J\subset \R$ we have
\begin{equation}\label{eq: def reverse doubling}
\frac{|I|}{|J|} \leq \delta \Longrightarrow \frac{\osc(f,I)}{\osc(f,J)} \leq \epsilon.
\end{equation}
We call $(\epsilon,\delta)$ the {\em reverse doubling constants}.
Equivalently, a non-constant function $f$ has doubling oscillation and reverse doubling oscillation if and only if there exist $\alpha>0 $, $\beta>0$, and $C\geq 1$ such that  
\begin{equation}
\label{eq: def doubling and reverse doubling}
\frac1C\left(\frac{|I|}{|J|}\right)^\alpha \leq\frac{\osc(f,I)}{\osc(f,J)} \leq C\left(\frac{|I|}{|J|}\right)^\beta, \qquad \forall I\subset J\subset \R.
\end{equation}

%Iterating this, we get that $f$ is $s$-homogeneous for some $s>0$: there exists a constant $C>0$ for which $\osc(f,[x-\delta r,x+\delta r])\geq C \delta^s \osc(f,[x-r,x+r])$ for all $x\in \R, r>0$ and $\delta\in (0,1)$. 
For every  function $f$ that has doubling and reverse doubling oscillation, we show that $(\R^d,\bd_f,\lambda)$ is a PI manifold, where $\lambda$ is the Lebesgue measure.
\begin{theorem}\label{thm: oscillationversion}
Given $d\in \N$ with $d\geq 3$ and a continuous function $f:\R\to \R$, let 
$(\R^d,\bd_f,\lambda)$ be the metric measure space, where $\bd_f$ is defined in \eqref{eq:metricdef} via the vector fields \eqref{eq:vectdef}, and $\lambda$ is the Lebesgue measure.
If $f $ has doubling oscillation, then 
$(\R^d,\bd_f,\lambda)$ is complete, geodesic, homeomorphic to Euclidean space, LLC and doubling, and there exists a constant $C>1$ such that 
\[
C^{-1}r^d \rho^{\frac{(d-2)(d-1)}{2}}\leq \lambda(B(\bfp,r))\leq C r^d \rho^{\frac{(d-2)(d-1)}{2}}, \qquad \forall\bfp\in \R^d \ \  \forall r>0,
\]
with $\rho=\osc(f,[p_1-r,p_1+r])$, where $p_1$ is the first component of $\bfp$.
If moreover $f$ has reverse doubling oscillation, then $(\R^d,\bd_f,\lambda)$ also satisfies a Poincar\'e inequality.
%and $(\R^d,\bd_f,\lambda)$ satisfies a Poincar\'e inequality. 
%where $\bd_f$ is defined in \eqref{eq:metricdef} via the vector fields \eqref{eq:vectdef}, and $\lambda$ is the Lebesgue measure.
The constants are controlled by the dimension $d$ and the doubling and reverse doubling constants of $f$.
\end{theorem}

From this last theorem, we will deduce Theorems~\ref{thm: mainthm} and~\ref{thm: mainthm2} by appropriately finding some  functions $f$ of doubling and reverse doubling oscillation. For $Q\in \left(d, \frac{d(d-1)+2}{2}\right)$, one obtains the examples by choosing $f$ to be a Weierstrass function for which 
\[
\osc(f,[x-r,x+r])\sim r^\alpha\quad \text{ with } \alpha:=\frac{2(Q-d)}{(d-2)(d-1)}\in (0,1).
\]
In fact, this is where the restriction on $Q$ comes from. It follows from differentiation that there are no $\alpha$-regular functions with $\alpha>1$. For $Q$ larger than $\frac{d(d-1)+2}{2}$, one must adjust the construction of $f$. One can choose $f$ to be a continuous increasing function whose distributional derivative is a singular doubling measure with appropriate dimensional behavior, and such a function just falls barely short of being of $\alpha$-regular oscillation; see Lemma~\ref{lem: oscillationfunctions} and Remark~\ref{rmk:oscillation}. We will construct the functions with the desired behaviors in Section~\ref{sec: doublingfunctions}. In both cases, the functions $f$ possess some self-similarity, which may be of independent interest.

Besides the construction of functions having doubling oscillation, there are two main technical steps in the proof of Theorem~\ref{thm: mainthm}. First, one needs to show a Ball-Box Theorem; see Theorem~\ref{thm: ballbox} below. This is not quite trivial since the vector fields in question are not smooth. Indeed, one cannot use the usual proof of Chow's theorem that uses iterated Lie brackets and exponential maps. Instead, we devise another proof that uses the special structure of the vector fields, and we observe that the controllability of the vector fields can be expressed in terms of the invertibility of a Vandermonde matrix; see Lemma~\ref{lem: distest}.

The paths used in the proof of the Ball-Box Theorem are explicit and can be used to ``controllably'' connect pairs of points. The Poincar\'e inequality is then implied by a result of Lanconelli and Morbidelli \cite[Theorem 2.1]{L}. Even here, we need to make a small adjustment to the argument since our maps are only continuous and not differentiable. However, the maps have a particularly simple structure that allows us to compute their measure-theoretic Jacobian and show that they are, in fact, volume-preserving. 

\subsection{Class of spaces}
The construction in the previous subsection yields the desired exotic PI spaces. However, the properties of their tangents, their rough self-similarity, and the moduli space of the possible constructions are better understood by studying the pointed classes of spaces coming out of the construction. In particular, we wish to preclude the possibility that our examples have a particular special structure only at macroscopic scales and that such effects would disappear after taking Gromov-Hausdorff limits. This is particularly crucial since we wish to argue that the class of examples can be seen as particular types of infinitesimal models, cf. discussions in \cite{KleinerSchioppa}.

The classes of spaces are parametrized by certain classes of functions with uniform estimates.
For $D>1, \epsilon,\delta\in (0,1)$, we denote by $\cF^{\rm doubl}_{D,\epsilon,\delta}$ the class of non-constant functions that have doubling oscillation with constant $D$ and reverse doubling oscillation with constants $(\epsilon,\delta)$. This class of functions includes the ones used in Theorem~\ref{thm: mainthm2}.
For $C>1$ and $\alpha\in (0,1]$, another class is 
\[
\cF^{\rm reg}_{C,\alpha}:=\left\{ f:\R\to \R: C^{-1}\left(\frac{|I|}{|J|}\right)^\alpha \leq \frac{\osc(f,I)}{\osc(f,J)}\leq C\left(\frac{|I|}{|J|}\right)^\alpha  \text{ for intervals } I\subset J \subset \R \right\}.
\]
In Theorem~\ref{thm: mainthm}, the class of spaces is $\mathcal{M}^{\rm reg}_{C,\alpha} :=\{(\R^d, \bd_f, c\lambda, \bfp): f\in \cF^{\rm reg}_{C,\alpha}, c=\lambda(B(\bfp,1))^{-1}\}$. In Lemma~\ref{lem: convergence}, we establish a stability property: if $\cF$ is any of the classes above, and if $f_n\in \cF$ converge locally uniformly to some non-constant $f$, then $f\in \cF$ and the associated spaces converge in the pointed measured Gromov-Hausdorff sense. Moreover, we show that after a normalization these classes of functions are compact, and this yields the compactness of the class in Theorem~\ref{thm: mainthm}. 

Our results also provide a way of approximating some of the non-smooth sub-Riemannian manifolds that we construct by using smooth sub-Riemannian metrics; see Lemma~\ref{lem: convergence}. This is obtained by choosing a sequence of smooth functions $f_n$ with doubling and reverse doubling oscillation approximating a given function $f$.  For example, such approximations can be produced by truncating the series defining the Weierstrass functions that we use in Section~\ref{sec: doublingfunctions}. This claim is not needed in the paper and we do not pursue it further.

\begin{remark} The class $\cF^{\rm reg}_{C,\alpha}$ is simple to define and leads to a family of $Q$-Ahlfors regular spaces with $Q$ given above. Once $Q>\frac{d(d-1)+2}{2}$, it is harder to define a class of functions, that is compact, such that each function in the collection corresponds to a space that has $\mu$-essential Hausdorff dimension $Q$, and every positive measure subset has Hausdorff dimension at least $Q$. This is mostly a technical challenge, and a possible condition is the following. Let $\alpha =\frac{2(Q-d)}{(d-2)(d-1)}$. Take any continuous  increasing function $\phi:[0,1]\to[0,\infty)$ with $\lim_{t\to 0}\phi(t)= 0$,  $\epsilon,\delta \in (0,1)$ and
\begin{align*}
&\mathcal{F}_{\phi,Q,D}=\Big\{f\in \mathcal{F}^{\rm doubl}_{D,\epsilon,\delta} : \forall I\subset \R, r\in (0,1)\\
& \ \ \lambda\left\{x\in I : \forall s\in (0,r|I|),  \left(\frac{2s}{|I|}\right)^{\alpha+\phi(r)} \leq \frac{\osc(f,[x-s,x+s])}{\osc(f,I)} \leq \left(\frac{2s}{|I|}\right)^{\alpha-\phi(r)}\right\} \geq (1-\phi(r)) |I| \Big\}.
\end{align*}
See Section~\ref{sec: GH} for the definition of $\mathcal{F}^{\rm doubl}_{D}$. It follows fairly directly from Lemma~\ref{lem: billinglsey} and the proof of Theorem~\ref{thm: mainthm2} that any $f\in \mathcal{F}_{\phi,Q,D}$ yields a space $(\R^d,\bd_f, \lambda, \bfz)$ satisfying the conclusions in Theorem~\ref{thm: mainthm2}. Moreover, this class of functions is invariant under translations and scalings, and thus the proof of Lemma~\ref{lem: convergence} shows that the associated collection of spaces is compact with respect to the pointed measured Gromov-Hausdorff convergence and invariant under scalings and changes of basepoints. Thus, this class of spaces is also closed under taking weak tangents. The class is non-empty for some $\phi$ by the constructions in Section~\ref{sec: doublingfunctions}.
\end{remark}

\subsection{Outline}

In Section~\ref{sec: preliminaries}, we introduce some terminology that was not covered in the introduction, study general continuous sub-Riemannian structures, recall properties of the Vandermonde matrix that are used extensively in our arguments, and give the crucial estimates for the distance $\bd$ that we construct. We establish the appropriate version of the Ball-Box Theorem in Subsection~\ref{subsec: ballbox}. We also present a version of the known Billingsley's Lemma on Hausdorff dimension that is needed to establish the claims about $\mu$-essential Hausdorff dimension.

In Section~\ref{sec: doublingfunctions}, we construct the functions having doubling oscillation that are used in the proofs of the main theorems. In Section~\ref{sec: GH}, we briefly discuss the Gromov-Hausdorff limits of our examples and show the convergence result in Lemma~\ref{lem: convergence}. These are then applied in Section~\ref{sec: PI}, where we prove the Poincar\'e inequality, which yields Theorem~\ref{thm: oscillationversion}. We prove the main theorems in Section~\ref{sec main proofs}. In the last section, Section~\ref{sec: analytic}, we briefly study the analytic dimension of continuous sub-Riemannian metrics. Some of these results may be of independent interest, and we formulate them in a general form for this reason.

\section{Preliminaries}\label{sec: preliminaries}

We follow \cite[Definition~1.1.1]{bbi} and consider metric spaces where the metric is possibly infinite-valued.
%\subsection{Notation}
An open ball in a metric space $(M, \bd)$ is a set of the form $B(\bfx,r):=\{\bfy\in M: \bd(\bfx,\bfy)<r\}$, with $\bfx\in M$ and $r>0$. 
For $B=B(\bfx,r)$ its deflation/inflation by $C>0$ is denoted by $CB=CB(\bfx,r)=B(\bfx,Cr)$. Every ball is taken with an associated center and radius, and different centers and radii might yield the same set. 
If we want to stress that balls are considered with respect to some distance $\bd$, clear from context, we add a subscript, e.g., $B_\bd$, to indicate which metric is used. This is mainly necessary since in $\R^d$, on occasion, we will consider  the Euclidean metric $\bd_{\eucl}$ as well as $\bd_f $, for some function $f $ as in \eqref{eq:metricdef}. By an interval $I$ we mean a sub-interval $[a,b] \subset \R$ with $a<b$ with $a,b\in \R$. Note that all intervals are closed, bounded and have positive length. For vectors $v\in \R^d$ the standard Euclidean norm is denoted $|v|$.

%{\color{red} {\bf Enrico would remove:} In the metric spaces that we consider, closed balls are compact, and we will be considering Radon measures $\mu$ on them that are positive and finite on balls: $\mu(B(\bfx,r))\in (0,\infty)$, for all $r>0$. }

%Usually, the metric space $(M, \bd)$ that will be used is $M:=\R^d$, for some $d\in\N$ with $d\geq 3$ and $\bd:=\bd_f $, for some function $f $ as in \eqref{eq:metricdef}.  

%Quite often, we will denote the Lebesgue measure on Euclidean spaces by $\lambda$. Where it is important to use a constant multiple of Lebesgue measure, we instead use $\mu$.

\subsection{Rectifiable curves in CC structures}\label{sec Rectifiable curves in CC structures}

In this section, we discuss a very general setting where one constructs a metric space starting from a manifold, a continuous subbundle of the tangent bundle, and a continuously varying norm. The non-standard assumption is the use of continuous bundles; the rest is the classic theory of Carnot-Carath\'eodory spaces; see \cite{LeDonneGTM}.

The next results aim to clarify that the length structure with respect to a Carnot-Carath\'eodory distance coincides with the length structure with respect to the Carnot-Carath\'eodory structure, also when the structure is continuous.
%in sub-Riemannian geometry, also in the presence of a continuous subbundle, the curves of finite length are horizontal, when reparametrized by arc length. 
Namely, we now consider a {\em continuous Carnot-Carath\'eodory structure} $(M,\Delta,\|\cdot\|) $, i.e., the set $M$ is $\R^d$, for some $d\in \N$, or more generally a manifold; 
 $\Delta\subseteq TM$ is a continuous subbundle of the tangent bundle, e.g., there are $X_1,\ldots, X_k$ pointwise linearly independent continuous vector fields on $M$ such that 
\begin{equation}\label{def Delta vectors}
\Delta_p:= \Delta\cap T_pM={\rm span} \{X_1(p),\ldots, X_k(p)\},\qquad \forall p\in M; 
\end{equation}
and $\|\cdot\|$ is a continuously varying norm on $TM$, i.e., a continuous function $TM\to \R $ that restricts to each tangent space $T_pM$ as a norm.
In this setting, 
we say that the vector fields  $X_1,\ldots, X_k$  form a {\em horizontal frame} for $\Delta$. Moreover,
we say that a curve $\bgamma: I\to M $ is {\em $\Delta$-horizontal} if it is absolutely continuous and $\dot\bgamma(t)\in \Delta,$ for almost every $t\in I$.
Then, we define the {\em CC length} of each $\Delta$-horizontal curve $\bgamma:I\to M $ as
\begin{equation}\label{def_CC_length}
 \len_{\|\cdot\|} (\bgamma):= \int_I \|\dot\bgamma(t)\| \dd t .\end{equation}
 The abbreviation $CC$ stands for Carnot-Carath\'eodory; but some authors use the term {\em sub-Finsler}.
 Finally, we define the {\em(associated) CC metric} $d:=d_{\Delta,\|\cdot\|}$ as the possibly infinite-valued distance function: \begin{multline}
 d(x,y):= 
\inf\bigg\{ \len_{\|\cdot\|} (\bgamma) \,\bigg| \,
 \bgamma:I\to M, \text{ $\Delta$-horizontal, }
 x,y\in\bgamma(I)
 %, \text{ and }\\
 %\dot\bgamma(t)\in \Delta, \text{ for a.e. } t\in [0,1]
 \bigg\},
 \qquad \forall x,y\in M.
\end{multline}

%If a curve is not tangent almost everywhere to the vector field $X$, then its length w.r.t. $d_X$ is infinite.

%Let $M$ be $\R^n$, for some $n\in \N$, or more generally a manifold. Let $\Delta\subseteq TM$ be a continuous subbundle of the tangent bundle, e.g., there are $X_1,\ldots, X_k$ pointwise linearly independent, continuous vector fields in $\R^n$ such that 
%$$\Delta_p={\rm span} \{X_1(p),\ldots, X_k(p)\},\qquad \forall p\in M. $$
%Fix also a continuously varying norm $\|\cdot\|$ on $TM$.
%Let $d=d_{\Delta,\|\cdot\|}$ the distance function defined as \begin{multline}
% d(x,y):= 
%\inf\bigg\{ \int_0^1 \|\dot\bgamma(t)\| \dd t \,\bigg| \,
% \bgamma:[0,1]\to M \text{ absolutely continuous, } x,y\in\bgamma([0,1]), \text{ and }\\ \dot\bgamma(t)\in \Delta, \text{ for a.e. } t\in [0,1]\bigg\}, \qquad \forall x,y\in M.
%\end{multline}
%{\color{red} Assume\footnote{E: can this assumption be removed?} that the distance function $d:=d_{\Delta,\|\cdot\|}$ induces the manifold topology.}

\begin{theorem} \label{d_X rectifiability}
Consider a continuous Carnot-Carathéodory structure $(M,\Delta,\|\cdot\|)$.
Let $\bgamma$ be a curve in $M$ that has finite length with respect to the CC metric $d_{\Delta,\|\cdot\|}$. Then, the reparametrization of $\bgamma$ by
arc length is $\Delta$-horizontal.
Moreover, for every $\Delta$-horizontal curve $\bgamma$, the length with respect to $d_{\Delta,\|\cdot\|}$ equals $ \len_{\|\cdot\|}(\bgamma)$. 
\end{theorem}
\proof
%a beginning of a possible proof:
%Since $X$ is continuous, it is uniformly continuous on a fixed compact set $K$, i.e., for any given $\eps$ there is a $\delta$ so that, if two points $p,q\in K$ are such that $|p-q|\leq\delta$, then $|X_p-X_q|\leq\delta$.
%Cover the curve with balls of radius $\delta$.
%Use some covering lemma...
%[...]
%in each $\delta$-ball choose two points 
 The statement is local in nature. We can therefore think that $M$ is a neighbourhood of ${\bf 0}$ in $\R^d$ and, by scaling the norm, we may assume that $d$ is greater than the Euclidean distance $d_E$.
 Consequently, if $\bgamma: I\to \R^d$ is parametrized by
arc length with respect to the length of $d:=d_{\Delta,\|\cdot\|}$, we have 
$$d_E(\bgamma(s),\bgamma(t))\leq d (\bgamma(s),\bgamma(t))\leq \len_d(\bgamma|^{t}_{s})=\left|t-s\right|,\qquad \forall s,t\in I.$$
In other words, the curve $\bgamma$ is (Euclidean) Lipschitz, so absolutely continuous and differentiable almost everywhere, by Rademacher's theorem.

Let $t$ be a point of differentiability. Without loss of generality, we assume $t=0, \bgamma(t)={\bf 0},$ $\dot\bgamma(0)$ exists and \begin{equation}\label{ipotesi Delta}
 \Delta_{\bgamma(0)}=\R^k\times\{0\}^{d-k}\subseteq \R^d. 
\end{equation}
So, to prove that $\bgamma$
%$=:(\gamma_1,\gamma_2)\in \R^k\times\R^{d-k}$ 
is horizontal at $t=0$, we need to show that $\dot\bgamma(0)
\in \R^k\times\{0\}^{d-k}$.
%={\bf 0}\in \R^{d-k}$. 
Suppose, by way of contradiction, that this is not the case, then, up to reindexing the variables, we have that the last component $\dot\gamma_d(0)$ of $\dot\bgamma(0)$ is not 0. 
%Ab absurd, let $L^X(\gamma)<+\infty$.
%Hence %$L(\gamma):=L^E(\gamma)\leq L^X(\gamma)<+\infty$, and so 
%we can assume that $\gamma$ is parametrized w.r.t. the $d_X$ arc length,
%\begin{equation}\label{arc length}
%L^X(\gamma_{[0,t]})=t.
%\end{equation}
%Let $\gamma=(\gamma_1,\gamma_2)$. Assume that $X_{(0,0)}=(1,0)$ and $|X_p|=1$, for any $p$ in a neighborhood.
%By assumption, $\dot\gamma_2(0)\neq0$. 
Without loss of generality, there is $c>0$ such that $\dot\gamma_d(0)=2c>0$.
From this we can assume that 
\begin{equation}\label{stima}
\gamma_d(t)>ct,\qquad \text{for all small enough } t>0.
\end{equation} 

Let $X_1,\ldots, X_k$ be pointwise linearly independent, continuous vector fields that define 
$\Delta$, as in \eqref{def Delta vectors}. Up to a biLipschitz change of the metric, we can assume that the norm $\|\cdot\|$ is the Euclidean norm, and that the vector fields $X_1,\ldots, X_k$ are orthonormal.
%we are working with the sub-Riemannian distance that makes these vector fields orthonormal.

By the assumption \eqref{ipotesi Delta}, for each $\eps >0$, choose $r=r_\eps$ such that on the Euclidean ball we have
%, for each $p$ in the Euclidean ball $B^E(0,r)$ of center $0$ and radius $r$, we have
\begin{equation}\label{X2 minore di eps}
|(X_j(p))_d|<\eps, \qquad \forall p\in B_E({\bf 0} , r), \forall j\in \{1, \ldots, k\},
\end{equation}
where $(X)_d$ denotes the $d$'th component of a vector field $X$.

From the definition of $d$, for each fixed time $h\in I$, there exists a horizontal curve $\sigma:[0,\tilde h]\to \R^d$, from ${\bf 0}$ to $\bgamma(h)$, parametrized with respect to the (Euclidean) arc length, with 
\begin{equation}\label{Lunghezzadistanza}
\tilde h = \len_E(\sigma)< 2 d(0, \bgamma(h)) \leq 2 h,
\end{equation}
where the last inequality follows from $\bgamma$ being in arc length.

We point out that, for $h$ small enough,  the curve $\sigma$ will be in $B_E(0,r)$: for all $h\in (0,r/2]$ and $t\in [0,\tilde h]$, we have
$$|\sigma(t) | \leq \len_E (\sigma) \leq 2d(0, \bgamma(h)) \leq 2 h <r.$$
From this last calculation, the estimate (\ref{X2 minore di eps}) holds on $\sigma|_{[0,\tilde{h}]}$.

%{\color{red} Because the topology given by the metric $d$ is the usual one, for $h$ small enough, the curve $\sigma_h$ will be in $B_E(0,r)$; more esplicitely, if $h$ is such that $d(0, \bgamma(h))\leq r$ ($\gamma$ and $d$ are continuous!), then, (\ref{Lunghezzadistanza}) implies that the length of $\sigma= \sigma _h$ is less than $r$ and so $\sigma$ is in the ball. }

Since $\sigma$ is horizontal, it is absolutely continuous, and $\|\dot\sigma\| = 1$ almost everywhere. Consequently,
$$ |\dot\sigma_d| =\left| \left( \sum_{j=1}^k\langle X_j, \dot\sigma\rangle X_j(\sigma)\right )_d\right| \leq 
   \sum_{j=1}^k\|\dot\sigma\||\left( X_j(\sigma)\right )_d|  =\sum_{j=1}^k |\left( X_j(\sigma)\right )_d|
$$
%Lipschitz and the integral of its derivative, which is $\pm X$ at almost each point, i.e., $\sigma(t)=\int_0^t\dot\sigma(s)ds.$
%Let $\sigma=(\sigma_1,\sigma_2)$. 
Hence, by (\ref{X2 minore di eps}), we deduce
\begin{equation}\label{sigma2}
|\sigma_d(t)|\leq \int_0^t\sum_{j=1}^k|(X_j(\sigma(s)))_d|{\rm d} s\leq k \eps t.
\end{equation}

%Set $\tilde h=L^E(\sigma)$, since $\sigma$ is parametrized w.r.t. the (Euclidean) arc length, $\sigma(\tilde h) =\bgamma(h)$.
From (\ref{stima}) and then (\ref{sigma2}), for all $h\in (0,r/2]$ we get
 $$0<c h \stackrel{\eqref{stima}}{<}\bgamma_d(h)=\sigma_d(\tilde h)\stackrel{\eqref{sigma2}}{<}k \eps \tilde h.$$
The obtained inequality is in contradiction with \eqref{Lunghezzadistanza} as soon as $\eps$ is taken smaller than $c/(2 k )$. 
We obtained a contradiction that yields the fact that $\bgamma$ is horizontal.

%{\color{red}\footnote{Enrico is tempted to remove this second part. If we leave it, we should improve it.}
Regarding the second part of the statement, for simplicity, we only consider the case when the norm is constantly equal to the Euclidean norm, whereas for the general case, one should localize freezing the norm; see \cite[Proposition 4.2.4]{LeDonneGTM}. We show that for every horizontal curve $\gamma$, we have $\len_d(\gamma)=\len_E(\gamma).$ Indeed, $\len_E\leq \len_d$ since $d_E\leq d,$
and, from the definition of $d$ and the fact that $\gamma$ is horizontal,
\begin{eqnarray*}
\len_d(\gamma)&\stackrel{\rm def}{=}& \sup_{\mathcal P}\sum_{i}d\left(\gamma(t_i),\gamma(t_{i+1})\right)\\
&\leq& \sup_{\mathcal P}\sum_i \len_E (\gamma|^{t_{i+1}}_{t_i})\\
&= &\sup_{\mathcal P}\len_E(\gamma)= \len_E(\gamma)
\end{eqnarray*}
where the $\sup$ is over the partitions $\mathcal P=\{t_i\}$ of the domain of $\gamma.$ 
\qed

From the previous result, we have that a horizontal curve $\bgamma$ is parametrized proportionally to arc length if and only if it has {\em constant speed}, i.e., there is $\lambda\in \R_+$ such that $\|\dot\bgamma\|\equiv \lambda$; it is parametrized by arc length if moreover $\lambda=1$. 
We also have that $\bgamma$ has {\em bounded speed} if there is $\lambda\in \R_+$ such that $\|\dot\bgamma\|\leq \lambda$.

For a better comprehension of the significance of the next result, we stress that, in general, the associated CC metrics may not induce the manifold topology, nor need it be locally compact. 
\begin{corollary}\label{cor:subseq}
    Consider a continuous Carnot-Carathéodory structure $(M,\Delta,\|\cdot\|)$ equipped with its associated CC metric $d_{\Delta,\|\cdot\|}$.
    \begin{enumerate}
    \item 
    Every point admits a neighborhood $U$ such that each sequence of curves in $U$ with a common starting point and uniformly bounded CC-length has, 
    after reparametrizing on $[0,1]$ with bounded speed,
      a subsequence that
    converges uniformly with respect to the manifold topology.

    %WRONG VERSION: If $\bgamma_n: [0,1]\to M$, for $n\in \N$, is a sequence of curves such that $\bgamma_1(0)=\bgamma_n(0)$ and  $\sup \{ \len_{\|\cdot\|} (\bgamma_n) \}\leq \infty$,
   % then, up to reparametrization, the sequence $\bgamma_n$  converges uniformly with respect to the manifold topology to a curve $ \bgamma: [0,1]\to M$.
        \item If $\bgamma_n: I\to M$, for $n\in \N$, is a sequence of curves that have uniformly bounded CC-speed and are converging uniformly with respect to the manifold topology to a curve $ \bgamma: I\to M$, then the curve
        $\bgamma$ is $\Delta$-horizontal and 
         the length structure from \eqref{def_CC_length} satisfies
         \begin{equation}
            \label{eq limits of lenths}
         \len_{\|\cdot\|} (\bgamma) \leq \liminf_{n\to \infty}  \len_{\|\cdot\|} (\bgamma_n). \end{equation}
          \item The metric space $(M,d_{\Delta,\|\cdot\|} ) $ is locally geodesic, i.e.,  every point $\bfp\in M$ has a neighborhood $B_{d_{\Delta,\|\cdot\|}}(\bfp,r)$ where each pair of points can be connected by a curve in $M$ whose length equals the distance of the pair. 
          \item The distance $d_{\Delta,\|\cdot\|}$ is an inner distance.
          \item The metric space $(M, d_{\Delta,\|\cdot\|})$ is locally complete.
    \end{enumerate}
\end{corollary}  
\begin{proof} The statements are local in nature. Fix a point $p\in M$. After restricting to a sufficiently small coordinate neighborhood $U$ of $p$, we may identify $U$ with a bounded open subset of $\mathbb R^d$. We may further assume that $U$ is compactly contained in a larger chart. Since $\Delta$ and $\|\cdot\|$ are continuous, there exists a constant $C\geq 1$ such that \begin{equation}\label{eq:EuclCCcompare} |v|_{\rm Eucl}\leq C\|v\|, \qquad \forall q\in U,\ \forall v\in \Delta_q. \end{equation} Consequently, up to making $U$ smaller we have
\[ d_{\rm Eucl}(x,y)\leq C\, d_{\Delta,\|\cdot\|}(x,y), \qquad \forall x,y\in U. \] 
 
\noindent {\bf (1).} Let $(\bgamma_n)_{n\in \N}$ be a sequence of curves in $U$ with a common starting point and uniformly bounded CC-length  with \[ L:=\sup_{n\in \N} \len_{\|\cdot\|}(\bgamma_n)<\infty . \] After reparametrizing with bounded speed (e.g. proportionally to arc-length), we may assume \[ \len_{\|\cdot\|}(\bgamma_n|_{[s,t]}) \leq L|t-s|, \qquad \forall s,t\in[0,1]. \] By Theorem~\ref{d_X rectifiability}, \[ \len_{d_{\Delta,\|\cdot\|}}(\bgamma_n|_{[s,t]}) = \len_{\|\cdot\|}(\bgamma_n|_{[s,t]}), \] and therefore \[ d_{\Delta,\|\cdot\|}(\bgamma_n(s),\bgamma_n(t)) \leq L|t-s|. \] Using \eqref{eq:EuclCCcompare}, \[ d_{\rm Eucl}(\bgamma_n(s),\bgamma_n(t)) \leq CL|t-s|. \] Hence the family $(\bgamma_n)$ is equi-Lipschitz in Euclidean coordinates. Since all curves start at the same point, \[ d_{\rm Eucl}(\bgamma_n(t),\bgamma_n(0)) \leq C\, d_{\Delta,\|\cdot\|}(\bgamma_n(t),\bgamma_n(0)) \leq CL, \] so the images of all $\bgamma_n$ lie in a common compact subset of a larger chart. By the Arzelà--Ascoli theorem, a subsequence converges uniformly to a curve $\bgamma:[0,1]\to M$, with respect to the manifold topology. 
 
\noindent {\bf (2).} In Lemma~\ref{lem:limcurvehorizontal}, we shall prove a more general statement in the setting of varying structures. Thus, \eqref{eq limits of lenths} follows from \eqref{eq limits of lenths - varying}.

\noindent {\bf (3).} Let $x,y$ be sufficiently close to $p$ in the distance $d_{\Delta, \|\cdot\|}$. Choose a minimizing sequence of horizontal curves joining $x$ and $y$. By (1), after reparametrization a subsequence converges uniformly to a horizontal curve $\bgamma$. By (2), \[ \len_{\|\cdot\|}(\bgamma) \leq \liminf_n \len_{\|\cdot\|}(\bgamma_n) = d_{\Delta,\|\cdot\|}(x,y). \] Since the opposite inequality follows from the definition of the distance, \[ \len_{\|\cdot\|}(\bgamma) = d_{\Delta,\|\cdot\|}(x,y). \] Theorem~\ref{d_X rectifiability} now gives \[ \len_{d_{\Delta,\|\cdot\|}}(\bgamma) = d_{\Delta,\|\cdot\|}(x,y). \] Hence every sufficiently small neighborhood is geodesic. 
 
\noindent {\bf (4).} Let $x,y\in M$ and let $\varepsilon>0$. Choose a horizontal curve $\bgamma$ joining $x$ to $y$ such that \[ \len_{\|\cdot\|}(\bgamma) \leq d_{\Delta,\|\cdot\|}(x,y)+\varepsilon. \] By Theorem~\ref{d_X rectifiability}, \[ \len_{d_{\Delta,\|\cdot\|}}(\bgamma) = \len_{\|\cdot\|}(\bgamma). \] Therefore \[ d_{\Delta,\|\cdot\|}(x,y) = \inf_{\bgamma} \len_{d_{\Delta,\|\cdot\|}}(\bgamma), \] so $d_{\Delta,\|\cdot\|}$ is an inner metric. 
 
\noindent {\bf (5).} Let $(p_n)$ be a $d_{\Delta,\|\cdot\|}$-Cauchy sequence contained in a sufficiently small closed neighborhood that is contained in a coordinate neighborhood $U$. From \eqref{eq:EuclCCcompare}, \[ d_{\rm Eucl}(p_n,p_m) \leq C\, d_{\Delta,\|\cdot\|}(p_n,p_m), \] hence $(p_n)$ is Euclidean Cauchy and converges to some $p_\infty\in U$. Passing to a subsequence, we may assume \[ \sum_{n=1}^{\infty} d_{\Delta,\|\cdot\|}(p_n,p_{n+1}) <\infty. \] For each $n$, by (3), choose a geodesic $\bgamma_n$ joining $p_n$ to $p_{n+1}$. Concatenating the curves $\bgamma_n$ yields a horizontal curve of finite CC-length. By (1) the concatenation converges uniformly to a continuous curve ending at $p_\infty$. Consequently, \[ d_{\Delta,\|\cdot\|}(p_n,p_\infty) \leq \sum_{k=n}^{\infty} d_{\Delta,\|\cdot\|}(p_k,p_{k+1}), \] and the right-hand side tends to $0$. Hence $p_n\to p_\infty$ in the CC metric. Therefore every sufficiently small closed neighborhood is complete. \end{proof}
%Proof of 5. 
%Take a point $p\in M$. Since the norm $\|\cdot\|$ is continuous, we can take $r>0$ such that $B(p,r)$ is compactly contained in some local chart for $M$. 
%In partiucular, there is a constant $C$ such that $d_{Eucl} \leq C d_{\Delta,\|\cdot\|}$.
%Assume $p_n$ is a Cauchy sequence in $B(p,r)$ for $d_{\Delta,\|\cdot\|}$, and moreover we may assume $\sum d_{\Delta,\|\cdot\|}(p_n, p_{n+1}) <\infty$. Consequently, the sequence is also Cauchy for  $d_{Eucl}$. Thus, the sequence $p_n$ converges to some $p_\infty$ in the topology of the chart. We claim that  $p_n$ converges to  $p_\infty$ for the metric $d_{\Delta,\|\cdot\|}$. In fact, take curve $\bgamma_n$ with length smaller than $2 d_{\Delta,\|\cdot\|}(p_n, p_{n+1})$. 
%We concatenate these infinitely many curves, and get a curve that extends continuously to $p_\infty$. Thus, for each $n\in \N$, we obtain a curve joining $p_n $ to $p_\infty$ with length smaller that $2\sum_{k=n}^\infty d_{\Delta,\|\cdot\|}(p_k, p_{k+1})$, which is infinitesimal. 

%semicontinuity with respect to manifold convergence,  locally geodesic, inner metric. proof convex combination

%{\color{blue} $f_n:I\to \R$ uniformly Lipschitz and $f_n\to f$ pointwise, then after taking convex combinations of tails $f_n'\to f'$. (Mazur's lemma)}}

  Let $M$ be a smooth manifold.  
For each $n\in \N$, 
let $(M,\Delta_n,\|\cdot\|_n) $ be  a continuous Carnot-Carath\'eodory structure.
%let $\Delta_\lambda$ be a continuous subbundle of the tangent bundle $TM$ and let $\|\cdot\|_\lambda$ a continuously varying norm on $TM$.
We say that $(M,\Delta_n,\|\cdot\|_n)_{n\in \N} $  
continuously converges to a continuous Carnot-Carath\'eodory structure $(M,\Delta,\|\cdot\|) $ if the following two conditions hold: First, 
 there is $k\in \N$ and, locally in  $M$, there are 
 horizontal frames $X^{(n)}_1,\ldots, X^{(n)}_k$ for $\Delta_n$
 such that, as $n\to \infty$, the sequence of vector fields $X^{(n)}_1,\ldots, X^{(n)}_k$ converges uniformly to a  horizontal frame for $\Delta$.
 Second, the norms $\|\cdot\|_n$ converge to the norm $\|\cdot\|$ uniformly on compact sets of $TM$.

\begin{lemma}\label{lem:limcurvehorizontal} 
Let $(M,\Delta_n,\|\cdot\|_n)_{n\in \N} $  be a sequence of continuous Carnot-Carath\'eodory structures that are
continuously converging to a continuous Carnot-Carath\'eodory structure $(M,\Delta,\|\cdot\|) $.
For $n\in \N$, let $\bgamma_n:I\to M$ be a curve.
Assume that the sequence of horizontal curves $(\bgamma_n)_{n\in\N}$ has uniformly bounded $CC$-speed for $(\Delta_n,\|\cdot\|_n)$ and that it converges uniformly to an absolutely continuous curve $\bgamma: I\to M$, with respect to the topology of $M$.
%, and that, in local coordinates, the sequence of derivatives $\dot\bgamma_n$ converges weakly to the derivative $\dot\bgamma$. 
Then $\bgamma$ is $\Delta$-horizontal and 
         the length structure from \eqref{def_CC_length} satisfies
        \begin{equation}
            \label{eq limits of lenths - varying}
         \len_{\|\cdot\|} (\bgamma) \leq \liminf_{n\to \infty}  \len_{\|\cdot\|_n} (\bgamma_n). \end{equation}
\end{lemma}

\begin{proof}
The assertion is local, so fix a compact subinterval $J\subset I$ and a coordinate neighborhood containing $\bgamma(J)$ and, for all sufficiently large $n$, $\bgamma_n(J)$.
We identify this neighborhood with an open subset of $\mathbb R^d$.
All arguments below are on $J$; a decomposition by such intervals gives the result on $I$.

By continuity of the limiting norm and its positive definiteness on $\Delta$, after shrinking the coordinate neighborhood there is $C>0$ such that
\[
 |v|_{\rm eucl}\leq C\|v\|_n , \qquad \forall v\in\Delta_n(q),
\]
for every sufficiently large $n$ and every $q$ in this neighborhood.
 
The uniform CC-speed bound therefore gives a uniform Euclidean $L^\infty$ bound on $\dot\bgamma_n$.
After discarding finitely many indices, the curves are Euclidean Lipschitz with a common constant.
Therefore, also their uniform limit $\bgamma$ is Lipschitz.

Write the locally convergent horizontal frames as
\[
 X^{(n)}_1,\ldots,X^{(n)}_k\longrightarrow X_1,\ldots,X_k
\]
uniformly, where $X_1,\ldots,X_k$ frame $\Delta$.
For almost every $t\in J$, horizontality of $\bgamma_n$ gives coefficients $a_{nj}(t)$ such that
\[
 \dot\bgamma_n(t)=\sum_{j=1}^k a_{nj}(t)X^{(n)}_j(\bgamma_n(t)).
\]
The frame vectors remain uniformly linearly independent on a possibly smaller neighborhood, so their coefficient maps are uniformly bounded there.
Consequently, the coefficients $a_{nj}$ are uniformly bounded in $L^\infty(J)$.
Passing to a subsequence, for each $j$ we have weak-* convergence in $L^\infty(J)$ of $a_{nj}$ to some $a_j\in L^\infty(J)$, as $n\to \infty$.

The uniform convergence of $\bgamma_n$ and the frames implies uniform convergence of $X^{(n)}_j(\bgamma_n(\cdot))$ to $X_j(\bgamma(\cdot))$.
Consequently, by the weak-* convergence of $a_{nj}$, we obtain that for all $j$   the sequence 
$a_{nj}(\cdot)X^{(n)}_j(\bgamma_n(\cdot))$ converges weak-*  to 
$a_{j}(\cdot)X_j(\bgamma(\cdot))$.

 Therefore, for each $s,t\in J$, the   weak-* convergence  implies
\begin{align*}
 \bgamma(t)-\bgamma(s)
 &=\lim_{i\to\infty}\int_s^t\dot\bgamma_i(u)\,\dd u\\
 &=\lim_{i\to\infty}\sum_{j=1}^k\int_s^t a_{ij}(u)X^{(i)}_j(\bgamma_i(u))\,\dd u\\
 &=\sum_{j=1}^k\int_s^t a_j(u)X_j(\bgamma(u))\,\dd u.
\end{align*}
Consequently, the curve $\bgamma$ is $\Delta$-horizontal on $I$.

Regarding the lengths, notice that since the sequence of  norms is converging locally uniformly and the sequence of tangent vectors is  weak-* converging, we have  
\begin{equation}\label{eq:lscnorms}
\int_J \|\dot\bgamma\| \dd t \leq \liminf_{n\to\infty} \int_J {\|\dot\bgamma_n(\cdot)\|_n} \dd t.
\end{equation}
This can be seen in the following way. In local charts, we can identify $\dot\bgamma_n$ with vectors in $\R^d$ converging weak-* to $\dot\bgamma(t)$. Now, the norms $\|\cdot\|_n, \|\cdot\|$ also depend on the basepoint of the tangent vector. Thus, we can write them as $\|\cdot\|_{n,t}, \|\cdot\|_t$, where $\|\cdot\|_{n,t}$ is the norm on the tangent space of $\bgamma_n(t)$ and $\|\cdot\|_t$ is the norm at $\bgamma(t)$. The norms vary continuously in $t \in I$, and $\|\cdot\|_{n,t}\to\|\cdot\|_t$ locally uniformly. A standard application of weak-* convergence together with the convexity of the integral expression gives
$$\int_J {\|\dot\bgamma(t)\|_t} \dd t \leq \liminf_{n\to\infty} \int_J {\|\dot\bgamma_n(t)\|}_t \dd t.$$
By uniform convergence, and since the curves have uniformly bounded speed, $\lim_{n\to \infty}\|\dot\bgamma_n(t)\|_{n,t}-\|\dot\bgamma_n(t)\|_{t}=0$ and \eqref{eq:lscnorms} follows.

%{\color{red} If we want, we add 1-3 sentences and perhaps one display to reduce this to a more familiar version of lower semicontinuity of norms.}
Hence, by Fatou's lemma, \[ \len_{\|\cdot\|}(\bgamma|_J) = \int_J {\|\dot\bgamma(t)\|}\,\dd t \leq \liminf_{n\to\infty} \int_J {\|\dot\bgamma_n(t)\|_n}\,\dd t=  \liminf_{n\to\infty} \len_{\|\cdot\|_n}(\bgamma_n|_J). \] 
\end{proof}
%\subsection{Extra notation in $\R^d$} \todo{Maybe move this section?}

\subsection{Vandermonde matrix}
Filiform metrics, such as the one given by \eqref{eq:vectdef}, can be conveniently studied with the aid of Vandermonde matrices.  For this subsection, fix  $n\in \N$ with $n\geq 2$.  The (transpose of the) {\em Vandermonde matrix} in $\R^n$ is defined by
\begin{equation}
 \label{def:Vandermonde}
{\bf Van}(x_1,\dots, x_{n}):=\left(\begin{array}{cccc} 1 & 1 & \cdots & 1 \\ 
 x_1 & x_2 & \cdots & x_n \\
 x_1^2 & x_2^2 & \cdots & x_{n}^2 \\
 \vdots & \vdots & \ddots & \vdots \\ 
 x_1^{n-1} & x_2^{n-1} & \cdots & x_n^{n-1} \end{array}\right), \qquad \text { for } x_1,\dots, x_{n}\in \R .
\end{equation}

We recall some simple matrix properties and identities. Each matrix $\bfV:= {\bf Van} (x_1,\dots, x_{n})$ has determinant $\det(\bfV)=\prod_{i<j} (x_j-x_i)$, and is invertible if and only if $x_i\neq x_j$ for all $i\ne j$. 
%Let $x_1=-1,\dots, x_d=1 \in [-1,1]$ be equi-spaced points $x_k=-1+2\frac{(k-1)}{d-1}$ for $k=1,\dots, d$. Let $V_{d,1}=V(x_1,\dots, x_d)$. We then have, $\det(V_{n,1})=\prod_{1\leq i < j \leq n} (x_i-x_j) = c_d>0$, and some constant $C>0$ so that
%\[
%[-C^{-1},C^{-1}]^d \subset V_{d,1}([-1,1]^d) \subset [-C,C]^d.
%\]
 %Then 
%\[
%V(rx_1, \dots, rx_n)=\diag(1,r,r^2, \dots, r^{d-1})V(x_1,\dots, x_n). 
%\]
%and 
%\[
%\prod_{i=1}^d [-r^{i-1} C^{-1},r^{i-1}C^{-1}] \subset V_{d,1}%([-1,1]^d) \subset \prod_{i=1}^d [-r^{i-1} C,r^{i-1}C]. 
%\]
For $a\in \R$, let $\bfL_{a}$ be the $n\times n$ matrix with components 
\begin{equation}\label{eq:def:L}
[\bfL_a]_{ij}:=\begin{cases}
\binom{i-1}{j-1}(-a)^{i-j}&\text{ if } 1\leq j \leq i \leq n\\ 
 0 & \text{ otherwise}
\end{cases}
\end{equation}
%for $1\leq j \leq i \leq d$ and $[\bfL_a]_{ij}=0$ otherwise.
Notice that $\bfL_{a}$ is a lower triangular matrix with diagonal entries equal to one. Thus, $\det(\bfL_{a})=1$. For all $a,x\in \R$ we have the identity 
\begin{equation}\label{eq L applied}
\bfL_a\left(\begin{array}{c} 1 \\ x \\ x^2 \\ \vdots \\ x^{n-1}\end{array}\right)=\left(\begin{array}{c} 1 \\ x-a \\ (x-a)^2 \\ \vdots \\ (x-a)^{n-1}\end{array}\right).
\end{equation}
In particular, we get that $\bfL_{a}\bfL_b=\bfL_{a+b}$ for all $a,b\in \R$, and $\bfL_{-a}=\bfL_{a}^{-1}$ for $a\in \R$ and
\begin{equation}\label{eq:decomposition1}
\bfL_{x_0}{\bf Van}(x_1,\dots, x_n) = {\bf Van}(x_1-x_0, \dots, x_n-x_0), \qquad \forall x_0,x_1,\dots, x_{n}\in \R .
\end{equation}
Thus $a\to \bfL_a$ defines a one-parameter subgroup. Moreover, for every Vandermonde matrix, we have
\begin{equation}\label{eq:decomposition2}
{\bf Van}(\delta x_1, \dots, \delta x_n) = \diag(1,\delta, \delta^2, \cdots, \delta^{n-1}){\bf Van}(x_1, \dots, x_n)
, \qquad \forall \delta,x_1,\dots, x_{n}\in \R .
\end{equation}
Here, $\diag(\lambda_1,\dots, \lambda_n)$ denotes the diagonal matrix with entries $\lambda_1,\dots, \lambda_n\in \R$. Write 
\begin{equation}\label{eq:def:D}
\bfD_{\delta}:=\diag(1,\delta, \delta^2, \cdots, \delta^{n-1}), \qquad \text{ for } \delta\in \R.
\end{equation}
With these, we can approximately invert a Vandermonde matrix. Next, with the notation $\|\bfM\|$, we mean the operator norm of a matrix $\bfM$, where on $\R^n$ we use the Euclidean norm $|v|$ for a vector $v\in \R^n$.

\begin{lemma}\label{lem: approxinv}
Fix $n\in \N$, with $n\geq 2$, and $x_1,\dots, x_{n}\in \R$. 
 Let $m :=\min_{i\neq j}|x_i-x_j|$. The Vandermonde matrix $\bfV:={\bf Van}(x_1,\dots, x_n)$ is invertible if and only if $m >0$.  
%Moreover, for each $K\geq 1$ there exists\footnote{{\color{red} E: would that be useful to express $C$ as an explicit function of $K$?}} a constant $C$ such that if 
%$m >0$ and
%$M\leq Km $,
 Further, if $m >0$, then for all $x_0\in \R$ the matrix
 \[
\bfM :=\bfD_{m }^{-1}\bfL_{x_0}\bfV
 \]
 is invertible with
% \[
% \|\bfM^{-1}\|+\|\bfM\|\leq n^{n-1} (M/m )^{(n-1)^2}+n (M/m )^{n-1},
% \]
  \[
 \|\bfM^{-1}\| \leq n^{n-1} \max\{1,M/m \}^{(n-1)^2} ,
 \]
 where $M:=M(x_0):=\max_{i } |x_i-x_0|$.
\end{lemma}
\begin{proof}
    The proof follows by direct calculation. We have $$\bfM={\bf Van}((x_1-x_0)/m ,(x_2-x_0)/m , \dots, (x_n-x_0)/m ),$$ and thus each entry of $\bfM$ is bounded by $\max\{1,M/m \}^{n-1}$, and thus $\|\bfM\|\leq n \max\{1,M/m \}^{n-1}$.

    By assumption $|x_i-x_j|\geq m $ for all $i\neq j$ and thus $|\det(\bfM)|\geq 1$. Let $\sigma_i>0$ be the singular values of $\bfM$ in increasing order. Then $\|\bfM\|=\sigma_n$ and $\|\bfM^{-1}\|=\sigma_1^{-1}$. We have 
    \[
    1 \leq |\det(\bfM)|=  \prod_{i=1}^n \sigma_i\leq \|\bfM\|^{n-1}\sigma_1,
    \]
    and thus
    \[
    \|\bfM^{-1}\| \leq \|\bfM\|^{n-1}\leq n^{n-1} \max\{1,M/m \}^{(n-1)^2}.
    \]
    
\end{proof}
%{\color{blue} Added an upper bound for $C$ and a proof.}

 Based on the above lemma, the matrix $\bfV^{-1}$ can be approximated by $\bfD_{m}^{-1}\bfL_{x_0}$, where the error is given by multiplication by $\bfM^{-1}$. Our estimates will depend on the norms of these operators, and thus this simplification corresponds to a multiplicative error.

\subsection{Singular filiform-type sub-Riemannian metrics}

We fix $d\in \N$ with $d\geq 3$. 
If $\bfX$ is a continuous vector field on $\R^d$, its value at $\bfp\in \R^d$ is denoted by $\bfX(\bfp)$.  

We say that it is {\em uniquely globally integrable} if for every $\bfp\in \R^d$ there is a unique solution $\bgamma_\bfp:\R\to \R^d$ to $\bgamma_\bfp'(t)=\bfX({\bgamma(t)})$ with $\bgamma_\bfp(0)=\bfp$. For uniquely integrable vector fields, we denote the solution as $e^{t\bfX}\bfp=\bgamma_\bfp(t)$. The vector fields $\bfX$ and $\bfY=\bfY_f$ in \eqref{eq:vectdef} are uniquely globally integrable vector fields, as shown in \eqref{flows}.

We will derive upper and lower bounds for the distance \eqref{eq:metricdef} in terms of the oscillation of $f:\R\to \R$. The flows of the vector fields $\bfX,\bfY$ from \eqref{eq:vectdef} are simple to compute:
\begin{equation}\label{flows}
e^{t\bfX}\bfp=\bfp+t\bfX(\bfp)=\bfp+t\bfe_1\qquad \text{and} \qquad e^{t\bfY}\bfp=\bfp+t\bfY(\bfp),\qquad \forall \bfp\in \R^d,\forall t\in\R,
\end{equation}
where $\bfe_1$ is the unit vector in the first coordinate direction, and we use the fact that $\bfY$ leaves $x_1$ constant, and it is a constant vector field in each slice $\{x_1=a\}$ for $a\in \R$.

We say that a path $\bgamma=(\gamma_1,\dots, \gamma_d):[0,1]\to \R^d$ is {\em horizontal} if it is absolutely continuous and $\bgamma'(t) \in {\rm {span}}\{\bfX(\bgamma(t)) , \bfY({\bgamma(t)})\}$ for a.e. $t\in [0,1]$. Because of the form of the vector fields, for horizontal curves, we have
$\bgamma'(t)=\gamma_1'(t)\bfX({\bgamma(t)}) + \gamma_2'(t) \bfY({\bgamma(t)})$ for a.e. $t\in [0,1]$.
In this case, we then define 
\[
\len(\bgamma):=\int_0^1 \sqrt{\gamma_1'(t)^2 + \gamma_2'(t)^2}\dd t.
\]
%In terms of such paths, 
We get $\bd(\bfp,\bfq)\leq \len(\bgamma)$ whenever $\bgamma$ is a horizontal curve connecting $\bfp$ to $\bfq$. 

\begin{remark}\label{rmk endpoint}

For all $f \in C^0(\mathbb{R})$ and all $d > 2$, the endpoint map for the structure given by  
 $\bfX,\bfY$ as in \eqref{eq:vectdef}
 is well-defined. Namely, for all $a, b : [0, 1] \to \mathbb{R}$ integrable and for all $\bfp \in \mathbb{R}^d$, there exists a unique absolutely continuous curve $\gamma : [0, 1] \to \mathbb{R}^d$ such that 
\[
\begin{cases} 
\gamma(0) = \bfp \\
\dot{\gamma} = a \bfX(\gamma) + b \bfY(\gamma) .
\end{cases}
\]
In fact, writing $\gamma = (\gamma_1, \dots, \gamma_d)$, the ODE becomes 
\[
(\dot{\gamma}_1, \dots, \dot{\gamma}_d) = \left( a, b, b \cdot (f \circ \gamma_1), \dots, b \cdot (f \circ \gamma_1)^{d-2} \right).
\]
Thus, writing $\bfp = (p_1, \dots, p_d)$, we integrate the first component: 
\begin{equation}\label{first component}
\gamma_1(t) = p_1 + \int_0^t a(s) \, \dd s
\end{equation}
and then the rest:
\begin{equation}\label{second components}
\gamma_i(t) = p_i + \int_0^t b(s) \left( f \left( p_1 + \int_0^s a(\tau) \, \dd\tau \right) \right)^{i-2} \dd s, \qquad \forall i \in \{2, \dots, d\},
\end{equation}
which is integrable because $a, b$ are so, and $f$ is continuous. 
In addition, notice that if $f_n \in C^0(\mathbb{R})$ is a sequence that converges locally uniformly to $f \in C^0(\mathbb{R})$, then, fixed $a, b$, we get that the sequence of points $\gamma(1)$ associated to $f_n$ converges to the point $\gamma(1)$ associated to $f$.
\end{remark}

Next, we shall describe the horizontal paths that we use to connect pairs of points. Let $d\geq 2$ be an integer, $\bft:=(t_1,\dots, t_{d})\in \R^d$, and $\bfy:=( y_1, \dots, y_{d-1})\in \R^{d-1}$. 
%Denote $\bfy=(y_1,\dots, y_{d-1})$. 
Consider $\bfp\in \R^d$ and consider the sequence of points $\bfp^1, \ldots, \bfp^{d+1}$ defined as 
\begin{equation}
    \label{some points}
\begin{cases} 
\bfp^1:=\bfp, &\\
\bfp^{i+1}:=e^{-t_i \bfX}e^{y_i \bfY}e^{t_i\bfX}\bfp^i
,&\forall i\in\{1, \dots, d-1\}\\
\bfp^{d+1}:=e^{t_{d}\bfX} \bfp^{d}.
\end{cases}
%\bfp^1:=\bfp, \bfp^{i+1}:=e^{-t_i \bfX}e^{y_i \bfY}e^{t_i\bfX}\bfp^i
\end{equation}
%for $i=1, \dots, d-1$ and let $\bfp^{d+1}=e^{t_{d+1}X} \bfp_{d}$. 
These points are connected by horizontal paths defined as follows using the flows of $\bfX,\bfY$. For $i\in\{1,\dots, d-1\}$ set
\begin{equation}
    \label{some curves}
\bgamma^i(s):=\begin{cases} e^{3s t_i\bfX}\bfp^i & \text{for } s\in [0,1/3] \\
e^{3(s-\frac{1}{3})y_i \bfY}e^{t_i\bfX}\bfp^i & \text{for } s\in (\frac{1}{3},\frac{2}{3}] \\
e^{-3(s- \frac{2}{3})t_i\bfX}e^{y_i \bfY}e^{t_i\bfX}\bfp^i & \text{for } s\in (\frac{2}{3},1].\end{cases}
\end{equation}
Also, define for $s\in [0,1]$
\[
\bgamma^{d}(s):=e^{st_{d}\bfX} \bfp^{d}.
\]
Each $\bgamma^i$ goes from $\bfp^i$ to $\bfp^{i+1}$.
Let $\bgamma$ be the concatenation of $\bgamma^i$ for $i=1,\dots, d$. 
The paths $\bgamma^i$ are horizontal and so is $\bgamma$. The length of $\bgamma$ is direct to compute: 
\begin{equation}\label{eq:gamlen}
\len(\bgamma)=|t_d|+\sum_{i=1}^{d-1} \left(2|t_i| + |y_i|\right).
\end{equation}

\begin{figure}
\begin{tikzpicture}[>=stealth, thick]

% Axes
\draw[->] (0,-2.4) -- (0,3) node[above] {$x_2$};
\draw[->] (-1,-2) -- (8,-2) node[right] {$x_1$};

% Top horizontal arrow
\draw[->] (0.4,1.6) -- (6.1,1.6) node[midway,above] {$t_d$};

% Points on vertical axis
\fill (0,1.5) circle (2pt);
\fill (0,-0.2) circle (2pt);

% Step structure
\fill (3,1.5) circle (2pt);
\fill (3,-0.2) circle (2pt);
\fill (1.5,-0.2) circle (2pt);
\fill (1.5,-2) circle (2pt);

% Horizontal segments
\draw (0,1.5) -- (6.5,1.5);
\draw (0,1.5) -- (2,1.5);
\draw (0,-0.2) -- (3,-0.2);
\draw (2,-2) -- (7,-2);

% Vertical segments
\draw (3,1.5) -- (3,-0.2);
\draw (1.5,-0.2) -- (1.5,-2);

% Small arrows along segments 
\draw[->] (0.5,-0.1) -- (2.6,-0.1) node[midway,above] {$t_2$};

\draw[<-] (0.5,1.4) -- (2.6,1.4) node[midway,below] {$-t_2$};

\draw[->] (0.3,-1.9) -- (1.2,-1.9);
\node at (0.9,-1.7) {$t_1$};

\draw[<-] (0.3,-0.35) -- (1.2,-0.35);
\node at (0.8,-0.6) {$-t_1$};

% Vertical arrows
\draw[->] (3.2,0.3) -- (3.2,1.2) node[midway,right] {$y_2$};
\draw[->] (1.7,-1.5) -- (1.7,-0.5) node[midway,right] {$y_1$};

% Right endpoint
\fill (6.5,1.5) circle (2pt);

% Text
\node at (7,2.7) {$d=3$};

\end{tikzpicture}
\caption{Projection of the curve $\gamma$ on the $(x_1,x_2)$ plane. }
\end{figure}
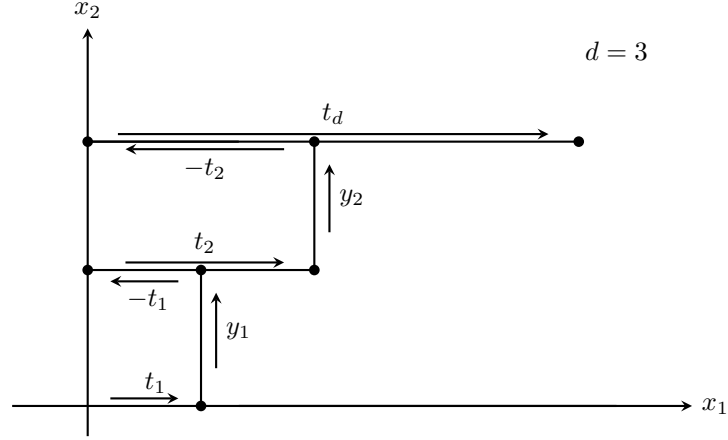
 
 Here and below we identify $\R^d=\R\times \R^{d-1}$ and write points as $\bfp=(p_1,\bfp_{-1}) \in \R^d$ for $p_1\in \R, \bfp_{-1}\in \R^{d-1}$. In the next lemma, we use the Vandermonde matrix \eqref{def:Vandermonde} with $n:=d-1$. %Where needed, we will write $\bfe_i$ for the standard basis vectors of $\R^d$.  

\begin{lemma}\label{lem: simpleflow}
Let $\bfp=(p_1,\bfp_{-1} )\in \R^d$ and $f:\R\to \R$.
\begin{enumerate}
\item%[\eqref{lem: simpleflow}.i]
For $\bfy\in \R^{d-1}$, and $\bft\in \R^{d}$,
the associated point $\bfp^{d+1}$ constructed in \eqref{some points} using \eqref{eq:vectdef} satisfies 
\begin{equation}\label{Madrid19052026_15}
\bfp^{d+1}=
\bfp + (t_{d},{\bf Van}(f(p_1+t_1), \dots f(p_1+t_{d-1}))\bfy).
\end{equation}
\item%[\eqref{lem: simpleflow}.i]
Let $\bfq =(q_1,\bfq_{-1})\in \R^d$, $(t_1,\dots, t_{d-1})\in \R^{d-1}$. Assume that $\bfV:={\bf Van}(f(p_1+t_1), \dots f(p_1+t_{d-1}))$ is invertible and put 
$\bfy:=\bfV^{-1}(\bfq_{-1}-\bfp_{-1})$, and $t_d:=q_1-p_1$. Then, the  points   in \eqref{some points} associated to $\bft$, $\bfy$, and $\bfp$ give $\bfp^{d+1}=\bfq$ and consequently
\begin{equation}\label{Madrid19052026_16}
\bd(\bfp,\bfq)\leq |p_1-q_1| + 2d \left(\max_{i\in \{1,\dots,d-1\}}|t_i|\right) +d \left |\bfV^{-1}(\bfq_{-1}-\bfp_{-1})\right|.
\end{equation}
\end{enumerate}
\end{lemma}
\begin{proof}
%Let $(x^k_i)_{i=1}^d$ be the components of $\bfp^k$. 
By the explicit flows in \eqref{flows}, we have 
$$ (\bfp^{k+1})_i  = y_k f(p_1+t_k)^{i-2} + (\bfp^{k})_i, \qquad \forall i\in \{2,\ldots, d\}, \forall k\in\{1, \ldots, d-1\} $$
and
$$ (\bfp^{k+1})_1  =  (\bfp^{k})_1, \qquad \forall k\in\{1, \ldots, d-1\} .$$
Thus, by induction, one shows for each $k\in\{1,\dots, d\}$ that
\[
(\bfp^k)_1=p_1\quad \text{ and } \quad (\bfp^k)_i = (\bfp)_i + \sum_{j=1}^{k-1} f(p_1+t_j)^{i-2} y_j,\qquad \forall i\in \{2,\ldots, d\} .
\]
Thus, $\bfp^{d}=\bfp + (0,{\bf Van}(f(p_1+t_1), \dots, f(p_1+t_{d-1}) )\bfy)$. Applying the final flow $e^{t_{d}\bfX}$ to $\bfp^{d}$ yields   \eqref{Madrid19052026_15}.

For the second part of the claim, the first part of the proof and the definitions of $\bfV$ and $\bfy$ give
%choose $\bfy:=\bfV^{-1}(\bfq_{-1}-\bfp_{-1})$, and $t_d:=q_1-p_1$. Then $\bfp^{d+1}=\bfq$ and 
$$\bfp^{d+1}=\bfp + (t_{d},{\bf Van}(f(p_1+t_1), \dots f(p_1+t_{d-1}))\bfy )=\bfq.$$
The points give a horizontal curve $\bgamma$ of length \eqref{eq:gamlen} connecting $\bfp$ to $\bfq$:
\begin{eqnarray*} 
\bd(\bfp,\bfq)&\leq& 
\len(\bgamma)\\
&\stackrel{\eqref{eq:gamlen}}{=}& |t_d|+\sum_{i=1}^{d-1} \left(2|t_i| + |y_i|\right)\\
&= & |p_1-q_1| + \sum_{i=1}^{d-1}\left(2|t_i| +  |(\bfV^{-1}(\bfq_{-1}-\bfp_{-1}))_i| \right)\\
&\leq&
|p_1-q_1| + 2d (\max_{i\in \{1,\dots,d-1\}}|t_i| )+d |\bfV^{-1}(\bfq_{-1}-\bfp_{-1})|.
\end{eqnarray*}
\end{proof}

More bounds for the distance are obtained as follows. Recall that closed intervals in $\R$ are denoted $I=[a,b]$, with $a<b \in \R$, and $2I:=[a-(b-a)/2, b+(b-a)/2]$ is an interval with the same center and twice the length. The length of an interval $I=[a,b]$ is denoted $|I|:=b-a$. 

\begin{lemma}\label{lem: distest} Let $f:\R\to \R$ be continuous.  Let $\bfp= (p_1, \bfp_{-1}), \bfq= (q_1, \bfq_{-1} )\in \R\times\R^{d-1}$.

\begin{enumerate}
 \item If $f:\R\to \R$ is constant, then $\bd(\bfp,\bfq)=\infty$ unless $\bfq-\bfp=(t,s,f(p_1)s, \dots, f(p_1)^{d-2}s)$ for some $t,s\in \R$. 
 
 \item If for some interval $I\subset \R$ the function $f$ is constant on $2I$ and $p_1,q_1\in I$, then $\bd(\bfp,\bfq)\geq |I|$ unless $\bfq_{-1}-\bfp_{-1}=(s,f(p_1)s, \dots, f(p_1)^{d-2}s)$ for some $s\in \R$. 
 
 \item Assume that $I\subset \R$ is an interval and that $\delta:=\osc(f,2I)>0$. %Let $\bfp, \bfq\in \R^d$ with 
 Assume $p_1,q_1\in I$, 
  and $s_1,\dots, s_{d-1}\in 2I$
 %and $t_1,\dots, t_{d-1}\in 2I$
 are such that $f(s_i)\neq f(s_j)$ for all distinct $i,j\in \{1,\ldots,d-1\}$. Let $\bfV:={\bf Van}(f(s_1),\dots, f(s_{d-1} ))$. Then, we have
 \begin{multline}\label{eq bounds for d}
 \min\left(|I|, 
 \frac{1}{2}|p_1-q_1|+\frac{1}{2d}\left|\bfD_{\delta^{-1}}\bfL_{f(p_1)}(\bfq_{-1}-\bfp_{-1})\right|\right)\leq \bd(\bfp,
 \bfq)\\
 \leq (4 d + 1) |I|  +d\left|\bfV^{-1}(\bfq_{-1}-\bfp_{-1})\right|.
 \end{multline}
Moreover, if $(s_i)_{i=1}^{d-1}$ are chosen so that $(f(s_i))_{i=1}^{d-1}$ forms an equi-partition of $f(2I)$, then there exists a constant $C_d\geq 1$ depending only on $d$, for which 
 \begin{equation}\label{eq another bounds for d}
\bd(\bfp,
 \bfq)\leq C_d  \left(|I| +   \left|\bfD_{\delta}^{-1}\bfL_{f(p_1)}(\bfq_{-1}-\bfp_{-1})\right| \right) .
 \end{equation}
 \end{enumerate}
%\todo{Check. Can we absorb $|p_1-q_1| $ into $|I|$? }
\end{lemma}
\begin{proof}
If $f$ is constant, then $\bfY=(0,1,f(p_1), \dots, f(p_1)^{d-2})$ and $\bfX=(1,0,\dots, 0)$ are constant vector fields and clearly the claim follows, because only points that differ by $ t \bfX + s\bfY$, for some   $ t,s\in \R$ can be connected. Next, if $I\subset \R$ is an interval such that $f|_{2I}$ is constant and $p_1,q_1\in I$, then unless $\bfq_{-1}-\bfp_{-1}=(s,f(p_1)s, \dots, f(p_1)^{d-2}s)$, every horizontal curve $\bgamma$ connecting them must have a first coordinate $\gamma_1$ that leaves the interval $2I$ and re-enters. This yields the desired lower bound for the length by a simple argument as in \eqref{eq:p1bound}. 

Next, for the claimed equation \eqref{eq bounds for d}, we first prove the lower bound for the distance, and hence for the length of any curve joining the points.
Assume that $\bgamma=(\gamma_1,\dots, \gamma_d)=(\gamma_1,\bgamma_{-1}):[0,1]\to \R^d$ is a horizontal curve between $\bfq $ and $\bfp $.
%i.e., $\bgamma'(t)=\gamma'_1(t)\bfX_{\bgamma(t)}+\gamma_2(t)'\bfY_{\bgamma(t)}$ for a.e. $t\in [0,1]$. 
Note that $\bgamma_{-1}'(t)=\gamma_2'(t)\bfY_{-1}({\bgamma(t)})$. 
If the first coordinate $\gamma_1$ of $\bgamma$ takes  some values outside the interval $2I$, then the length of $\bgamma$ would exceed $|I|$, and we get the lower bound by a simple argument as in \eqref{eq:p1bound}. 
%Thus, if this happens for all such $\gamma$'s, then the lower bound in \eqref{eq bounds for d} follows immediately. 
Suppose instead that  $\bgamma$ is such that $\gamma_1$ takes values within $2I$. 
We are going to give a different lower bound for $\bd(\bfp, \bfq)$. 
To do so, we are going to use the matrices
$\bfD_{\delta} $ and $\bfL_a$ defined in \eqref{eq:def:D} and \eqref{eq:def:L}, respectively, and we see them as acting on $\R^{n}$, with $n:=d-1$, seen inside $\R^d$ as those points that have 0 first coordinate. In particular, 
since the first component of $\bfY=(0,\bfY_{-1})\simeq \bfY$ vanishes, we identify it with a vector in $\R^{d-1}$.
%Let $\delta=\osc(f,2I)$. 
Then, using the definitions, together with \eqref{eq L applied}, we calculate
\begin{equation}\label{eq:calcolo}
\bfD_{\delta^{-1}}\bfL_{f(p_1)}\bfY_{-1}({\bgamma(t)}) = \sum_{i=2}^{d} \left(\frac{f(\gamma_1(t))-f(p_1)}{\delta}\right)^{i-2}\partial_i.
\end{equation}
Noticing that each coordinate of this last vector field has absolute value at most one, we bound:
\begin{eqnarray*}
 \left|\bfD_{\delta^{-1}}\bfL_{f(p_1)}(\bfq_{-1}-\bfp_{-1})\right | 
&=&
\left |\bfD_{\delta^{-1}}\bfL_{f(p_1)}(\bgamma_{-1}(1)-\bgamma_{-1}(0))\right|
\\
&=&
\left|
\int_0^1 \dfrac{\dd}{\dd t}
\bfD_{\delta^{-1}}\bfL_{f(p_1)}\bgamma_{-1}(t)\dd t\right| \\
&=&
\left|
\int_0^1 
\bfD_{\delta^{-1}}\bfL_{f(p_1)}\gamma_{2}'(t)\bfY_{-1}({\bgamma(t)})\dd t\right| \\
&\stackrel{\eqref{eq:calcolo}}{\leq} &
d \int_0^1 
|\gamma_{2}'(t)|
 \dd t \leq d\,\len(\bgamma).
\end{eqnarray*}
In addition, we have another bound: 
\begin{equation}\label{eq:p1bound}
|p_1-q_1|=\left| \int_0^1 \gamma_1'(t)\dd t \right| \leq \int_0^1 |\gamma_1'(t)|\dd t \leq \len(\bgamma).
\end{equation}
Taking the average of the previous two lower bounds for $\len(\bgamma)$ and taking the infimum over $\bgamma$, we get the lower bound in \eqref{eq bounds for d}.

Regarding the upper bound in \eqref{eq bounds for d}, the assumption on the $s_i$'s implies that the Vandermonde matrix $\bfV$ is invertible. Take 
$t_i:=s_i-p_1$, for $i\in \{1,\ldots, d-1\}$,
$t_d:=q_1-p_1$, and $\bfy := \bfV
%(f(p_1+t_1), \cdots, f(p_1+t_{d-1}))
^{-1}(\bfq_{-1}-\bfp_{-1})$. 
Lemma~\ref{lem: simpleflow} gives a chain of points from $\bfp$ to $\bfq$.
In addition, notice that $$
|p_1-q_1| \leq |I|\quad \text{ and } \quad
|t_i| = |s_i-p_1| \leq 2|I|,$$
since $s_i\in 2I$ and $p_1\in I$.
The upper bound in \eqref{eq bounds for d} then follows from the estimate \eqref{Madrid19052026_16}.
%\eqref{eq:gamlen}.

Finally, if $s_1,\dots, s_{d-1}$ is chosen so that 
%$\{f(s_{i}): i=1,\dots, d-1\}$ 
$f(s_{1}) ,\dots, f(s_{d-1})$ 
is an equi-partition of $f(2I)$, then $m:=\min_{i\neq j} |f(s_{i})-f(s_{j})|=\frac{\delta}{d-2} $  
and $M:=\max_{i }|f(s_{i})-f(p_{1})|\leq \delta$. 
So, for the matrix 
$\bfM :=\bfD_{m }^{-1}\bfL_{f(p_1)}\bfV$,
Lemma~\ref{lem: approxinv}, applied with $x_0:=f(p_1)$ and $x_i:=f(s_i)$, gives 
\begin{equation}\label{eq-bound-M}
     \|\bfM^{-1}\| \leq n^{n-1} \max\{1,M/m \}^{(n-1)^2} 
\leq (d-1)^{d-2} (d-2 )^{(d-2)^2} ,
\end{equation}
which only depends on $d$. Moreover, we bound one of the terms in the right-hand side of \eqref{eq bounds for d} as
\begin{eqnarray*}
    d\left|\bfV^{-1}(\bfq_{-1}-\bfp_{-1})\right|
    &=& d\left| \bfM^{-1}   \bfD_{m }^{-1}\bfL_{f(p_1)}(\bfq_{-1}-\bfp_{-1})\right|\\
&\leq&  
    d\,\| \bfM^{-1} \|\,\| \bfD_{d-2} \|\,\left| \bfD_{\delta}^{-1}\bfL_{f(p_1)}(\bfq_{-1}-\bfp_{-1})\right|,
\end{eqnarray*}
where we have used that $\bfD_{m } = \bfD_{\delta/(d-2) }= \bfD_{\delta} \bfD_{d-2}^{-1} $.
Consequently, together with the bound \eqref{eq-bound-M}, we get the final estimate \eqref{eq another bounds for d} for a suitable $C_d$. 
\end{proof}

We combine the previous estimates to characterize when the above-constructed $(\R^d,\bd)$ is a metric space with finite distances and when it is homeomorphic to $\R^d$ equipped with the standard topology. Recall that we say that $f$ has positive oscillation if $f$ is not constant on any closed and bounded interval $I$ of positive length.

\begin{lemma}\label{lem: characterization}
 The function $\bd$ in \eqref{eq:metricdef} is a distance function  such that on every compact set there exists $C>0$ such that $\bd_{\rm eucl}\leq C \bd$. The metric space $(\R^d,\bd_f)$ is a length space. 
 The distance function is finite-valued if and only if $f$ is not constant. This distance function induces the standard topology of $\R^d$ if and only if $f$ has positive oscillation. 
\end{lemma}
\begin{proof}
 By construction, the function $\bd$ is constructed via a length structure (in the sense of \cite[Section~2.1.1]{bbi}), is symmetric, and satisfies the triangle inequality. Moreover, for every compact set $K\subseteq \R^d$, the function $f$ is bounded on the first-coordinate projection of $K$, being continuous. Thus, the Euclidean norms of the vector fields $\bfX_f$ and $\bfY_f$ can a bounded on $K$. Therefore, for any compact set, there exists $C>0$ such that $\bd_{\rm eucl}\leq C \bd$ holds on this set. In particular, the function $\bd$ is positive.

If $f$ is constant, then Clause (1) in Lemma~\ref{lem: distest} says that $\bd$ is not finite-valued. If $f$ is not constant, for all $\bfp$ and $\bfq$ there is a large enough interval $I$ such that \eqref{eq another bounds for d} gives a finite value.

If $f$ does not have positive oscillation, then it is constant in some interval. Thus,  by Clause (2) in Lemma~\ref{lem: distest}, we get a different topology.
Suppose next that  $f$ has positive oscillation. On the one hand,  the local property $\bd_{\rm eucl}\leq C \bd$ tells us that $\bfp \to \bfq $ in the Euclidean topology if $  \bd(\bfp,\bfq)\to 0$. On the other hand, 
if $\bfp \to \bfq $, then \eqref{eq another bounds for d} gives $  \bd(\bfp,\bfq)\to 0$, by choosing $I$ small enough (fixing $\delta>0$). 
\end{proof}

A metric space $(X,d)$ is {\em boundedly compact} if all closed bounded sets are compact. A metric space $(X,\bd)$ is {\em geodesic}, if for every $\bfx,\bfy\in X$ there exists a rectifiable curve $\bgamma$ for which $\len_\bd(\bgamma)=\bd(\bfx,\bfy)$. Notice that by Theorem~\ref{d_X rectifiability} this length is equal to $L_{\|\cdot\|}(\bgamma)$, and this is how we use the following proposition in later sections.

\begin{proposition}\label{prop:boundedcompact}
    Let $f:\R\to \R$ be a continuous function that has doubling oscillation, and take the metric $\bd_f$ defined in \eqref{eq:metricdef} via the vector fields \eqref{eq:vectdef}. Then, $(\R^d,\bd_f)$ is a complete, geodesic, and boundedly compact metric space that is homeomorphic to $\R^d$ equipped with the Euclidean metric.
\end{proposition}
\begin{proof}
    From Lemma~\ref{lem: characterization}, we know that $(\R^d,\bd_f)$ is a length space homeomorphic to $\R^d$ equipped with the Euclidean metric. Note here that $L_{\|\cdot\|}$ defines a length structure in the sense of \cite[Definition 2.1.1]{bbi} and Theorem~\ref{d_X rectifiability} shows that $L_{\|\cdot\|}$ equals $L_{\bd_f}$. Thus, it is locally compact. Next, to show bounded compactness, it suffices to show that every ball $B_\bd(\bfz, R)$ has bounded Euclidean diameter. This follows immediately from Lemma~\ref{lem: distest}. Now, \cite[Proposition~2.5.28]{bbi} implies that $(\R^d, \bd_f)$ is complete. Then, \cite[Theorem 2.5.23]{bbi} (alternatively \cite[Proposition 3.1.6]{LeDonneGTM}) implies that $(\R^d, \bd_f)$ is geodesic. 
\end{proof}

\subsection{Ball-Box Theorem} \label{subsec: ballbox}

For 
%$\bfx=(x_1,\dots, x_{d-1})\in \R^{d-1}$, 
$\delta>0$ and $r>0$, define the $(\delta,r)$-{\em box} in $\R^d$ as
\begin{eqnarray*}
{\BBox}(\delta,r)&:=&[-r,r] \times \bfD_{\delta}\left(\prod_{i=1}^{d-1}[-r,r]\right) 
\\&=& [-r,r] \times [-r,r]  \times \left[ -r\delta , r\delta \right]\times   \left[-r\delta^{2},r\delta^{2}\right]\times\cdots \times  \left [-r\delta^{d-2},r\delta^{d-2}\right]
%\bfx+\bfD_{\delta}\left(\prod_{i=1}^{d-1}[-r,r]\right) = \prod_{k=1}^{d-1} [x_k-r\delta^{k-1},x_k+r\delta^{k-1}]
.
\end{eqnarray*}

Let $\bfp= (p_1, \bfp_{-1}), \bfq= (q_1, \bfq_{-1} )\in \R\times\R^{d-1}$. We define a type of translation operator centered at $\bfp$ by
\[
\bfT^f_{\bfp}(\bfq):= \bfp+(q_1,\bfL^{-1}_{f(p_1)}(\bfq_{-1})).
\]
Where $f$ is clear from context, or fixed, we may omit the superscript.

\begin{theorem}[Ball-Box Theorem]\label{thm: ballbox}
Given $d\in \N$ with $d\geq 3$ there exists $c_d\in(0,1)$ with the following property: Let $f:\R\to \R$ be continuous and let   $\bd:=\bd_f$ be as in \eqref{eq:metricdef} via the vector fields \eqref{eq:vectdef}.
Assume that $\osc(f,I)>0$ for every interval $I\subset \R$.
% There exist a constant $c>0$ such that 
 Then, for every $\bfp %=(p_1,\bfp_{-1})
 \in \R^d$ and every $ r>0$ we have
\begin{eqnarray*}
\bfT_{\bfp}( {\BBox}(\delta_1,c_d r))
%\left[p_1-\frac{r}{4d},p_1+\frac{r}{4d}\right] \times \bfL_{f(p_1)}^{-1}(R(\bfL_{f(p_1)}(\bfp_{-1}),\delta_1,c r))
 \subset B_\bd(\bfp,r) \subset 
\bfT_{\bfp}( {\BBox}(\delta_2, 2dr)),
%[p_1-2r,p_1+2r] \times \bfL_{f(p_1)}^{-1}(R(\bfL_{f(p_1)}(\bfp_{-1}),\delta_2, 2r)),
\end{eqnarray*}
%
%\begin{multline*}
%\left[p_1-\frac{r}{4d},p_1+\frac{r}{4d}\right] \times \bfL_{f(p_1)}^{-1}(R(\bfL_{f(p_1)}(\bfp_{-1}),\delta_1,c r))\\ \subset B(\bfp,r) \\
%\subset [p_1-2r,p_1+2r] \times \bfL_{f(p_1)}^{-1}(R(\bfL_{f(p_1)}(\bfp_{-1}),\delta_2, 2r)),
%\end{multline*}
where $\delta_1:=\osc(f,[p_1-{2 c_d r} ,p_1+{2 c_d r}])$ and $\delta_2:=\osc(f,[p_1-2r,p_1+2r])$. 
\end{theorem}

\begin{proof}
Fix $\bfp =(p_1,\bfp_{-1})\in \R^d$ and $r>0$.
Let $I:= [p_1-cr,p_1+cr]$ with $c=c_d\in(0,1)$ to be chosen; so $\delta_1=\osc(f,2 I)$.
Let $\bfq=(q_1,\bfq_{-1})\in \bfT_{\bfp}( {\BBox}(\delta_1,c r))$. Therefore, we have $|p_1-q_1|\leq cr$ and 
\begin{equation*}
\bfL_{f(p_1)}(\bfq_{-1}-\bfp_{-1})
\in [-cr,cr]  \times \left[ -cr\delta_1 , cr\delta_1 \right]\times   \left[-cr\delta_1^{2},cr\delta_1^{2}\right]\times\cdots \times  \left [-cr\delta_1^{d-2},cr\delta_1^{d-2}\right].
\end{equation*}
From this, we get
\begin{equation*}
\bfD_{\delta_1}^{-1}\bfL_{f(p_1)}(\bfq_{-1}-\bfp_{-1})
\in [-cr,cr]  ^{d-1}.
\end{equation*}
Thus, $|\bfD_{\delta_1}^{-1}\bfL_{f(p_1)}(\bfq_{-1}-\bfp_{-1})|\leq cr \sqrt{d-1}$.
Apply \eqref{eq another bounds for d} in Lemma~\ref{lem: distest} with $I $ defined above 
%(x-r/(4d),x+r/(4d))$ 
to obtain, for some  constant $C_d\geq 1$,  
\[
\bd(\bfp,\bfq)\leq C_d\left (2cr+|\bfD_{\delta_1}^{-1}\bfL_{f(p_1)}(\bfq_{-1}-\bfp_{-1})|\right) \leq C_dc\left(2+\sqrt{d-1}\right) r.
\] 
Once $c$ is chosen smaller than $ \frac{1}{C_d(2+\sqrt{d-1})}$, the first containment in the statement follows. 

For the second containment, take any $\bfq\in B_{\bd}(\bfp,r)$, so $|p_1-q_1|\leq\bd (\bfp, \bfq)< r$. We apply \eqref{eq bounds for d} from Lemma~\ref{lem: distest} with $I:=[p_1-r,p_1+r]$, having $p_1,q_1\in I$. We get 
\[
\frac{1}{2}|p_1-q_1|+\frac{1}{2d}|\bfD_{\delta_2}^{-1}\bfL_{f(p_1)}(\bfq_{-1}-\bfp_{-1})|\leq r,
\]
and thus  
\[
\left|\bfD_{\delta_2}^{-1}\bfL_{f(p_1)}(\bfq_{-1}-\bfp_{-1})\right|\leq 2dr.
\]
Consequently,
\begin{align*}
\bfq_{-1}&\in \bfp_{-1}+\bfL_{f(p_1)}^{-1}\left(\bfD_{\delta_2}\left(\prod_{i=1}^{d-1}[-2dr,2dr]\right)\right).
\end{align*}
Since we also had $|p_1-q_1|< r< 2dr$ , we conclude that $\bfq\in \bfT_{\bfp}( {\BBox}(\delta_2, 2dr))$.
\end{proof}

\subsection{Doubling property and LLC}
From this Ball-Box theorem, we can establish the first parts of Theorem~\ref{thm: oscillationversion}. Recall that a metric space $(X,d)$ is \emph{linearly locally contractible} (LLC) if there exists a constant $C>0$ such that every ball $B(x,r)$ is contractible within $B(x, C r)$. $(X,d,\mu)$ is measure doubling if there exists a constant $D\geq 1$ such that  $\mu(B(x,2r))\leq D\mu(B(x,r))$ for every ball $B(x,r)$. Finally, an unbounded space $(X,d,\mu)$ is $Q$-Ahlfors regular if $C^{-1}r^Q \leq \mu(B(x,r))\leq Cr^Q$ for all $x\in X, r>0$.

Both of these follow for functions of doubling oscillation. Recall that a function $f:\R\to \R$ has {\em doubling oscillation}, if there exists a constant $D > 1$ such that 
%for every interval $I\subset \R$ one has
\begin{equation}\label{eq:defdoublosc}
0<\osc(f,2I)\leq D \osc(f,I), \qquad \text{ for every interval } I\subset \R,
\end{equation}
where the oscillation is defined in \eqref{oscillation}. Many functions that are not too flat have this property: e.g. nonconstant polynomials and $f(x)=\sin(x)$. For every such $D$, we say that $f$ is of {\em $D$-doubling oscillation}. 

A stronger property holds for regular functions. A function $f$ of positive oscillation has {\em $\alpha$-regular oscillation}, if there exists a constant $C\geq 1$ such that 
\begin{equation}\label{eq:defregosc0}
\qquad \ \ \ \ \frac{1}{C}  {|I|}^\alpha \leq  {\osc(f,I)}\leq C {|I|} ^\alpha,\qquad \text{ for all intervals }   I\subset \R.
\end{equation}
Equivalently, one has that 
there exists a (possibly different) constant $C\geq 1$ such that 
\begin{equation}\label{eq:defregosc}
\qquad \ \ \ \ \frac{1}{C}\left(\frac{|I|}{|J|}\right)^\alpha \leq \frac{\osc(f,I)}{\osc(f,J)}\leq C\left(\frac{|I|}{|J|}\right)^\alpha,\qquad \text{ for all intervals } I\subset J\subset \R.
\end{equation}
%for every interval $I\subset \R$. 
For every such $C$, we say that $f$ is {\em $(C,\alpha)$-regular}.

\begin{lemma}\label{lem:doublingpart}
    Given $d\in \N$ with $d\geq 3$ and $f:\R\to \R$ continuous, let 
$(\R^d,\bd_f,\lambda)$ be the metric measure space, where $\bd_f$ is defined in \eqref{eq:metricdef} via the vector fields \eqref{eq:vectdef}, and $\lambda$ is the Lebesgue measure.
If $f $ has doubling oscillation, then 
$(\R^d,\bd_f,\lambda)$ is LLC and doubling, and there exists a constant $C\geq 1$ such that  for every $\bfp=(p_1,\bfp_{-1})\in \R^d$ 
\begin{equation}
\label{eq measure expansion}
C^{-1} r^d \rho^{\frac{(d-2)(d-1)}{2}} \leq \lambda(B(\bfp,r))\leq C r^d \rho^{\frac{(d-2)(d-1)}{2}}, \qquad \forall\bfp\in \R^d \ \  \forall r>0
\end{equation}
with $\rho=\osc(f,[p_1-r,p_1+r])$. Moreover, if $f$ has $\alpha$-regular oscillation then $\lambda$ is $Q$-Ahlfors regular with $Q=d+\frac{(d-2)(d-1)}{2}\alpha$.
\end{lemma}

\begin{proof}
    Let $f:\R\to\R$ be a continuous function that has doubling oscillation. Recall that $\det(\bfL_{a})=1$. Therefore, $\bfT_\bfp$ is volume preserving for every $\bfp\in \R^d$, and we get from Theorem~\ref{thm: ballbox}
    \begin{equation}\label{eq:volbounds}
        \delta_1(r)^{\frac{(d-2)(d-1)}{2}} (2c_d r)^d\leq \lambda(B(\bfp,r)) \leq \delta_2(r)^{\frac{(d-2)(d-1)}{2}} (4dr)^d
    \end{equation}
    for $\delta_1(r):=\osc(f,[p_1-2c_dr, p_1+2c_d r])$ and $\delta_2(r):=\osc(f,[p_1-2r, p_1+2r])$. Since $f$ has doubling oscillation, for some constant $c>1$ we have $\delta_2(2r)\leq c\delta_1(r)$. Thus
    \[
    \lambda(B(\bfp,2r))\leq \delta_2(2r)^{\frac{(d-2)(d-1)}{2}} (8dr)^d \leq c^{\frac{(d-2)(d-1)}{2}} \left(\frac{4d}{c_d}\right)^d\lambda(B(\bfp,r)).
    \]
    This gives the doubling property. Similarly, by doubling, since $\delta_1(r)$  and $\delta_2(r)$ are both comparable to $\rho=\osc(f,[p_1-r,p_1+r])$, we get \eqref{eq measure expansion}.

    For the Ahlfors regularity, the argument is similar. First, $(\R^d, \bd_f)$ is unbounded since it is boundedly compact and homeomorphic to $\R^d$ by Proposition~\ref{prop:boundedcompact}. Note that $\alpha$-regular oscillation implies that 
    \begin{align*}
    \osc(f,[p_1-r,p_1+r]) &\leq C \osc(f,[-r-|p_1|-1,r+|p_1|+1])\left(\frac{r}{r+1+|p_1|}\right)^\alpha \\
    &\leq C^2 \osc(f,[-1,+1])r^\alpha,
    \end{align*}
    and
    \begin{align*}
    \osc(f,[p_1-r,p_1+r]) &\geq C^{-1} \osc(f,[-r-|p_1|-1,r+|p_1|+1])\left(\frac{r}{r+1+|p_1|}\right)^\alpha \\
    &\geq C^{-2} \osc(f,[-1,+1])r^\alpha.
    \end{align*}
    This, together with \eqref{eq:volbounds} gives Ahlfors-regularity with $Q=d+\frac{(d-2)(d-1)}{2}\alpha$ since here $\rho$ is comparable to $r^\alpha$.

    Finally, each set $\BBox(\delta,r)$ is convex in $\R^d$ and can be contracted to the point $\bfz$ by homotheties, and this contraction has an image contained in $\BBox(\delta,r)$. Therefore, by Theorem~\ref{thm: ballbox} we can contract $B_{\bd_f}(\bfp,r)$ to $\{\bfp\}$ within the set $\bfT_{\bfp}(\BBox(\delta_2(r),2dr))$. Let $R=\frac{2dr}{c_d}$. Then 
    \[
    \BBox(\delta_2(r),2dr) \subset \BBox(\delta_1(R), c_dR),
    \]
    and by Theorem~\ref{thm: ballbox} 
    \[
    \bfT_{\bfp}(\BBox(\delta_2(r),2dr)) \subset \bfT_{\bfp}(\BBox(\delta_1(R), c_dR)) \subset B_{\bd_f}(\bfp, R)
    \]
    and LLC follows.
\end{proof}

%{\color{red} Do we want to discuss the following: This implies that the filiform structure is doubling/Ahlfors regular if and only if $f$ is such. Consequently, it is not possible to give an Ahlfors regular example with $Q$ larger than the upper bound in Theorem~\ref{thm: mainthm}.} \todo{Comment why heuristically we have a BB theorem? anything more general?}

\subsection{A lemma on Hausdorff measures} 

If $\mu$ is a measure on a metric space $(X,d)$, we say that $\mu$ has \emph{essential Hausdorff dimension} $Q$ if every full-measure set $K\subset X$ has Hausdorff dimension at least $Q$ and there exists a full-measure set $K\subset X$  that has Hausdorff dimension $Q$.

The definition of Hausdorff measure and dimension is the standard one; see, for example \cite{Mat}. We will need the following version of Billingsley's classical lemma. A metric space $(X,d)$ is (metric) \emph{doubling}, if there exists an $N\in \N$ s.t. for every ball $B(x,R)\subset X$ there exist $x_1,\dots, x_N\in X$ with $B(x,R)\subset \bigcup_{i=1}^N B(x_i,R/2)$. Notice that having a measure doubling implies metric doubling by a classical volume comparison argument.  

\begin{lemma}\label{lem: billinglsey} 
Let $Q\in (0,\infty)$. Let $(X,d)$ be a doubling metric space and $\mu$ a non-zero Radon measure.
\begin{enumerate}
\item If $\mu$ is $Q$-Ahlfors regular, then $X$ has Hausdorff dimension $Q$ and $\mu$ has essential Hausdorff dimension $Q$.

\item If for $\mu$-a.e. $\bfp\in X$ the measure satisfies 
\[
\lim_{r\to 0}\frac{\log(\mu(B(\bfp,r)))}{\log(r)}=Q
\]
then $\mu$ has essential Hausdorff dimension $Q$. 
\end{enumerate}
\end{lemma}
\begin{remark}
    In fact, the second claim above shows a stronger statement: Every set $K\subset X$ of positive $\mu$-measure has Hausdorff dimension at least $Q$. 
\end{remark}
The lemma follows by a fairly standard Lusin and covering argument; see, for example,  \cite[Theorem 6.9]{Mat} or \cite[Proposition 2.1]{Y}. The proof in these references is for Euclidean spaces, but the same argument applies for every doubling metric space. %Notice that the assumption in (2) of the previous lemma could also be stated in terms of a simpler expression:
%\[
%Q = \liminf_{r\to 0}\frac{\log\left(\mu(B(\bfx,r))\right)}{\log(r)}= \limsup_{r\to 0}\frac{\log\left(\mu(B(\bfx,r))\right)}{\log(r)},
%\]
%where the above liminf is also called the {\em lower dimension} $d_*(\mu,\bfx)$ of $\mu$ at $\bfx\in X$, while the limsup is called {\em the upper dimension} $d^*(\mu,\bfx)$.

\section{Constructions of functions with doubling or regular oscillation}\label{sec: doublingfunctions}

In this section, we construct the $ \alpha$-regular functions and the functions of doubling and reverse doubling oscillation needed in the proofs of the main theorems.

\begin{lemma}\label{lem: oscillationfunctions} If $\alpha\in(0,1]$, then there exists a function $f:\R\to \R$ with $\alpha$-regular oscillation. If $\alpha>1$, then there exists a monotone $f:\R\to\R$ with doubling oscillation and
\[
\lim_{r\to 0} \frac{\log(\osc(f,(x-r,x+r)))}{\log(r)}=\alpha.
\]
for a.e. $x\in \R$. 
\end{lemma}

The proof will be contained in the next two subsections, where in each we handle the two cases in the lemma that depend on $\alpha>0$. For $\alpha\in (0,1)$, the $\alpha$-regular functions are constructed as Weierstrass functions, and for $\alpha>1$ one uses a singular doubling measure. The $\alpha=1$ case is trivial, and one can use the identity function $f(x)=x$.

\begin{remark}\label{rmk:oscillation} It is clear that an $\alpha$-regular function cannot be monotone unless $\alpha=1$. Further, it is simple to see by differentiation that no $\alpha$-regular function can exist for $\alpha>1$. Thus, Lemma~\ref{lem: oscillationfunctions} is essentially sharp. Moreover, for $\alpha>1$ our constructions of monotone functions $f$ are the best we can do, in that they just barely fall short of being $\alpha$-regular at a.e. $x\in \R$. 
\end{remark} 

\begin{remark}
    There are many constructions of $\alpha$-regular functions, and we present only one. An interesting construction is obtained by constructing functions whose graphs are Bedford-McMullen carpets. This procedure is used in \cite[Example 4.8]{Cdim} to construct a function $f:[0,1]\to[0,1]$, and one can see that their construction produces a function that has $\alpha$-regular oscillation (with the natural modification of the definition for a map defined only on $[0,1]$). We refer to this reference for details of the construction and sketch briefly how regularity follows. One obtains the graph of such $f$ as an attractor of a self-affine iterated function system where the $x$-direction is scaled by $1/m$ and the $y$-direction by $1/\ell$ for some integers $m>\ell\geq 2$. It follows from this, as \cite[Figure 4.2]{Cdim} shows well, that the image of every $m$-adic interval $I$ of length $m^{-k}$ in $[0,1]$ under $f$ is an $\ell$-adic interval of length $\ell^{-k}$. Let $\alpha=\log_{m}(\ell)\in (0,1)$. The lower bound for $\alpha$-regular oscillation is obtained by the fact that every interval in $[0,1]$ contains an $m$-adic interval of comparable length. The upper bound is obtained similarly since every interval is contained in at most two adjacent $m$-adic intervals of comparable length. While this procedure constructs a function on $[0,1]$, a function $f:\R\to \R$ can be obtained from this by taking a subsequential limit of maps $g_r(x)=(f(x_0+rx)-f(x_0))r^{-\alpha}$ as $r\searrow 0$ and $x_0\in(0,1)$ and applying the Arzel\'a-Ascoli theorem.
\end{remark}

\subsection{Weierstrass functions}
    First, we will consider $\alpha\in (0,1)$ and prove the $\alpha$-regular part of Lemma~\ref{lem: oscillationfunctions} by using Weierstrass functions. Let $b\geq 2$ be an integer, and let $a:=b^{-\alpha}\in (1/b,1)$.
    %be such that $\alpha = \log_{b}(\frac{1}{a})$. 
    Consider
\begin{equation}
\label{Weierstrass}
f(x):=\sum_{n\in \Z} a^n \sin(2 \pi b^{n} x),\qquad x\in \R.
\end{equation}
Let us briefly discuss the convergence of this series, since it is more conventional to consider the sum over $n\in \N$ in the definition of the Weierstrass function. For $x\in \R$, we can estimate
\[
\sum_{n\in \Z} a^n |\sin(2 \pi b^{n} x)| \leq \sum_{n\geq 0} a^n + 2\pi |x|\sum_{n<0} (ab)^{n}<\infty,
\]
from which absolute and locally uniform convergence follows. Modulo this detail of changing the sum over $\N$ to that over $\Z$, the following proof is standard. For $b\gg 1$ one could use a simple classical estimate for geometric sums, cf. the argument for \cite[Theorem 7.18]{Rprinciples}. For $b$ small, the slightly more involved argument below is needed and has appeared in \cite[Proposition 2.8]{KMPY}. Here, we give a self-contained argument that works for all integers $b\geq 2$. It is based on a simple estimate for the oscillation of a trigonometric sum, cf. \cite[Proposition 2.5, Lemma 2.6]{KMPY}.

\begin{proof}[Proof of Lemma~\ref{lem: oscillationfunctions} for $\alpha\in (0,1)$]
We first prove an oscillation upper bound for $b$-adic intervals. Fix $m\in \Z$ and consider any interval $I_{k,m}:=[kb^{-m}, (k+1)b^{-m}]$ of length $b^{-m}$ for $k\in \Z$. Fix $N\in \N$, whose precise value will be determined later. We divide the sum for $f$ as in \eqref{Weierstrass} according to frequencies into the low, mid, and high frequency   terms 
\[
f^{m,-}(x):=\sum_{n\leq m-N} a^n \sin(2\pi b^n x), \ \ \ \ f^{m,+}(x):=\sum_{n\geq m+N} a^n \sin(2\pi b^n x)
\]
\[
f^{m,\sim}(x):=\sum_{m-N < n< m+N} a^n \sin(2\pi b^n x).
\]

By estimating the sum for $f^{m,+}$ from above by a geometric sum, we obtain
\[
\osc\left(f^{m,+}, I_{k,m}\right)\leq 2\|f^{m,+}\|_{L^\infty} \leq 2 \frac{a^{m+N}}{1-a}.
\]
By using the mean value theorem on each term of $f^{m,-}$, we get
\begin{align*}
\osc\left(f^{m,-}, I_{k,m}\right)&\leq \sum_{n\leq m-N} \osc(a^n \sin(2\pi b^n x), I_{k,m}) \\
&\leq \sum_{n\leq m-N} 2\pi \left(ab\right)^n b^{-m} \\
&\leq 2\pi\frac{(ab)^{m-N+1}}{ab-1} b^{-m}=2\pi \frac{a^{m-N+1}}{ab-1} b^{1-N}.
\end{align*}
Choose $N\geq 3$ so big that 
\begin{equation}
    \label{eq small oscillation tales}
\osc(f^{m,-}, I_{k,m})\leq a^{m}/8\quad \text{  and } \quad \osc(f^{m,+}, I_{k,m})\leq a^m / 8,\qquad \forall k,m. 
\end{equation}

Regarding $f^{m,\sim}$, counting the $2N-1$ terms and using $\osc(a^n\sin)\leq2a^n$ 
gives $\osc(f^{m,\sim}, I_{k,m}) \leq 2(2N-1)a^{m-N}$.

Consequently,  
\begin{eqnarray*}
\osc\left(f, I_{k,m}\right) &\leq& \osc\left(f^{m,\sim}, I_{k,m}\right) + \osc\left(f^{m,+}, I_{k,m}\right) + \osc\left(f^{m,-}, I_{k,m}\right) \\
&\leq&
2(2N-1)a^{m-N}+a^{m}/8+a^{m}/8\\
&\leq& 4 N a^{m-N}.
\end{eqnarray*}
For each $I\subset \R$ with $|I|>0$ we can choose some $m\in \Z$ so that $|I|\in [b^{-m}, b^{1-m})$. The interval $I$ intersects at most $2b$ intervals $I_{k,m}$, for $k\in \Z$, and thus we get
\[
\osc(f,I)\leq 2b(4N)a^{m-N}\leq 8bNa^{-N} |I|^\alpha ,
\]
since $|I|^\alpha \in [a^{m}, a^{m-1})$. This yields the upper bound \eqref{eq:defregosc0} for the $\alpha$-regularity. 

The lower bound for the $\alpha$-regularity is trickier and requires Fourier analysis. 
Consider an arbitrary interval $I\subset \R$ with midpoint $x$.
For the $N$ chosen above for \eqref{eq small oscillation tales}, pick $m\in \Z$ so that $6b^{-m+N} \leq |I| < 6b^{-m+N+1}$.   Choose $k\in \Z$ so that $x\in I_{k,m}$.
Given these $k,m\in \Z$, there exists a unique $l=l(k,m)$ with $I_{k,m}\subset I_{l ,m-N}$. By orthogonality,
\[
\int_{I_{l,m-N}} \left(f^{m,\sim}(t)-f^{m,\sim}(t+b^{-m}/2)\right)\sin(2 \pi b^m t)\dd t = a^m b^{N-m}.
\]
Thus, there exists some $t\in I_{l,m-N}$ for which $|f^{m,\sim}(t)-f^{m,\sim}(t+b^{-m}/2)|\geq a^m$. The points $t$ and $t+b^{-m}/2$ lie in 
the same or adjacent $b$-adic intervals of length $b^{-m}$.
Therefore, for at least one of these intervals, say $I_{z,m}$, for some $z$, the oscillation is at least half of $a^m$, i.e.,  
\[
\osc(f^{m,\sim}, I_{z,m})\geq a^m/2.
\]
But, for this interval, we get
\[
\osc(f,I_{z,m}) \geq \osc(f^{m,\sim}, I_{z,m})-\osc(f^{m,-}, I_{z,m})-\osc(f^{m,+}, I_{z,m})\geq a^m/4.
\]
The interval $I_{z,m}$ is contained in either $I_{l,m-N}$ or one of its neighbors. In either case, because of the length of $I$, we have
$I_{z,m}\subset   I$.
% We have $d(I_{z,m}, I_{k,m})\leq 2b^{N}b^{-m}$. 
%We have $I_{k,m}\subset B(x,b^{-m})$, and $I_{z,m}\subset (x-b^{-m}-2b^{N-m}, x+b^{-m}+2b^{N-m}) \subset I$. 
In particular, 
\[
\osc(f,I)\geq \osc(f,I_{z,m})\geq a^m/4 = \frac{\left(b^{-m}\right)^{\alpha}}{4}\geq 4^{-1}\left(\frac{|I|}{6 b^{N+1}}\right)^\alpha.
\]
Thus,
\[
\osc(f, I)\geq C_{\alpha}|I|^\alpha,
\]
where $C_{\alpha}:=4^{-1}6^{-\alpha}b^{-(N+1)\alpha}$. This gives the lower bound \eqref{eq:defregosc0} for the oscillation of $f$ and thus the $\alpha$-regularity.
\end{proof}

Note that in the proof $\lim_{\alpha\to 0} C_{\alpha}=0$, since $N\to \infty$ as $\alpha\to 0$ and $b^\alpha=a^{-1}$. Thus, the functions $f$ and their estimates degenerate as $\alpha\to 0$.

\subsection{Doubling measures} 

Let $b\in \N$, $b\geq 2$ and let $\bp = (p_0,\dots, p_{b-1})\in (0,1)^b$ be a probability vector for which $p_0=p_{b-1}$. For $k\in \Z$ and each $q\in b^k \Z$ we expand $q=\pm \sum_{i=k}^N q_i b^{i}$ for some $q_i\in \{0,\dots, b-1\}$ for some $N\in \N$ with $N\geq \max\{0,k\}$. Let $I_q^+:=q+[0,b^{k}]$ and $I_q^-:=q-[0,b^{k}]$. The intervals $I_q^\pm$ are $b$-adic intervals at scales $b^{k}$. 

There exists a unique measure $\mu_\bp$ satisfying 
\begin{equation}\label{def:mp}
\mu_\bp (I_q^{\sigma(q)})=p_0^{-N}\prod_{i=k}^N p_{q_i},
\end{equation} 
where $\sigma(q):=+$ if $q>0$, $\sigma(q):=-$ if $q<0$ and when $q=0$, we interpret this as allowing $\sigma(0)$ to be either sign $\pm$.
By applying Caratheodory's construction and the fact that $\bp$ is a probability vector, $\mu_\bp$ extends to a well-defined measure on $\R$.

 It is direct to see that $\mu_\bp$ is doubling, since adjacent $b$-adic intervals at the same scale have comparable measure. This can be seen as follows. First, if $q=0$, then the adjacent intervals $I^\pm_q$ have the same mass. The cases $q>0$ and $q<0$ are symmetric, so we may restrict to the case $q>0$. The adjacent intervals of $I_q^+$ are $I_{q\pm b^{k}}^{+}$. The base $b$ expansions of $q$ and $q\pm b^k$ differ either in their  $k$'th base-$b$ digit, or in multiple digits in positions $l,\dots, k$ for some $l>k$ due to carrying over. In the latter case, the digits in positions $k,\dots,l-1$ can only change from $b-1$ to $0$ or vice versa. In either case, there is only one factor of $\mu_{\bp}$ that can change, since $p_{b-1}=p_0$, and comparability follows. %\footnote{{\color{red} Is wince a word?}}
 
 The function we consider is
\begin{equation}\label{eq:f}
f(x):=\begin{cases}\mu_{\bp}[0,x] & x\geq 0 \\ -\mu_{\bp}[x,0] & x<0 \end{cases}.
\end{equation}
The measure $\mu_{\bp}$ is non-atomic since it is a doubling measure on $\R$. Thus, $f$ is continuous.

\begin{proof}[Proof of Lemma~\ref{lem: oscillationfunctions} for $\alpha>1$] Choose $b=3$ and fix the probability vector $\bp=(a,1-2a,a)$ for $a\in (0,1/3)$ determined by the equation
\[
\frac{-2\log_3(a)}{3}+\frac{-\log_3(1-2a)}{3}=\alpha.
\] 
Choose $\mu_\bp$ as in \eqref{def:mp} and define $f$ by \eqref{eq:f}.
 We have $\osc(f, I)=\mu_{\bp}(I)$, and since $\mu_{\bp}$ is doubling, we have that $f$ is a monotone function with doubling oscillation.

For $\lambda$-a.e. $x\in \R$, we can expand it uniquely in its $b$-ary expansion as
\[
x=\pm \sum_{i=-\infty}^N q_i b^i
\]
and by the law of large numbers, for a.e. $x\in \R$ we have for $j=0,1,2$ that
\begin{equation}\label{eq:lln}
\lim_{k\to -\infty} \frac{|\{i\in [k,\dots, N] : q_i=j\}|}{N-k+1}=\frac{1}{3}.
\end{equation}
Let $I_{k}(x)$ be a triadic interval of length $3^{k}$ that contains $x$. This interval is unique for a.e. $x\in \R$. We get
\[
\mu_{\bp}(I_k(x))=p_0^{-N} \prod_{i=k}^N p_{q_i}=a^{-N+|\{i\in [k,\dots, N] : q_i=0\}|+|\{i\in [k,\dots, N] : q_i=2\}|}(1-2a)^{|\{i\in [k,\dots, N] : q_i=1\}|}.
\]
Thus, we get
\[
\mu_{\bp}(I_k(x))=3^{-N\log_3(a)+(|\{i\in [k,\dots, N] : q_i=0\}|+|\{i\in [k,\dots, N] : q_i=2\}|)\log_3(a)+|\{i\in [k,\dots, N] : q_i=1\}|\log_3(1-2a)}.
\]
Let $r>0$ and choose $k\in \Z$ such that  $3^{k-1}\leq r\leq 3^{k}$. Since $\mu_{\bp}$ is doubling, we have that $\mu_\bp(B(x,r))$ is comparable to $\mu_\bp(I_k(x))$. This, together with \eqref{eq:lln} gives for a.e. $x\in \R$
\[
\lim_{r\to 0}\frac{\log(\mu_\bp(B(x,r)))}{\log(r)}=\lim_{k\to -\infty}\frac{\log(\mu_\bp(I_k(x)))}{k\log(3)} = \alpha.
\]
\end{proof}

\section{The collection of spaces and Gromov-Hausdorff limits}\label{sec: GH}

%The metric spaces   that we will  consider are of the form  $(M, \bd)$ with $M:=\R^d$, for some $d\in\N$ with $d\geq 3$ and $\bd:=\bd_f $, for some function $f $ as in \eqref{eq:metricdef}.  
%We will denote the Lebesgue measure on $\R^d$ by $\lambda$. Where it is important to use a constant multiple of the Lebesgue measure, we instead use $\mu$.

 Let $C,\alpha,D>0$. We denote by $\cF^{\rm reg}_{C,\alpha}$ the collection of $(C,\alpha)$-regular functions in the sense of \eqref{eq:defregosc}. 
We denote by $\cF^{\rm doubl}_{D,\epsilon, \delta}$ the collection of functions with doubling oscillation with constant $D$, in the sense of  \eqref{eq:doublingf}, and reverse doubling oscillation with constants $(\epsilon,\delta)\in (0,1)^2$ in the sense of 
\eqref{eq: def reverse doubling}. Notice that for every $C,\alpha>0$ there exists a constant $D>1$ and $(\epsilon,\delta)\in (0,1)^2$ such that  $\cF^{\rm reg}_{C,\alpha}\subset \cF^{\rm doubl}_{D,\epsilon, \delta}$. These families enjoy compactness properties given by the following lemma.

\begin{lemma}\label{lem:equicont} Let $D>1, (\epsilon,\delta)\in (0,1)^2$. If $\mathcal{F} \subset \cF^{\rm doubl}_{D,\epsilon, \delta}$ is a collection of functions with $f(0)=0$ and $ \osc(f,[-1,1])=1$, for every $f\in \mathcal{F}$, then $\mathcal{F}$ is pre-compact in the topology of locally uniform convergence.
\end{lemma}
\begin{proof}
    By the Arzelà–Ascoli theorem, it suffices to show that the family $\mathcal{F}$ is locally uniformly equicontinuous and equibounded. Fix $R\geq 1$. By repeatedly using that  oscillation is doubling, we have $\osc(f,[-4R,4R]) \leq M:=\osc(f,[-1,1])\cdot D^{\log_2(R)+3}$ for all $f\in \mathcal{F}.$ This implies local equiboundedness since $f(0)=0$ for all $f\in \mathcal{F}$.

    For the equicontinuity fix $\overline{\epsilon}>0$.  First, choose a $k_0\in \N$, such that  $M\epsilon^{k_0}\leq \overline{\epsilon}$. Let $\overline{\delta}=R\delta^{k_0}$. Now, let $x,y\in [-R,R]$ with $|x-y|\leq \overline{\delta}$ and $n\in \N$ be arbitrary. For the segment $J=[x,y]$ we have by the reverse doubling property:
    \[
    \osc(f,J)\leq \epsilon^{k_0} \osc(f, \delta^{-k_0} J) \leq \epsilon^{k_0} \osc(f,[-4R,4R])\leq \overline{\epsilon}.
    \]
    Here, we used the fact that $\delta^{-k_0} J$ has length at most $R$, and thus is contained in the interval $[-4R,4R]$. Since $\overline{\epsilon},R$ were arbitrary, equicontinuity follows.
\end{proof}
%The restriction to monotone functions is done in order to ensure that every such an $f$ is locally uniformly equicontinuous, once scaled to satisfy $f(0)=0,f(1)=1$. \footnote{{\color{red} Enrico doesn't get it...}} See also Remark~\ref{rmk:closednessclass}. This will be relevant in Section~\ref{sec: GH}.  CLEAN THIS UP.}
Next, we move to study the collection of pointed spaces associated to such functions:
\[\cM^{\rm reg}_{C,\alpha}:=\{(\R^d,\bd_f,c\lambda, \bfp): f \in \cF^{\rm reg}_{C,\alpha}, c=\lambda(B_{\bd_f}(\bfp,1))^{-1}\}
\]
and 
\[
\cM^{\rm doubl}_{D,\epsilon, \delta}:=\{(\R^d,\bd_f,c\lambda, \bfp): f\in \cF^{\rm doubl}_{D,\epsilon, \delta}, c=\lambda(B_{\bd_f}(\bfp,1))^{-1}\},
\]
where $\bd_f$ is  the distance defined in \eqref{eq:metricdef}  via the vector fields \eqref{eq:vectdef}, and $\lambda$ is the Lebesgue measure. Our goal is to show that these collections of metric measure spaces are closed with respect to pointed measured Gromov-Hausdorff convergence, which we now define.

Let $A, N\subset X$ be two subsets of a metric space $(X,d)$. We say that $N$ is an {\rm $\epsilon$-net} for $A$ if $A\subset \bigcup_{n\in N} B(n,\epsilon)$. Let $C>0,L\geq 1$. A mapping $f:X\to Y$ between two metric spaces $(X,d_1), (Y,d_2)$ is called a $(L,C)$-{\em quasi-isometric embedding} if 
\[
\frac{1}{L} d_1(x,y) - C \leq d_2(f(x),f(y))\leq Ld_1(x,y) + C, \qquad \forall x,y\in X.
\]

%\begin{remark}\label{rmk:closednessclass}
% The collection $\cF^{\rm reg}_{C,\alpha}$ is closed under uniform convergence, but the class of functions of doubling oscillation is not. For this reason, we impose the condition of being monotone for $\cF^{\rm doubl}_D$. Indeed, the each of the functions $f_n(t)=\frac{1}{n}\sin(nt)$, for $n\in \N$ has doubling oscillation for some uniform $D$, and the resulting sequence converges uniformly to a constant function, which fails to have positive oscillation. This is the reason we add monotonicity. 
%\end{remark}

We say that a sequence of pointed metric spaces
\[
(X_n,\bd_n,\bfp_n)
\]
\emph{pointed Gromov-Hausdorff converges} to a pointed metric space
\[
(X,\bd,\bfp)
\]
if there exist mappings $\phi_n:X_n\to X$ such that for all $R>0$ and $\delta>0$ there exists a sequence $\epsilon_n \searrow 0$ such that  
\begin{enumerate}
 \item $\phi_n(\bfp_n)=\bfp$;
 \item $\phi_n|_{B(\bfp_n, R)}$ are $(1,\epsilon_n)$-quasi-isometric embeddings;
 \item $\phi_n(B(\bfp_n, R))$ is an $\epsilon_n$-net for $B(\bfp, R-\delta)$.
\end{enumerate}
We call such a sequence of functions $\phi_n$ a {\rm Hausdorff approximating sequence}.

We say that a sequence of pointed metric measure spaces 
$(X_n,\bd_n,\mu_n, \bfp_n)$ \emph{pointed measured Gromov-Hausdorff converges} to $(X,\bd,\mu,\bfp)$ if it pointed Gromov-Hausdorff converges and additionally there is some measurable Hausdorff approximating sequence $\phi_n$ for which the push-forward measures $(\phi_n)_{*}(\mu_n)$ converge vaguely to  $  \mu$, in the sense that, 
%Recall, that the final condition means that 
for all compactly supported continuous functions $u:X\to \R$, we have
\[
\lim_{n\to\infty} \int u \circ \phi_n  \dd\mu_n=\int u \dd\mu.
\]
%\footnote{{\color{red} should we say that $f\in \cF^{\rm doubl}_{D}$ is part of the conclusion in the next result?}}

The following key lemma gives Gromov-Hausdorff convergence from uniform convergence of the relevant functions.

\begin{lemma}\label{lem: convergence} Let $D>1, \epsilon,\delta\in (0,1)$.  Let $(c_n)_{n\in \N}$ be a sequence of positive scalars with $c_n\to c>0$.
 Assume that $(f_n)_{n\in \N}$ is a sequence of continuous functions with $f_n \in \cF^{\rm doubl}_{D,\epsilon, \delta}$ for every $n\in \N$, that converges locally uniformly to a non-constant function $f$. Then $f\in \cF^{\rm doubl}_{D,\epsilon, \delta}$ and the sequence
 \[
 (\R^d,\bd_{f_n},c_n\lambda, \bfz)
 \]
 pointed measured Gromov-Hausdorff converges to 
 \[
 (\R^d,\bd_{f},c\lambda,\bfz).
 \]
\end{lemma}
\begin{proof} 
We will argue that we can choose the Hausdorff approximating sequence $\phi_n={\rm id}$, and this will give the pointed measured Gromov-Hausdorff convergence. For this sequence of maps, the vague convergence of measures is trivial, and all we need to prove is the pointed Gromov-Hausdorff convergence with this particular choice of approximating sequence. This can be reduced with \cite[Proposition 12.1.3]{LeDonneGTM} to local uniform convergence of the distance functions $\bd_{f_n} \to \bd_{f} $. Next, we verify two uniform estimates that are needed to apply this proposition.

First, the sequence $(f_n)_{n\in \N}$ converges locally uniformly to $f$ and thus $f$ is continuous. Local uniform convergence implies that for every interval $J\subset \R$ we have $\lim_{n\to \infty} \osc(f_n,J)=\osc(f,J)$. Since the doubling and reverse doubling inequalities are stable under limits, this implies $f\in \cF^{\rm doubl}_{D,\epsilon, \delta}$. Since $f$ is not constant, there is some interval $I$ such that  $\osc(f, I)>0$. By applying doubling repeatedly, we get that $\osc(f, J)>0$ for every interval $J\subset \R$ and $f$ has positive oscillation. Thus, $\bd_f$ induces the Euclidean topology on $\R^d$ by Lemma~\ref{lem: characterization}.

Next, we argue that the family of metrics $\bd_{f_n}$ is locally equicontinuous in the following sense. For every $R>0$ and $\epsilon>0$ there exists a $\delta>0$ such that  $|\bfp-\bfq|\leq \delta$ and $|\bfp|, |\bfq|\leq R$ imply $\bd_{f_n}(\bfp, \bfq)\leq \epsilon$ for all $n\in \N$. Fix $\epsilon>0, R>0$. By local uniform convergence, there is an upper bound $M\geq 1$ such that  $\sup_{|x|\leq 2R}|f_n(x)| \leq M$ for all $n\in \N$. Similarly, there is a lower bound $\omega>0$ such that  $\osc(f_n,[-2R,2R]) \geq \omega$ for all $n\in \N$. Let $c_d\in (0,1)$ be the constant from Theorem~\ref{thm: ballbox}. By iterating the doubling bound, we get an $\eta\in (0,1)$ such that  for all $p_1\in [-R,R]$ we have $\osc(f_n,[p_1-2c_d\epsilon,p_1+2c_d\epsilon]) \geq \eta$ for all $n\in \N$.

 By the Ball-Box Theorem, Theorem~\ref{thm: ballbox}, we have 
\[
\bfT_{\bfp}^{f_n}(\BBox(\eta,c_d\epsilon))\subset B_{\bd_{f_n}}(\bfp,\epsilon).
\]
The matrix $\bfL_{f_n(p_1)}$ has entries bounded in absolute value by $(dM)^d$, and thus its operator norm is bounded by $d(dM)^{d}$. Thus, if $\delta=\frac{1}{2d(dM)^{d}}\min(\eta,c_d\epsilon\eta^{d-2})$ and $|\bfq-\bfp|\leq \delta$, then 
\[
\left(\bfT_{\bfp}^{f_n}\right)^{-1}(\bfq)=(q_1-p_1,\bfL_{f_n(p_1)}(\bfq_{-1}-\bfp_{-1}))\in \BBox(\eta,c_d\epsilon)
\]
and $\bfq \in \bfT_{\bfp}^{f_n}(\BBox(\eta,c_d\epsilon))$. Therefore, $\bfq \in B_{\bd_{f_n}}(\bfp,\epsilon)$ and we get $\bd_{f_n}(\bfq,\bfp)\leq \epsilon$. This yields the desired equicontinuity.

Next, we establish a uniform boundedness principle. Let $R>0$. As before, there exists an $M\geq 1$ such that  $\osc(f_n,[-2R,2R]) \leq M$ and $|f_n(0)|\leq M$ for all $n\in \N$.  By the Ball-Box Theorem, Theorem~\ref{thm: ballbox},  we have
\[
B_{\bd_{f_n}}(0,R)\subset \bfT_{\bfz}^{f_n}(\BBox(M,2dR)).
\]
Since the operator norm of $\bfL^{-1}_{f_n(0)}$ is bounded by $d(dM)^{d}$, we have 
\[
\bfT_{\bfz}^{f_n}(\BBox(M,2dR)) \subset [-2dR,2dR] \times \bfL^{-1}_{f_n(0)}(B(0,2d^{3/2}R M^{d-2})) \subset B(0,4 M^{2d-2} d^{d+5/2} R).
\]
This yields $\bigcup_{n\in \N} B_{\bd_{f_n}}(0,R)\subset B(0,4 M^{2d-2} d^{d+5/2} R)$. By Proposition~\ref{prop:boundedcompact}, we have 
\[
\diam_{\bd_f}\left(\bigcup_{n\in \N} B_{\bd_{f_n}}(0,R)\right)<\infty \ \ \forall R>0.
\]
This verifies the main assumption of \cite[Proposition 12.1.3]{LeDonneGTM}, and what remains is to prove that $\bd_{f_n}(\bfp,\bfq)\to \bd_f(\bfp,\bfq)$ uniformly when $\bfp,\bfq$ lie in an arbitrary bounded subset. By local uniform equicontinuity of the metrics, it suffices to prove this pointwise. Thus, fix $\bfp,\bfq\in \R^d$.

Denote by $\bfX, \bfY_{f_n}, \bfY_f$ the vector fields as defined in \eqref{eq:vectdef}, with the added subscript indicating the function used. Let $\Delta_n$ ($\Delta$) be the horizontal bundle spanned by $\bfX$ and $\bfY_{f_n}$ ($\bfX$ and $\bfY_f$), and let $\|\cdot\|_n$ ($\|\cdot\|$) be a continuous norm on the tangent space of $\R^d$ whose restriction to $\Delta_n$ ($\Delta$) corresponds to the inner product norm induced by declaring $\bfX, \bfY_{f_n}$ ($\bfX, \bfY_{f}$) orthonormal. We may assume $\|\cdot\|_n\to \|\cdot\|$.

Our objective is to show for all $\bfp,\bfq\in \R^d$ that
 \begin{equation}\label{lim dist}
      \lim_{n\to\infty} \bd_{f_n}(\bfp,\bfq)=\bd_f(\bfp,\bfq).
  \end{equation}
We will reduce this to two inequalities. First, pass to a subsequence realizing the limit inferior $\liminf_{n\to \infty} \bd_{f_n}(\bfp,\bfq)$. Then, by Proposition~\ref{prop:boundedcompact}, we may choose geodesics $\bgamma_n$ connecting $\bfp$ to $\bfq$ for $\bd_{f_n}$ such that  $\len_{\|\cdot\|_n}(\bgamma_n)=\bd_{f_n}(\bfp,\bfq)$. Since the metrics are uniformly bounded, these curves have uniformly bounded length and can be parameterized by uniformly bounded speed. These curves have also uniformly bounded Euclidean length, since $f_n$ are uniformly bounded on bounded sets, and thus we may extract a subsequence that converges uniformly in $\R^d$ to a curve $\bgamma$. By Lemma~\ref{lem:limcurvehorizontal}, we have
\[
\bd_{f}(\bfp,\bfq) \leq \len_{\|\cdot\|}(\bgamma) \leq \liminf_{n\to \infty}\len_{\|\cdot\|_n}(\bgamma_n)=\liminf_{n\to \infty} \bd_{f_n}(\bfp,\bfq).
\]
This gives one of the inequalities for \eqref{lim dist}, and we just need an opposite one for the limit superior.

By Proposition~\ref{prop:boundedcompact}, there exists a geodesic $\bgamma$ with $\bgamma(0)=\bfp$ to $\bgamma(1)=\bfq$ for $\bd_{f}$ such that  $\len_{\|\cdot\|}(\bgamma)=\bd_{f}(\bfp,\bfq)$. By Theorem~\ref{d_X rectifiability}, the curve $\bgamma$ is horizontal and thus for a.e. $t$ in its domain $\dot\bgamma = a(t) \bfX+b(t)\bfY_{f}$. By Remark~\ref{rmk endpoint}, we can find $\Delta_n$-horizontal curves $\bgamma_n$ with $\dot\bgamma_n = a(t) \bfX+b(t)\bfY_{f_n}$ that start at $\bgamma_n(0)=\bfp$ and $\lim_{n\to \infty} |\bgamma(1)-\bgamma_n(1)|=0$. By equicontinuity of the metrics, we also get
\[
\lim_{n\to \infty} \sup_{m\in \N} \bd_{f_m}(\bfq,\bgamma_n(1))=0.
\]
Thus,
\begin{align*}
\limsup_{n\to \infty} \bd_{f_n}(\bfp,\bfq)&\leq \limsup_{n\to \infty} \left(\bd_{f_n}(\bfp,\bgamma_n(1))+\sup_{m\in \N} \bd_{f_m}(\bfq,\bgamma_n(1))\right) \\
&\leq \limsup_{n\to \infty} \len_{\|\cdot\|_n}(\bgamma_n) \\
&=\limsup_{n\to \infty} \len_{\|\cdot\|}(\bgamma)  = \bd(\bfp,\bfq).
\end{align*}

\end{proof}

When extracting convergent subsequences of a sequence of spaces, we need the following small lemma to rescale the metric and measure on the space. We say that two pointed metric measure spaces $(X,d_1,\mu,\bfp), (Y, d_2, \nu, \bfq)$ are {\em isomorphic} if there is an isometry $F: X\to Y$ that preserves the measure with $F(\bfp)=\bfq$.

\begin{lemma}\label{lem: rescaling} Let $D>1, r>0, \bfp\in \R^d$ and $\epsilon,\delta \in (0,1)$. If $f\in \cF^{\rm doubl}_{D,\epsilon,\delta}$, then for all $c>0$ the pointed metric measure space
\[
(\R^d,\bd_{f}/r,c\lambda,\bfp)
\]
is isomorphic to
\[
(\R^d,\bd_{g},c'\lambda,0)
\] 
for some $c'>0$, and some $g\in \cF^{\rm doubl}_{D,\epsilon,\delta}$ with $g(0)=0$ and $\osc(g,[-1,1])=1$. In fact, one can choose $g(x):=\osc(f,[p_1-r,p_1+r])^{-1}(f(p_1+rx)-f(p_1))$. If, moreover, $f\in \cF^{\rm reg}_{C,\alpha}$, then this choice gives $g\in \cF^{\rm reg}_{C,\alpha}$.
\end{lemma}
\begin{proof}
 Let $\rho:=\osc(f,[p_1-r,p_1+r])^{-1}$ and $g$ be as in the statement.  Define 
 \[
F((x_1,\bfx_{-1})):=((x_1-p_1)/r,\bfD_{\rho}\bfL_{f(p_1)}(\bfx_{-1}-\bfp_{-1})/r).
 \] 
 Then a direct calculation gives $DF(\bfx) \bfX_{f}(\bfx)=\bfX(F(\bfx))/r$ and $DF(\bfx) \bfY_{f}(\bfx)=\bfY_{g}(F(\bfx))/r$. Thus, $F$ maps horizontal curves with respect to $\Delta_f$ to a horizontal curve with respect to $\Delta_g$, and $F^{-1}$ vice versa. The map $DF$ scales horizontal vectors by factor $1/r$, and thus is an isometry between $\bd_f/r$ and $\bd_g$. Next, $DF$ has constant Jacobian, and thus the claim follows by choosing $c'=c/|\det(DF)|$.

 The classes $\cF^{\rm doubl}_{D,\epsilon,\delta}$ and $\cF^{\rm reg}_{C,\alpha}$ are invariant under scaling and thus the claims on $g\in \cF^{\rm doubl}_{D,\epsilon,\delta}$ and $g\in \cF^{\rm reg}_{C,\alpha}$ follow.
\end{proof}

Let $\mathcal{M}$ be a family of pointed metric measure spaces $(X,\bd,\mu, \bfp)$. We say that $\mathcal{M}$ is {\em compact with respect to pointed measured Gromov-Hausdorff convergence} if every sequence of spaces has a subsequence that converges in the pointed measured Gromov-Hausdorff sense to an element of $\mathcal{M}$.

%A pointed metric measure space $(X,d,\mu,\bfp)$ is \emph{normalized} if $\mu(B(\bfp,1))=1$. 
A family $\mathcal{M}$ of pointed metric measure spaces is \emph{invariant under scalings and changes of basepoints}, if for every $(X,d,\mu,\bfp)\in \mathcal{M}$, $\bfq\in X$ and $r>0$ we have $(X,d/r, \mu/\mu(B(\bfq,r)), \bfq)\in \mathcal{M}$.

\begin{theorem}\label{thm: closed} For every $D>1, C\geq 1 $, $\alpha\in (0,1]$, and $\epsilon,\delta \in (0,1)$, the families
$\cM^{\rm reg}_{C,\alpha}$ and $\cM^{\rm doubl}_{D, \epsilon, \delta}$ are invariant under scalings and changes of basepoints, and compact with respect to pointed measured Gromov-Hausdorff convergence.
\end{theorem}
\begin{proof}
 We focus on $\cM^{\rm doubl}_{D,\epsilon,\delta}$ since the case of regular functions is identical. 

 Next, we argue invariance under scalings and translations. Let $(\R^d,\bd_{f},c\lambda, \bfp)\in \cM^{\rm doubl}_{D,\epsilon,\delta}$  and $\bfq\in \R^d$ and $r>0$. By Lemma~\ref{lem: rescaling}, we have  that the space $(\R^d,\bd_{f}/r,\lambda/ \lambda(B_{\bd_f}(\bfq,r)), \bfq)$ is isomorphic to $(\R^d,\bd_{g},d'\lambda, \bfz)$ for some $g\in \mathcal{F}^{\rm doubl}_{D,\epsilon,\delta}$. The isomorphism preserves the volume of the unit ball, and thus $(\R^d,\bd_{f}/r,\lambda/ \lambda(B(\bfq,r)), \bfq) \sim (\R^d,\bd_{g},d'\lambda, \bfz) \in \cM^{\rm doubl}_{D,\epsilon,\delta}$. 
 
 Next, we argue convergence. Let 
 $(\R^d,\bd_{f_n},c_n\lambda, \bfp_n)\in \cM^{\rm doubl}_{D, \epsilon,\delta}$, with $c_n=\lambda(B_{\bd_{f_n}}(\bfp_n, 1))^{-1}$, and $\bfp_n\in \R^d$. By Lemma~\ref{lem: rescaling}, every sequence of spaces 
 \[
 (\R^d,\bd_{f_n},c_n\lambda,\bfp_n)
 \]
 is isomorphic to
 \[
 (\R^d, \bd_{g_n}, d_n\lambda, \bfz),
 \]
 for some $d_n>0$ and some $g_n\in \cF^{\rm doubl}_{D,\epsilon,\delta}$ with $g_n(0)=0, \osc(g_n,[-1,1])=1$. Moreover, we have that $d_n=\lambda(B_{\bd_{g_n}}(\bfz,1))^{-1}$, since the isomorphism preserves the volume of the unit ball. By Lemma \ref{lem:doublingpart}, the constants $d_n>0$ are bounded uniformly from above and below by positive and finite quantities. Thus, after passing to a subsequence, we infer that $d_n\to d$.

 By Lemma~\ref{lem:equicont}, the sequence $g_n:\R\to \R$ is precompact in the topology of locally uniform convergence. Up to passing to a subsequence, we may thus assume that $g_n$ converges locally uniformly to a continuous function $g\in \cF^{\rm doubl}_{D,\epsilon,\delta}$ with $g(0)=0, \osc(g,[-1,1])=1$.  Now, the pointed measured Gromov-Hausdorff convergence $(\R^d, \bd_{g_n}, d_n\lambda, \bfz)\to (\R^d, \bd_{g}, d\lambda, \bfz)$ follows from Lemma \ref{lem: convergence}. Finally, the correct normalization $d=\lambda(B_{\bd_{g}}(\bfz,1))^{-1}$ follows since from vague convergence together with geodesicity implies $\lambda(B_{\bd_{g_n}}(\bfz,1))\to \lambda(B_{\bd_{g}}(\bfz,1))$. Here, we use the standard fact \cite[Corollary 2.2]{Buckley} that since $(\R^d, \bd_{g}, d\lambda, \bfz)$ is geodesic and doubling $\lambda(\partial B_{\bd_{g}}(\bfz,1))=0$.
\end{proof}

The previous theorem implies, in particular, that if $(\R^d,\bd_f,c\lambda,\bfp)\in \cM^{\rm reg}_{C,\alpha}$, and if $r_n\searrow 0, \bfp_n \in \R^d$, then $(\R^d,\bd_f/r_n,\frac{\lambda}{\lambda(B(\bfp_n, r_n))},\bfp_n)\in \cM^{\rm reg}_{C,\alpha}$ and a subsequence of these pointed measured Gromov-Hausdorff converges. Any such subsequential limit, i.e., any weak tangent, is thus also an element of $\cM^{\rm reg}_{C,\alpha}$. 

\section{Poincar\'e inequality}\label{sec: PI}

%We say that $(X,\bd,\mu)$ satisfies a $(1,1)$-Poincar\'e inequality if there exists a constant $C>0$ so that for all Lipschitz functions $u:X\to \R$ for every $B=B(x,r)\subset X$ we have
%\[
%\int |u-u_B| d\mu \leq Cr \int \lip[u] d\mu.
%\]
%Here, $u_B=\frac{1}{\mu(B)}\int_B u d\mu$.

\subsection{Controllability}

%To prove the Poincar\'e inequality for $(\R^d,\bd_f,\mu)$, when $f$ has doubling oscillation, 
We adapt some terminology from the work of Lanconelli and Morbidelli \cite{L}. 
Let $\Omega\subset \R^d$ be open and let $Q\subset \R^d$ be a neighborhood of the origin. We say that a {continuous} function $E:\Omega\times Q\to \R^d$ is an {\em almost exponential map of type $a>0$} if
\begin{enumerate}
\item $E(x,0)=x$ for every $x\in \Omega$.
\item $E(x,\cdot):Q\to \R^d$ is $C^1$ and one-to-one on $Q$, and $D(x,\bfh):=|\det(D_\bfh E(x,\bfh))|$ satisfies
\[
0<\frac{1}{a}D(x,0)\leq D(x,\bfh)\leq a D(x,0), \qquad\forall x\in \Omega, \forall \bfh\in Q.
\]
\end{enumerate}
Given a continuous Carnot-Carath\'eodory structure $(\R^d, \Delta, \|\cdot \|)$, we say that $E$ is {\em controllable with hitting time} $T>0$ if the following two conditions hold:
\begin{enumerate}
\item There exists a continuous function $\gamma:\Omega\times Q\times [0,T]\to \R^d$ such that for $(x,\bfh)\in \Omega\times Q$ the curve $t\in [0,T] \mapsto \gamma(x,\bfh,t)$ is a horizontal path, contains $x$ and $E(x,\bfh)$, and 
$\|\dot\gamma(x,\bfh,t)\|\leq 1$ for a.e. $t\in [0,T]$.
\item There is a constant $b>0$, such that for all $(\bfh,t)\in Q \times [0,T]$ the map $$\Phi_{\bfh,t}: x\in \Omega\longmapsto \gamma(x,\bfh,t)$$ is one-to-one, continuous, and 
\begin{equation}\label{eq:jacobianphi}
b^{-1}\lambda|_{\Phi_{\bfh,t}(\Omega)} \geq (\Phi_{\bfh,t})_*(\lambda|_{\Omega}), 
\end{equation}
where $\lambda$ keeps denoting the Lebesgue measure. 
\end{enumerate}
In the context of this article, the considered mapping $\Phi$ will not be $C^1$ in the parameter $x\in \Omega$. For this reason, we need to state the final condition on controllability in terms of the push-forward measure. In \cite{L}, this condition is instead expressed using the Jacobian.

By the same proof as in \cite{L}, except using Lipschitz functions instead of $C^1$ functions and the push-forward measure instead of the Jacobian at one point, we obtain the following theorem. For each function $u$, we use the notation 
\[
\Lip_d[u](x):=\limsup_{y\to x} \frac{|u(x)-u(y)|}{d(x,y)}.
\]

\begin{theorem}[after Lanconelli-Morbidelli, {\cite[Theorem 2.1]{L}}]\label{thm:lanconelli}
%\label{thm: PIlanconelli}{\color{blue} Do we need here $\bX,\bfY$ exactly as before? See below lemma. I would imagine yes, but if Lemma is true without, then no.}
Let $(\R^d, \Delta, \|\cdot\|)$ be a continuous Carnot-Carath\'eodory structure
%$(\R^d,\bd_f)$ be the metric space where   
%$\bd_f$ be the distance on $\R^d$ defined in \eqref{eq:metricdef} via   \eqref{eq:vectdef}, for some $f$ with positive oscillation, 
and let $\lambda$ be the Lebesgue measure.
Let $B:=B(\bfp,r)\subset \R^d$ be a ball with respect to $d_{\Delta,\|\cdot \|}$. 
Let $\Omega\subset B$ be open and bounded, and let $Q\subset \R^d$ be a bounded neighborhood of ${\bf 0}\in \R^d$. 
Assume that there exists an almost exponential map $E:\Omega\times Q\to \R^d$ and two constants $\alpha,\beta>0$ for which we have the following:
\begin{enumerate}
\item $\lambda(B)\leq \alpha \lambda(\Omega)$ and $B\subset E(x,Q)$ for every $x\in \Omega$,
\item $E$ is controllable with hitting time $T\leq \alpha r$, and
\item $\lambda((\alpha+1)B)\leq \beta \lambda(B)$.
\end{enumerate}
Then, there exists some $c>0$, which depends only on $\alpha,\beta$ and the constants $a,b$ above, such that for each $d_{\Delta,\|\cdot \|}$-Lipschitz function $u: (\alpha+1)B \to \R$  one has
\begin{equation}\label{eq:weakPI}
\int_B\int_{B} |u(\bfx)-u(\bfy)|\dd\lambda \dd\lambda \leq cr \lambda(B) \int_{(\alpha+1)B} \Lip_{d_{\Delta,\|\cdot \|}}[u] \dd\lambda.
\end{equation}
\end{theorem}

\begin{proof}%[Proof of Theorem~\ref{thm: PIlanconelli}]
Fix $B, u, E, \Omega$ as in the statement. The claim is obvious if $\lambda(B)=0$, and thus we assume that $\lambda(B)>0$. By controllability, it is easy to see that $E(x, Q)\subset (\alpha+1)B$. From this, we get
\begin{align*}
   \lambda(Q) D(\bfx,0)\leq a \int_Q D(\bfx,h) d\lambda \leq a \lambda((\alpha+1)B)\leq a\beta \lambda(B).
\end{align*}
Thus, 
\begin{equation}\label{eq:jacobest}
D(\bfx,0)\leq \frac{a\beta \lambda(B)}{\lambda(Q)}.
\end{equation}
First, by using a change of variables and the estimate for $D(\bfx, \bfh)$, we get:
\begin{align*}
    \int\int_{B\times B} |u(\bfx)-u(\bfy)|\dd\lambda_{\bfx} \dd\lambda_{\bfy} &\leq 2 \int\int_{B\times B} |u(\bfx)-u_\Omega|\dd\lambda_{\bfx} \dd\lambda_{\bfy} \\
    & \leq \frac{2\lambda(B)}{\lambda(\Omega)}\int\int_{\Omega\times B} |u(\bfx)-u(\bfy)|\dd\lambda_{\bfx} \dd\lambda_{\bfy} \\
    & \leq \frac{2\lambda(B)}{\lambda(\Omega)}\int_{\Omega} \int_{Q} |u(E(\bfx,0))-u(E(\bfx,\bfh))|D(\bfx,\bfh)\dd\lambda_{\bfh} \dd\lambda_{\bfx}
    \end{align*}
Next, we drop the subscript and write $\Lip[u]$ to streamline the calculations below.
The function $t\to u(\gamma(\bfx,\bfh,t))$ is Lipschitz by assumption and since $\|\dot\gamma(\bfx,\bfh,s)\|\leq 1$ we get
\[
|u(E(\bfx,\bfh))-u(E(\bfx,0))| \leq \int_{0}^T |u(\gamma(\bfx,\bfh,\cdot))'(s)| \dd s \leq \int_{0}^T \Lip[u](\gamma(\bfx,\bfh,s)) \dd s.
\]
    Substituting this into the previous integral and using Fubini yields
    \begin{align*}
    \int\int_{B\times B} |u(\bfx)-u(\bfy)|\dd\lambda_{\bfx} \dd\lambda_{\bfy} &\stackrel{\rm Ass. \ (1)}{\leq} 2a \alpha\int_{\Omega} \int_{Q} \int_{0}^T \Lip[u](\gamma(\bfx,\bfh,s)) \dd s \dd\lambda_\bfh D(\bfx,0)\dd\lambda_{\bfx} \\
    &\stackrel{\eqref{eq:jacobest}}{\leq} \frac{2 a^2\alpha \beta \lambda(B)}{\lambda(Q)}\int_{\Omega} \int_{Q} \int_{0}^T \Lip[u](\gamma(\bfx,\bfh,s)) \dd s \dd\lambda_\bfh \dd\lambda_{\bfx} \\
    & = \frac{2 a^2\alpha \beta \lambda(B)}{\lambda(Q)} \int_{Q} \int_{0}^T  \int_{\Omega} \Lip[u](\gamma(\bfx,\bfh,s)) \dd\lambda_{\bfx} \dd s \dd\lambda_\bfh \\
     & \stackrel{\eqref{eq:jacobianphi}}{\leq} \frac{2 a^2\alpha \beta \lambda(B)}{b\lambda(Q)} \int_{Q} \int_{0}^T  \int_{(\alpha+1)B} \Lip[u](\bfx) \dd\lambda_{\bfx} \dd s \dd\lambda_\bfh \\
     & \leq 2 a^2\alpha^2 \beta b^{-1} r \lambda(B) \int_{(\alpha+1)B} \Lip[u] \dd\lambda.
\end{align*}
   This yields the desired claim.
\end{proof}

\subsection{Poincar\'e inequality for filiform-type sub-Riemannian metrics}
Now, we specialize Theorem \ref{thm:lanconelli} to the specific continuous CC-structures that we constructed in \eqref{eq:metricdef}. 
\begin{theorem}\label{thm: Poincare} If $f\in C^0(\R)$ is a function that has doubling and reverse doubling oscillation, then $(\R^d,\bd_f,\lambda)$ satisfies a Poincar\'e inequality in the sense of \eqref{eq:PIineq}.
\end{theorem}

\begin{proof} 
By Theorem~\ref{thm:lanconelli} it suffices to construct for every ball $B(\bfp,r)$ and sets $Q,\Omega$, and an almost exponential map $E$ that satisfies the assumptions. Note here that the triangle inequality then implies that the left-hand side in \eqref{eq:PIineq} is bounded from above by $\lambda(B)^{-2}$  times the left-hand side of \eqref{eq:weakPI}. The Poincar\'e inequality follows from this and doubling in Lemma \ref{lem:doublingpart}.

Assume that $f$ is a function that has doubling oscillation with constant $D$ and reverse doubling oscillation with constants $(\epsilon_f,\delta_f)\in (0,1)^2$. Fix $\delta \in (0,(4d)^{-1})$, with its precise value determined in the course of the proof. Let $r>0,\bfp\in \R^d$, $t_1,\dots, t_{d-1}\in [-r/(4d),r/(4d)]$ be such that $\{f(p_1+t_i):i=1,\dots, d-1\}$ is an equipartition of $f([p_1-r/(4d),p_1+r/(4d)]).$ For $\bfq=(q_1,\bfq_{-1})\in \R^d$ write $\bfV_\bfq:={\bf Van}(f(q_1+t_1), \dots, f(q_1+t_{d-1}))$ and $\Omega := B(\bfp,\delta r)$, and $m_\bfq=\min_{i\neq j} |f(q_1+t_{i})-f(q_1+t_{j})|$. Let $\omega=\osc(f,[p_1-r/(4d),p_1+r/(4d)])$.

We have $m_{\bfp}=\frac{\omega}{d-2}$. Choose $k_0\in \N$ such that  $D\epsilon_f^{k_0}\leq \frac{1}{4(d-2)}$ and assume that $\delta \leq \delta_f^{k_0}/(4d)$. For every $\bfq\in \Omega$ we have $|q_1-p_1|\leq \delta r$, and thus 
\begin{align*}
|f(p_1+t_i)-f(q_1+t_i)|&\leq \osc(f, [p_1+t_i-\delta r, p_1+t_i + \delta r]) \\
&\leq \epsilon_f^{k_0}\osc(f, [p_1+t_i-r/(4d), p_1+t_i + r/(4d)]) \\
&\leq \epsilon_f^{k_0}\osc(f, [p_1-r/(2d), p_1+ r/(2d)]) \\
&\leq D\epsilon_f^{k_0}\omega \leq \frac{\omega}{4(d-2)}.
\end{align*}
Using this, we get $m_{\bfq} \geq \frac{\omega}{4(d-2)}$. Also, let $M_{\bfq}=\max_{i}|f(q_1+t_i)-f(q_1)|\leq \osc(f,[q_1-r/(4d), q_1+r/(4d)])\leq \osc(f,[p_1-r/(2d), p_1+r/(2d)]) \leq D \omega$.

Define for $\bfh=(h_1,\bfh_{-1})\in \R^d$
\[
E(\bfq,\bfh):=\bfq + (h_1, \bfV_{\bfq}\bfh_{-1}).
\]
Notice that $E$ is linear in $\bfh$, and thus $D(x,\cdot)=|\det(D_{\bfh}E(x,\cdot))|$ is constant. We thus obtain the conditions $(1)--(2)$ of an almost exponential map with $a=1$. Next, we verify controllability, and the remaining assumptions in Theorem \ref{thm:lanconelli}.

By the Ball-Box Theorem, Theorem~\ref{thm: ballbox}, with $\delta_2=\osc(f,[q_1-4r,q_1+4r])$ we have $\delta_2 \leq \osc(f,[p_1-5r,p_1+5r]) \leq D^{\log_2(d)+6}\omega$
\begin{equation}\label{eq:ballboxq}
B(\bfp, r)\subset B(\bfq,2r)\subset \bfT_{\bfq}(\BBox(D^{\log_2(d)+6}\omega, 4dr)).
\end{equation}
Let $E_\bfq(\bfh)=E(\bfq, \bfh)$. We have $E_\bfq^{-1}\bfT_{\bfq}(h_1, \bfh_{-1})=(h_1, \bfV_{\bfq}^{-1}\bfL^{-1}_{f(q_1)}\bfh_{-1})$. By Lemma~\ref{lem: approxinv}, 
\[
\bfV_{\bfq}^{-1}\bfL^{-1}_{f(q_1)}=\bfM_{\bfq}^{-1}\bfD_{m_\bfq}^{-1}
\]
for a matrix $\bfM_{\bfq}$ with 
\[
\|\bfM_\bfq^{-1}\|\leq (d-1)^{d-2}\max\left\{1,\frac{M_\bfq}{m_\bfq}\right\}^{(d-2)^2} \leq (d-1)^{d-2} (4(d-2)D)^{(d-2)^2}.
\]

\begin{align}
    E_\bfq^{-1}\bfT_{\bfq}(\BBox(D^{\log_2(d)+6}\omega, 4dr)) &\subset [-4dr,4dr]\times \bfM_{\bfq}^{-1}\bfD_{m_\bfq}^{-1} \bfD_{D^{\log_2(d)+6}\omega}\left([- 4dr, 4dr]^{d-1}\right) \nonumber\\ 
    &\subset [-4dr,4dr]\times \bfM_{\bfq}^{-1} \bfD_{D^{\log_2(d)+6}\omega m_q^{-1}} \left([- 4dr, 4dr]^{d-1}\right) \nonumber\\ 
    &\subset [-4dr,4dr]\times \bfM_{\bfq}^{-1} \bfD_{4(d-2)D^{\log_2(d)+6}} \left([- 4dr, 4dr]^{d-1}\right), \nonumber \\
    &\subset [-4dr,4dr]\times [- C_{D,d}r, C_{D,d}r]^{d-1}, \label{eq:inclusionballs}
\end{align}
where on the last line we used 
\[
\|\bfM_{\bfq}^{-1} \bfD_{4(d-2)D^{\log_2(d)+6}}\|_{\ell^\infty} \leq d\|\bfM_{\bfq}^{-1} \bfD_{4(d-2)D^{\log_2(d)+6}}\|\leq C_{D,d}/(4d)
\]
where $C_{D,d}$ is a sufficiently large constant and $\|\cdot\|_{\ell_\infty}$ is the operator norm with respect to the $\ell_\infty$-metric.

Let
\[
Q:=[-4dr,4dr]\times [- C_{D,d}r, C_{D,d}r]^{d-1}.
\]
From \eqref{eq:inclusionballs} and \eqref{eq:ballboxq}, we get

\[
B(\bfp,r)\subset \bfT_{\bfq}(\BBox(D^{\log_2(d)+6}\omega, 4dr)) \subset E(\bfq,Q); \qquad \forall \bfq\in \Omega.
\]
%for all $\bfx\in \Omega$. 
Next, after fixing $\bfq\in \Omega$, we consider the same points we did in \eqref{some points} with $y_i:=h_{i+1}$ and $\bfp_1:=\bfq$. However, we now parametrize the curves joining these points differently from the ${\bgamma}^i$ in \eqref{some curves}.
%Consider now $\bfx\in \Omega$ and consider the sequence of points 
%$\bfp^1, \ldots, \bfp^{d+1}$ defined as
% \[
%\begin{cases} 
%\bfp^1:=\bfx, &\\
%\bfp^{i+1}:=e^{-t_i \bfX}e^{h_{i+1} \bfY}e^{t_i\bfX}\bfp^i
%,&\forall i\in\{1, \dots, d-1\}\\
%\bfp^{d+1}:=e^{t_{d}\bfX} \bfp^{d}.
%\end{cases}
%\]
%\[\bfp^1=\bfx, \bfp^{i+1}=e^{-t_i \bfX}e^{h_{i+1} \bfY}e^{t_i\bfX}\bfp^i\]
%for $i=1, \dots, d-1$ and let $\bfp^{d+1}=e^{t_{d}\bfX} \bfp^{d}$. 
%These points are connected by a horizontal path defined as follows. 
To simplify notation, let $R:=\max(C_{D,d}, 4d)r$. For $i\in\{1,\dots, d-1\}$ set
\[
\bgamma_{\bfq,\bfh}(t):=\begin{cases} e^{\frac{t-3(i-1)R}{R} \cdot t_i\bfX}\bfp^i & \text{for } t\in [3(i-1)R, 3(i-1)R+R] \\
e^{\frac{t-3(i-1)R-R}{R} \cdot h_{i+1} \bfY}e^{t_i\bfX}\bfp^i & \text{for } t\in [3(i-1)R+R, 3(i-1)R+2R]\\
e^{-\frac{t-3(i-1)R-2R}{R} \cdot t_i\bfX}e^{h_{i+1} \bfY}e^{t_i\bfX}\bfp^i & \text{for } t\in [3(i-1)R+2R, 3iR].\end{cases}
\]
Also, define for $t\in [3(d-1)R,3dR]$
\[
\bgamma_{\bfq,\bfh}(t):=e^{\frac{t-3(d-1)R}{3R} h_{1}\bfX} \bfp^{d}.
\]
Set $T:=3dR$. We have $E(\bfq, \bfh)=\bgamma(\bfq,\bfh,T)$ and $\bgamma(\bfq,\bfh,0)=\bfq$. By a direct calculation we also have $\bgamma_{\bfq,\bfh}(t)'=a(t) \bX + b(t)\bfY$ with $\sqrt{a^2+b^2}\leq 1$ since $\bgamma_{\bfq,\bfh}(t)$ is piecewise a constant multiple of $\bfX$ or $\bfY$. The control function is now given by
\[
\gamma(\bfq,\bfh,t)=\gamma_{\bfq,\bfh}(t).
\]
Fix $t\in [0, T]$ and $\bfh\in Q$. A simple proof by induction on $i$ in the construction of $\bgamma_{\bfq, \bfh}$ shows that for $\bfq=(q_1,\bfq_{-1})$ we have
\[
\gamma_{\bfq,\bfh}(t)=(q_1+c_{t,\bfh}, \bfq_{-1}+H_{\bfh}(q_1,t))
\]
for some continuous $c_{t,\bfh}$ and $H_{\bfh}(q_1,t)$ which depends continuously on $q_1$. For every fixed $q_1$, this map is a translation that shifts the $x_1=q_1$ slice to $x_1=q_1+c_t$. This translation preserves the volume of the cross-section, and by Fubini, the map $\bfq\mapsto \gamma_{\bfq,\bfh}(t)$ is volume-preserving. Thus, we can take $b=1$ in the definition of controllability. From this formula, it is also direct to see that this map is injective.

%Since translation by the basis vectors $\bfe_j$ of $\R^d$ for $j\in\{2,\dots, d-1\}$ preserve the flows of the vector fields $\bfX,\bfY$, we get for all $s\in \R$
%\[
%\gamma(\bfq+s \bfe_j,\bfh,t)=\gamma_{\bfq,\bfh}(t)+s \bfe_j,
%\]
%and we get $\partial_{x_j}\gamma(\bfq,\bfh,t)=\bfe_j$ for $j\in\{2,\dots, d\}$. On the other hand, translating $\bfq$ by $s\bfe_1$ does not alter the flows of the vector field $\bfX$. Instead, it only changes the flows of $\bfY$ - which only affect the co-ordinates $\gamma(\bfq,\bfh,t)_j$ of the curve for $j=2,\dots, d-1$. Thus, we get
%\[
%\partial_{x_1}\bgamma(q,h,t)_1=1.
%\]
%In particular, by Lemma~\ref{lem: continuousvolpreserving}, the map $\bfq\mapsto \bgamma(\bfq,\bfh,t)$ is volume preserving and injective.

Only the conditions (1) and (3) from Theorem~\ref{thm:lanconelli} remain to be checked. Both conditions follow since Lemma~\ref{lem:doublingpart} gives that $\lambda$ is doubling with respect to the metric $\bd_f$. 
\end{proof}

%In the proof we needed the volume preservation of certain simple maps, which we verify next.
%\begin{lemma}\label{lem: continuousvolpreserving} Assume that $\Phi:\Omega\to\Omega'$ is a continuous bijection between two open sets $\Omega,\Omega'\subset \R^d$, $\Omega$ is connected, and that $\Phi$ has continuous partial derivatives in the directions of the variables $x_2,\dots, x_{d-1}$ and its first component $\Phi_1$ is continuously differentiable. If $\partial_{x_j} \Phi =\bfe_j+\sum_{k>j} a_{jk}\bfe_k$ for $j=2,\dots, d$ and $\nabla \Phi_1 = \bfe_1$, then $\Phi$ is volume preserving, i.e., $\Phi_*(\lambda|_{\Omega})=\lambda|_{\Omega'}$. 
%\end{lemma}
%\begin{proof} 
%Denote the Lebesgue measure on $\R^{d-1}$ by $\lambda_{d-1}$. Let $A\subset \Omega$ be measurable. Then, by Fubini,
%\[
%\lambda(A)=\int_{-\infty}^\infty \lambda_{d-1}(A_t) \dd t,
%\]
%where $A_t:=\{(x_2,\dots, x_d): (t,x_2\dots, x_d)\in A\}$.

%We have $\nabla \Phi_1=\bfe_1$. Since $\Omega$ is connected, there exists an $a$ so that $\Phi(\bfp)=p_1-a$ for all $\bfp\in \Omega$. Thus $\Phi(A_t)=(\Phi(A))_{t-a}$ for any $t\in \R$.

%We also have $\lambda_{d-1}(A_t)=\lambda_{d-1}(\Phi(A_t))=\lambda_{d-1}(\Phi(A)_{t-a})$, since for a fixed $t$ the mapping $A_t\mapsto \Phi(A_t)$ is a continuously differentiable map with Jacobian determinant unity. Then, Fubini implies
%\begin{align*}
%\lambda(\Phi(A))& =\int_{-\infty}^\infty \lambda_{d-1}(\Phi(A)_{t-a}) \dd t \\
%&=\int_{-\infty}^\infty \lambda_{d-1}(A_t) \dd t =\lambda(A).
%\end{align*}
%Thus, the map $\Phi$ is volume preserving.
%\end{proof}

\section{Proofs of main theorems}\label{sec main proofs}

We now prove the main theorems. For the analytic dimension claims, we refer to Theorem~\ref{thm: analyticdim} from the next section.

\begin{proof}[Proof of Theorem~\ref{thm: oscillationversion}]
The completeness, homeomorphism with $\R^d$, and geodesic properties follow from Proposition~\ref{prop:boundedcompact}. The doubling properties, volume estimates, and LLC follow from Lemma~\ref{lem:doublingpart}. The Poincar\'e inequality follows from Theorem~\ref{thm: Poincare}. 
\end{proof}

\begin{proof}[Proof of Theorem~\ref{thm: mainthm}] 
%For $Q=d$, one can use the Euclidean metric and there is nothing to check. Thus focus on $Q>d$. 
 For $Q=\frac{d(d-1)+2}{2}$, one can use $f(x)=x$ and the space $(\R^d, \bd_f, \lambda, \bfz)$ is the filiform Carnot group of the first type, which satisfies $(1)-(5)$ by classical results; see e.g. \cite{LeDonneGTM}. The claims on Ahlfors regularity and LLC, analytic dimension and the Poincar\'e inequality follow also from Lemma~\ref{lem:doublingpart},  Theorem~\ref{thm: analyticdim}, and Theorem~\ref{thm: Poincare} respectively. Carnot groups are homogeneous and admit scalings, and thus $\mathcal{F}=\{(\R^d, \bd_f, \lambda/\lambda(B(\bfz,1)), \bfz)\}$ is invariant under scalings and changes of basepoints. This class is clearly compact with respect to pointed measured Gromov-Hausdorff convergence. 

For $Q\in (d,\frac{d(d-1)+2}{2})$ let $\alpha :=\frac{2(Q-d)}{(d-2)(d-1)} \in (0,1)$. Let $\mathcal{F}^{\rm reg}_{C, \alpha}$ be the non-empty collection of $\alpha$-regular functions from Lemma~\ref{lem: oscillationfunctions}. In this case, the class of spaces is $\mathcal{F}=\cM^{\rm reg}_{C,\alpha}$ as constructed in Section~\ref{sec: GH}. This is closed under pointed measured Gromov-Hausdorff convergence by Theorem~\ref{thm: closed}. The claims on Ahlfors regularity and LLC, analytic dimension and the Poincar\'e inequality follow again from Lemma~\ref{lem:doublingpart},  Theorem~\ref{thm: analyticdim}, and Theorem~\ref{thm: Poincare} respectively. The completeness, homeomorphism with $\R^d$ and geodesic properties follow from Proposition~\ref{prop:boundedcompact}.
\end{proof}

\begin{proof}[Proof of Theorem~\ref{thm: mainthm2}] Assume $Q>\frac{d(d-1)+2}{2}$, one sets $\alpha :=\frac{2(Q-d)}{(d-2)(d-1)} >1$ and uses Lemma~\ref{lem: oscillationfunctions} to give a continuous monotone function $f$ of doubling oscillation for which 
\[
\lim_{r\to 0} \frac{\log(\osc(f,[p_1-r,p_1+r]))}{\log(r)}=\alpha, 
\]
for a.e. $p_1\in \R$. Any monotone function $f$ of doubling oscillation also has reverse doubling oscillation. This is seen as follows. Let $D>1$ be the constant of doubling oscillation. If $I\subset J$ is an interval of length at most half of $J$, then the closure of $J\setminus I$ contains an interval $K$ of length at least $1/4$ of the length of $J$. Since $f$ has doubling oscillation, we have $\osc(f,K)\geq D^{-3}\osc(f,J)$. Since $f$ is monotone, we have
$\osc(f,I)\leq \osc(f,J)-\osc(f,K) \leq (1-D^{-3})\osc(f,J)$. Thus, we get the reverse doubling oscillation property with $\epsilon=1-D^{-3}$and $\delta=1/2$.

Now, the measure doubling and LLC properties of $(\R^d, \bd_f, \lambda, \bfz)$ follows from Lemma~\ref{lem:doublingpart}, the Poincar\'e inequality from Theorem~\ref{thm: Poincare} and the analytic dimension from Theorem~\ref{thm: analyticdim}. The completeness, homeomorphism with $\R^d$, and geodesic properties follow from Proposition~\ref{prop:boundedcompact}. The fact that every blow-up of $(\R^d,\bd_f)$ is homeomorphic to Euclidean space follows from Theorem~\ref{thm: closed}, since $(\R^d, \bd_f, \lambda/\lambda(B_{\bd_f}(\bfz,1)), \bfz)\in \mathcal{M}^{\rm doubl}_{D,\epsilon,\delta}$, and the latter is closed under scalings, changes of basepoints, and Gromov-Hausdorff limits.

Finally, we show the claim on essential dimension. From the volume estimates in Lemma~\ref{lem:doublingpart} we get for Lebesgue a.e. $\bfp\in \R^d$ that
\[
\lim_{r\to 0} \frac{\log(\mu(B(\bfp, r)))}{\log(r)}=Q.
\]
That the essential dimension of $\mu$ is $Q$, and that every set of positive measure has Hausdorff dimension at least $Q$, now follow from this and Lemma~\ref{lem: billinglsey}. 

\end{proof}

\section{Analytic dimension}\label{sec: analytic}

We briefly discuss the analytic dimension of continuous Carnot-Carathéodory structures. In this section, we will always assume that measures $\mu$ are Borel measures that  are positive and finite on balls.

\begin{definition}\label{def:diffspace} A separable metric measure space $(X,\bd,\mu)$ (with $\bd(x,y)<\infty$ for every $x,y\in X$) is said to be a {\em Lipschitz differentiability space}, if there exist countably many measurable positive-measure sets $U_i\subseteq X$, $i\in \N$, with $\mu(X\setminus \bigcup_{i\in \N} U_i)=0$, and Lipschitz functions $\phi_i:X\to \R^{n_i}$, for some $n_i\in \N$, for which for every Lipschitz function $f:X\to \R$ and every $i\in \N$ for a.e. $\bfx\in U_i$ there exists a unique linear map $\dd f_{i,\bfx}:\R^{n_i}\to \R$ for which
\begin{equation}\label{eq:diff}
\Lip[f(\cdot)-\dd f_{i,\bfx}\circ \phi_i(\cdot)](\bfx)=0.
\end{equation}
\end{definition}
 Each pair $(U_i,\phi_i)$ is called a {\em Lipschitz differentiability chart}. We call $N:=\sup_{i\in \N} n_i$ the {\em analytic dimension} of $X$. Here we omit the assumption that $N<\infty$, which is assumed for example in \cite{KeithLD}. For more background on such spaces, see e.g. \cite{Ch, Bate, BateLi, SchioppaA,sebchar}. 

In \cite{Ch}, Cheeger showed that every PI space is a Lipschitz differentiability space, and that the analytic dimension can be bounded by the parameters in the Poincar\'e inequality and the doubling constants. Later, the results in \cite{Schioppa, PMR} improved the upper bound for analytic dimension to the Hausdorff and Assouad dimensions; see also \cite{teriseb}. In Theorem~\ref{thm: analyticdim} we will determine the analytic dimension of every continuous Carnot-Carath\'eodory structure that gives a differentiability space. The proof of the result relies on upper and lower bounds that come by slightly different means.

The lower bound is related to a notion of independence. Given a Lipschitz map $F=(f_1,\dots, f_M):X \to \R^M$ let ${\rm Ind}(F)$ be the collection of points $p$ such that 
     \begin{equation}\label{d:indep}
    \inf_{\stackrel{\mathbf{a}=(a_1,\dots,a_M)\in \R^M}{|\mathbf{a}|=1}} \Lip\left[\sum_{i=1}^M a_i f_{i}\right](p)>0 \ \ \ \text{holds.}
    \end{equation}
    We say that $F$ is {\em independent} on a set $\Omega$ if \eqref{d:indep} holds almost everywhere in $\Omega$.
    
     We state a well-known equivalence between differentiability and independence, which also gives a way of finding lower bounds for analytic dimension by finding independent tuples of Lipschitz functions.  

    \begin{lemma}\label{lem:diffind}
A separable metric measure space  $(X,d,\mu)$ is a differentiability space with analytic dimension at most $N$ if and only if for every Lipschitz $F: X\to \R^M$ such that  $\mu({\rm Ind}(F))>0$ it holds $M\leq N$. 
    \end{lemma}
   \begin{proof}
   The condition is sufficient for differentiability by \cite[Proposition~4.1]{KleinerMackay}, where \cite[Definition~2.3]{KleinerMackay} is equivalent to our condition on independence. Next, we show the necessity. Assume that $(X,\bd,\mu)$ is a differentiability space of analytic dimension $N$, and $F:X\to \R^M$ is such that  $\mu({\rm Ind}(F))>0$. It suffices to prove $M\leq N$. Assume not, i.e., $M>N$. Choose a chart $(U_i,\phi_i)$  with $\phi_i:X\to \R^{n_i}$, and $U_i\cap {\rm Ind}(F)$ having positive measure with $n_i\leq N<M$. Consider the differentials $df_{k,i}$ for $k=1,\dots, M$. Since $n_i<M$, for a.e. $p\in {\rm Ind}(F) \cap U_i$ there is some linear combination $\mathbf{a}=(a_1,\dots, a_M)\in \R^M$ with $|\mathbf{a}|=1$ of these differentials that vanishes, i.e., $\sum_{k=1}^M a_k df_{k,i,p} = 0$. At these points
    \[
    \Lip\left[\sum_{k=1}^M a_k f_{k}\right](p) = \Lip\left[\sum_{k=1}^M a_k (f_{k}-df_{k,i,p}\circ \phi_i)\right](p)\stackrel{\eqref{eq:diff}}{=}0.
    \]
    This contradicts \eqref{d:indep} and thus we must have $M\leq N$.
    \end{proof}

The previous lemma gives a lower bound for analytic dimension by the dimensionality of any independent map. We also need a way of finding an upper bound. This involves the notion of curve fragments and so-called BLD maps. 

A {\em curve fragment} in $X$ is a biLipschitz embedding $\gamma: K\to X$ for some compact non-empty set $K\subset \R$. In our case, we will be interested in a geodesic space $X$, where each curve fragment $\gamma$ can be enlarged to a curve in the following way. Write the complement of $K$ as a union of disjoint open intervals: $[\min(K),\max(K)]\setminus K = \bigcup_{i\in I} (a_i,b_i)$ for some finite or countable index set $I$. For each $i\in I$, let $\gamma_i:[a_i,b_i]\to X$ be a geodesic curve connecting $\gamma(a_i)$ to $\gamma(b_i)$. The Lipschitz constant of $\gamma_i$ is easily seen to be bounded by $\frac{d(\gamma(a_i),\gamma(b_i))}{|a_i-b_i|}$, which in turn is bounded by the Lipschitz constant of $\gamma$. Define the {\em geodesic interpolation} $\sigma:[\min(K),\max(K)]\to X$ as $\sigma(t)=\gamma(t)$ for $t\in K$ and $\sigma(t)=\gamma_i(t)$ for $t\in (a_i,b_i)$. By the bound on Lipschitz constants, it is direct to see that $\sigma$ is a Lipschitz curve.

The {\em metric derivative} of a curve fragment at a Lebesgue density point $t$ of $K$ is given by
\[
|\dot\gamma|(t):=\limsup_{s\in K \to t}\frac{d(\gamma(t),\gamma(s))}{|t-s|}.
\]
Similarly, if $F: X\to \R^n$ is Lipschitz, for each Lebesgue density point $t$ of $K$ we can define (whenever the limit exists) the derivative
        \[
        (F\circ \gamma)'(t):=\lim_{s\in K \to t}\frac{F( \gamma(s))-F( \gamma(t))}{s-t}.
        \]
If $\gamma$ is a curve, i.e., $K$ is an interval, the above limits exist almost everywhere, and for curve fragments the existence can be obtained by geodesically interpolating in $X$, and then noting that a subsequential limit equals a full limit.
Let $C\geq 1$. A \emph{weakly $C$-bounded length distortion (BLD)} map $F:X\to \R^n$ is a $C$-Lipschitz map such that  
\[
|\dot\gamma|(t)\leq C|(F\circ \gamma)'(t)|
\]
holds for all Lipschitz curves  $\gamma:[a,b]\to X$ and a.e. $t\in [a,b]$. The terminology is motivated by \cite{MV}. 

\begin{lemma}\label{lem:upperbound} Assume that $(X,d,\mu)$ is a geodesic Lipschitz differentiability space and that $\Omega\subset X$ is a non-empty open set. If $F=(f_1,\dots, f_n):X \to \R^n$ is a Lipschitz map such that $F|_\Omega$ is a weakly $C$-BLD map and $\mu(\Omega\cap {\rm Ind}(F))>0$, then for any chart $(U_i,\phi_i)$ with $\phi_i:X\to \R^{n_i}$ of $X$ with $\mu(U_i\cap {\rm Ind}(F) \cap\Omega)>0$, we have $n_i= n$.
\end{lemma}
\begin{proof}
    Let $M\geq n$ be the maximal integer such that  we can add $1$-Lipschitz coordinates to $F$ in such a way that $\widetilde{F}=(f_1,\dots, f_M)$ is Lipschitz and $\mu(\Omega\cap {\rm Ind}(\widetilde{F})\cap U_i)$ has positive measure. By Lemma~\ref{lem:diffind} applied to the restricted metric measure space $(U_i, d|_{U_i}, \mu|_{U_i})$, or just \cite[Lemma 4.2]{KleinerMackay}, there is an upper bound for $M\leq n_i$ and such a maximal $M$ exists. Here, one also uses a well-known fact that Lipschitz differentiability implies that porous sets have measure zero \cite[Theorem 2.4]{BS}. This in turn implies that for any Lipschitz function $f:X\to\R$ and any measurable subset $U\subset X$, we have $\Lip[f|_U](x)=\Lip[f](x)$ for $\mu$-a.e. $x\in U$, and thus independence is preserved when passing to subsets.

    By the proof of \cite[Lemma 4.2]{KleinerMackay}, it follows that $({\rm Ind}(\widetilde{F})\cap \Omega\cap U_i, \widetilde{F})$ is a chart. We must have $n_i=M$, since charts are independent (since otherwise the differential is not unique) and we can apply the argument from the previous lemma. Thus, it suffices to prove that $M=n$. 
    
    Assume for the sake of contradiction that $M>n$. Let $\mathcal{C}=\{v=(v_1,\dots, v_M) : v_M > 2C\|(v_1,\dots, v_{M-1})\|>0\}$ be an open cone of vectors.
    By density of directions in charts \cite[Theorem 9.5.]{Bate} there exists a curve fragment $\gamma:K\to {\rm Ind}(\widetilde{F})\cap \Omega\cap U_i$ with $\lambda(K)>0$, $K\subset \R$ such that  for a.e. $t\in K$
    \begin{enumerate}
        \item $t$ is a Lebesgue density point of $K$;
        \item $\widetilde{F}\circ \gamma$ is differentiable at $t$ and $(\widetilde{F}\circ \gamma)'(t)=\widetilde{w}(t)\in \mathcal{C}$
        \item $\gamma$ is metrically differentiable at $t$ and $|\dot\gamma|(t)>0$.
    \end{enumerate}
    Let $w(t)=(F\circ\gamma)'(t)$ be the vector formed by the first $n$ components of $\widetilde{w}(t)$. Let $\sigma$ be the geodesic interpolation of $\gamma$. We have that $\sigma$ is metrically differentiable and $\widetilde{F}\circ \gamma$ is differentiable, whenever $t$ is as above. Moreover, since a subsequential limit equals the full limit, $|\dot\gamma|(t)=|\dot{\sigma}|(t)$ and $(\widetilde{F}\circ \gamma)'(t)=(\widetilde{F}\circ \sigma)'(t)$ for all such $t\in K$ where both limits exist.

    We have for a.e. $t\in K$ that $|\dot\gamma|(t)=|\dot{\sigma}|(t)\leq C|(F\circ \sigma)'(t)|=C|(F\circ \gamma)'(t)|=C\|w\|$ since $F|_\Omega$ is weakly BLD. Since $\widetilde{w}\in \mathcal{C}$, we also have $2C\|w\|\leq |\widetilde{w}_{M}|$. Finally, $|\widetilde{w}_{M}|\leq |\dot\gamma|(t)$ since $f_{M}$ is $1$-Lipschitz. By combining these we get
    \[
    0<|\dot\gamma|(t)\leq C\|w\|\leq |\widetilde{w}_{M}|/2\leq  |\dot\gamma|(t)/2, 
    \]
    for a.e. $t\in K$, which is a contradiction. Thus $M=n$.
\end{proof}

\begin{theorem}\label{thm: analyticdim} Let $(\R^d, \Delta, \|\cdot\|)$ be a continuous Carnot-Carath\'eodory structure such that the space $(\R^d,d_{\Delta, \|\cdot\|})$ is boundedly compact and such that the topology generated by $d_{\Delta, \|\cdot\|}$ is the same as the Euclidean topology, and let $\mu$ be a Radon measure on $\R^d$. If $(\R^d, d_{\Delta, \|\cdot\|}, \mu)$ is a differentiability space, then 
%$N={\rm dim}(\Delta)$ is 
its analytic dimension equals the rank of $\Delta$.
\end{theorem}

\begin{proof}[Proof of Theorem~\ref{thm: analyticdim}]
Let $N$ be the rank of $\Delta$, i.e., the dimension of its fibers. 
   We may cover the space by open sets $\Omega$, such that some orthogonal projection onto an $N$-dimensional space is $L$-biLipschitz when restricted to $(\Delta_p, \|\cdot\|)$ for every $p\in \Omega$ for some $L>1$. Fix one such $\Omega$ and consider the orthogonal projection $\Pi|_{\Omega}$ onto the given subspace. Extend this by the McShane Theorem to a Lipschitz function on all of $(\R^d,d_{\Delta, \|\cdot\|})$, possibly increasing the Lipschitz constant by a constant.  To simplify notation, rotate so that the projection corresponds to the first $N$ coordinates, i.e., $\Pi=(\pi_1,\dots, \pi_N)$.
    
    First, we show that $\Pi$ is independent on $\Omega$, i.e., Equation \eqref{d:indep} holds almost everywhere with $F:=\Pi$. Up to biLipschitz equivalence, we may assume that the norm $\|\cdot\|$ comes from a scalar product, and so we choose a linearly independent set of continuous unit length vector fields $\mathbf{X}_1,\dots, \bfX_N$ in $\Omega$ that span $\Delta_p$, for each $p\in \Omega$. For each $k\in \{1,\dots, N\}$ and each $p\in \Omega$, by Peano's theorem, we can find a curve $\gamma_k:(-\epsilon,\epsilon)\to \R^d$ such that  $\gamma_k(0)=p$ and $ \dot\gamma_k(t)=\bfX_k(\gamma_k(t))$ for $t\in (-\epsilon,\epsilon)$. If $\mathbf{a}=(a_1,\dots, a_N)\in \R^N$ is a vector, then we have by differentiation along each of these curves that
    \[
     \Lip\left[\sum_{i=1}^N a_i \pi_{i}\right](p) \geq \max_{k\in\{1,\dots, N\}}\left(\left|\sum_{i=1}^N a_i \bfX_{k,i}\right|\right).
    \]
    Here $\bfX_{k, i}$ denote the $i$-th components of $\bfX_k$ at the point $p$. Let $\bfX$ be the full rank matrix with entries $(\bfX_{k,i})_{k,i=1}^N$ evaluated at $p$ for $k\in\{1,\dots, N\}$. This matrix is full rank in $\Omega$ by our choice of the projection $\Pi$. Then, we get
    \[
    \Lip\left[\sum_{i=1}^N a_i \pi_{i}\right](p) \geq \frac{1}{N} \|\bfX \mathbf{a}\|,
    \]
    which yields immediately \eqref{d:indep}. This yields a lower bound $N$ for the analytic dimension by Lemma~\ref{lem:diffind}.
    
     Next, we proceed to show the upper bound. For this, it suffices to prove that $\Pi$ is weakly $L$-BLD on $\Omega$, since then the claim follows from Lemma~\ref{lem:upperbound} by covering the space $\R^d$ by a countable collection of open sets as in the beginning of the proof. The $L$-BLD property follows from Theorem~\ref{d_X rectifiability} and its proof as follows. First, if $\gamma:[a,b]\to (\Omega, d_{\Delta, \|\cdot\|}),$ is a Lipschitz curve, then Theorem~\ref{d_X rectifiability} shows that $\gamma$ is $\Delta$-horizontal and the length functionals agree, that is for all $t\in [a,b]$ we have
     \[
     \mathbf{L}_{\|\cdot \|}(\gamma|_{[a,t]})=\mathbf{L}_{d_{\Delta, \|\cdot\|}}(\gamma|_{[a,t]}).
     \]
     The derivative of the length functional gives the metric derivative for a.e. $t\in [a,b]$ by \cite[Theorem~2.7.6]{bbi}. Thus, differentiating these with respect to $t$, we get
     \[
     |\dot \gamma|(t)=\|\dot\gamma(t)\| \ \  \text{ for a.e.} \ \ t\in [a,b].
     \]
     By the assumption that $\Pi$ is $L$-biLipschitz when restricted to $\Delta_p$ for $p\in \Omega$, we get $\|\dot\gamma(t)\|\leq L |(\Pi\circ \gamma)'(t)|$ and thus $\Pi$ is $L$-BLD.
\end{proof}

%One immediately obtains Theorem~\ref{thm: analyticdim} as a corollary of the previous theorem and Theorem~\ref{thm: Poincare}.

\bibliographystyle{acm}
\bibliography{refs}
\def\cprime{$'$}

\end{document}